\documentclass[a4paper,12pt]{article}
\usepackage[utf8]{inputenc}

\usepackage[english]{babel}
\usepackage{amsmath}
\usepackage{amsthm}
\usepackage{amssymb}
\usepackage{comment}
\usepackage{soul, color}
\usepackage{appendix}
\usepackage{wrapfig}

\usepackage[ruled,vlined]{algorithm2e}
\usepackage{subcaption,graphicx}

\usepackage{graphicx}
\graphicspath{{pictures/}}
\DeclareGraphicsExtensions{.pdf,.png,.jpg,.gif}

\newcommand{\te}{\theta}
\newcommand{\pd}{\partial}

\newcommand{\N}{\ensuremath{\mathbb{N}}}
\renewcommand{\S}{\ensuremath{\mathbb{S}}}

\newcommand{\Z}{\ensuremath{\mathbb{Z}}}
\newcommand{\R}{\ensuremath{\mathbb{R}}}
\newcommand{\C}{\ensuremath{\mathbb{C}}}
\newcommand{\T}{\ensuremath{\mathbb{T}}}
\newcommand{\pP}{\mathcal{P}}
\newcommand{\G}{\mathcal{G}}

\newcommand{\X}{\ensuremath{\mathcal{X}}}
\newcommand{\w}{\ensuremath{\omega}}
\newcommand{\et}{\eta}
\newcommand{\trans}{\mathrm{T}}
\newcommand{\al}{\alpha} 
\newcommand{\tr}{\mathrm{tr}}

\newcommand{\bb}[1]{\mathbf{#1}}
\newcommand{\lr}[1]{\langle{#1}\rangle}
\newcommand{\lrf}[1]{\lfloor{#1}\rfloor}
\newcommand{\nofty}[1]{\|{#1}\|_{\infty}}

\newcommand{\im}{\mathrm{i}}
\newcommand{\dx}{\mathrm{d}}
\newcommand{\e}{\mathrm{e}}

\newcommand{\ds}{\displaystyle}

\newcommand{\M}{\mathcal{M}}
\newcommand{\HH}{\mathcal{H}}

\newcommand{\ar}{\mathrm{arccos}}
\newcommand{\supp}{\mathrm{supp}}
\newcommand{\St}{\ensuremath{\mathbb{S}}^2}

\theoremstyle{plain}
\newtheorem{thm}{Theorem}[subsection]
\newtheorem{lem}{Lemma}[subsection]
\newtheorem{cor}{Corollary}[subsection]

\theoremstyle{definition}

\theoremstyle{remark}

\numberwithin{equation}{section}
\title{\Large \bf Super-Resolution on the Two-Dimensional Unit Sphere}

\author{ \small  Frank Filbir, Kristof  Schr\"oder, Anna Veselovska}
\date{\footnotesize{* \color{red}
This is an extended version of our  submitted manuscript for publication titled as ''Recovery of Atomic Measures on the Unit Sphere''. 
The present manuscript will not be submitted for publication, but it is rather meant to provide  details missing in the submitted version.  
} }

\begin{document}

\maketitle

\begin{abstract}
We study the problem of recovering an atomic measure on the unit 2-sphere $\S^2$ given finitely many moments with respect to spherical harmonics.
The analysis relies on the formulation of this problem as an optimization problem on the space of bounded Borel measures on $\S^2$ as it was considered by Y. de Castro \& F. Gamboa and E. Cand\'es \& C. Fernandez-Granda in 2013. We construct a dual certificate using a kernel given in an explicit form and make a concrete analysis of the interpolation problem. Numerical examples are provided and analyzed.
    
\end{abstract}

\newpage

\tableofcontents

\section{Introduction}

 Signals on a spherical manifold arise  in various applications, staring from medical imaging \cite{Arridge1999, Taguchi2001, Johansen2013},
  computer graphics \cite{Ramamoorthi2004, Sloan2008} and  sound recording \cite{Meyer2001, Meyer2003} to  astrophysics \cite{Jarosik2011} and  topography \cite{Audet2011}. 
  As mentioned in  \cite{McEwen2013}, in many of these settings inverse problems arise, where an unknown signal need to be recovered from linear measurements acquired through a convolution process. 
 In this work,  we consider the case, when a spatially highly resolved
signal is modelled as a weighted sum of Dirac measures $\mu^\star$ on the two-dimensional Euclidean $\St$, 
 and the information one can access is  only the convolved version of the signal $\mu^\star\ast D_N $ with the Dirichlet kernel $D_N$ on the sphere, for possibly low $N$. The problem of recovery of  a spatially highly resolved
signal from its coarse scale information is called the \textit{super-resolution problem} or, in other words, the \textit{ de-convolution  problem}. 
 
 In general, this sort  of problem
has been treated in different geometric settings and with respect to different
systems of functions. Exact measure recovery in the classical Fourier setting, i.e. when the underlying space is the torus $\T^d$, has a very long history starting with the initial work by G. R. de Prony in 1795 \cite{Prony} in the univariate case, and then moving on to different one-dimensional  \cite{Roy1989, Schmidt1989, JurgenFrank2011}  and  multi-dimensional Prony-based techniques  \cite{Potts2013, Kunis2016, Cuyt2018, Cuyt2020, Kunis2020} that have stabilized and generalized the method in various directions. Recently, the Prony's method has also been extend to the $d$-dimensional sphere $\mathbb{S}^d$.   

Less than ten years ago, considering the one-dimensional de-convolution problem in the light of convex optimization,  several authors \cite{DeCastro2012,Candes2014} have proposed a variational recovery approach  that is to minimize the \textit{total variation} over the set of all finite complex  measures supported on $\T$, given the convolved version of the measure. It has been
shown that if the support atoms of the measure   are well separated then the target  measure $\mu^\star$
is the unique solution of the minimization problem, therewith a sufficient criteria for $\mu^\star$ being the unique solution is the existence of so-called \textit{dual certificate}, that is, a polynomial of degree $N$ whose sup-norm is reached at the points of the measure support.

These two fundamental papers have ushered in new ways of treating the super-resolution problem.
Beside the fact that this variational recovery method does not need the number of unknowns points beforehand, it also has proved to enjoy stability in the case when the low frequency information of corrupted by noise \cite{Candes2013, Peyre2015a, Peyre2015b}. 
Hereupon, the generalization to higher dimensions on the torus has been considered in 
\cite{Fernandez2016, Peyre2019, Lasserre2019}. 
Another big advantage of considering the super-resolution as the minimization problem    is its adaptivity to different geometric settings, namely, semi-algebraic domains in higher dimensions \cite{DeCastro2016}, or compact smooth Riemannian manifolds such es the rotation group $SO(3)$ \cite{FrankKrisrof2016} and two-dimensional sphere \cite{Bendory2015a, Bendory2015b}.
Lately, this procedure has been generalized to
short-time Fourier measurements
\cite{Boelcskei2018}.

In this work, we consider  de-convolution  problem on the two-dimensional Euclidean $\St$ as a total variation minimization problem.
Although to prove uniqueness of an optimal solution, we follow a general idea from \cite{Candes2014, FrankKrisrof2016}.  
 the actual construction of a dual certificate requires localization
estimates for interpolation kernels and its derivatives on $\St$ with explicit constants, and heavily depends on special behaviour of the geodesic distance with respect to the boundedness of the derivatives and on the fact that there is no nowhere not vanishing vector field on the sphere, due to {\it the Hairy ball theorem}.
From numerical point of view, we consider two approaches. First, we use the dual formulation of the minimization problem solving it via a single semi-definite
program (SDP), going along the same line as  \cite{FrankKrisrof2016, Bendory2015b}.
In the second approach we discretize the primal problem beforehand, then solve the corresponding finite-dimensional optimization problem. 
To analyze the convergence of the discretization process we build on results stated in \cite{Tang2013}.
We also would like to mention the for the nonnegative total variation minimization problem   an alternative way of construction of a dual certificate that involves some algebraic techniques has been proposed in \cite{Kunis2020}, and 
in pure compressed sensing setting
a recovery of sparse signal on the two-dimensional sphere has been considered in \cite{Rauhut2011}

The outline of this paper is as follows. In Section 2,  briefly  the
necessary analytical tools on the sphere including spherical harmonics are introduced, and the
problem of super-resolution is stated.
 In Section 3 provides the  localization results for the chosen interpolating kernel that are essential for construction of a dual certificate. The actual construction of  a dual certificate as a solution of the Hermite-type interpolation problem is the content of Section 4. Finally, we finish by presenting the numerical solution and the discretization of the problem in Section 5.     




 
 
 




\section{Unit Sphere and  Super-Resolution }


In this section we briefly summarize  analytical tools on the two dimensional sphere  $\St$  and  state the super-resolution problem. 


\subsection{Analysis on the Sphere }

The unit sphere is an embedded sub-manifold of the three-dimensional Euclidean space $\R^3$  given by  
\begin{equation*}
    \St=\{x \in \R^3\colon \|x\|_2=1\},
\end{equation*}
where the $l_2$-norm is $\|x\|_2=\sqrt{x^\trans x}.$ Such 
embedding provides a very simple definition of the tangent space $ T_x\St$ at a point $x\in \St$, that namely it is given as the orthogonal complement of the linear subspace $\mathrm{span}\{x\}$, i.e.
\begin{equation*}
    T_x\St=\{y\in \R^3\colon \lr{x,y}=0\}, 
\end{equation*}
where $\lr{x,y}= x^\trans y$ is the standard inner product. 
Induced by the Riemannian metric of the ambient space $\R^3$, the Riemannian metric on the sphere $ g_{\S^2}: T_x\St \times T_x \St \to  \R $ is given for all $ x \in \St $  by 
\begin{equation*}
    g_{\S^2}(v, w): =\lr{v,w}, \quad v,w \in T_x \St. 
\end{equation*}
In this metric the geodesic distance between two points $ x,y \in \St$ is given by the great-circle distance
\begin{equation*}
    d(x,y)=\ar(\lr{x,y}), \quad x,y \in \St. 
\end{equation*}
Let us also shortly describe the differential structure on the sphere. We will use two different
local coordinates on $\St$. First, for  each starting point $x\in \St$ and a direction $v\in T_x\St$, there exists the unique geodesic $\gamma_{x,v}$ such that $\gamma_{x,v}(0)=x$ and $\gamma_{x,v}'(0)=v$, and the equation of such geodesic reads as
\begin{equation*}
    \gamma_{x,v}(t)=\cos{(\|v\|_2 t )} x+ \sin{(\|v\|_2t)}\frac{v}{\|v\|_2}.
\end{equation*}
Then  at a point $x\in \St$ the exponential map ${\exp_x{v}: \ T_x \St \to \St}$  is given by
\begin{equation}\label{exp_map_sphere}
    \exp_x{v}=\gamma_{x,v}(1)
\end{equation}

Now, let us fix an orthonormal basis $\eta_1^x, \eta_1^x \in T_x\St $ such that $\eta_2^x= \eta_1^x \times x$ and  $\eta_1^x= x \times \eta_2^x$. 
It is always possible to do so, although we can not choose a  local bases in a continuous way, as there is no continuous
nowhere not vanishing vector field on the sphere, due to {\it the Hairy ball theorem. }  This is a special property of $\St$ is in contrast to some other manifolds, e.g. the rotation group $SO(3)$, where the tangent space is basically a translation of the tangent space at the
identity. It was a classical problem to determine which of the spheres $\mathbb{S}^n=\{x \in \R^n\colon \|x\|_2=1\}$ are parallelizable, then it has been shown that along with $\mathbb{S}^0$ and the unit circle $\mathbb{S}^1$ parallelizable are only $\mathbb{S}^3$ and $\mathbb{S}^7$~\cite{Milnor}. 

One way to obtain  a local bases $\eta_1^x, \eta_1^x \in T_x\St $ is to choose a point $z\in \St$  and an orthonormal  basis $\eta_1^z, \eta_1^z \in T_z\St $ and then set 
\begin{equation}
    \eta_i^x= \left\{\begin{matrix} \e^{d(x,z)\cdot [\frac{z\times x}{\sin{(d(x,z))}}] \eta_i^z}, & x \ne -z,\\
    -\eta_i^z,& x= -z,  \end{matrix}\right.
\end{equation}
where for a vector $v\in\R^3 $
\begin{equation*}
    [v]= \begin{pmatrix}
     0 & -v_3& v_2\\
     v_3& 0&-v_1\\
     -v_2 & v_1& 0
    \end{pmatrix}
\end{equation*}
is the corresponding skew-symmetric matrix in the algebra $\mathfrak{so}(3)$. In other words, we rotate the local
coordinate system at $z$, which is continuous for all points but the antipodal point $-z$. 

Combination of  the coordinates of a vector $v\in T_x\St $ in the basis  $\eta_1^x, \eta_1^x \in T_x\St $ with the exponential map~\eqref{exp_map_sphere}  yields the normal coordinates
centered at $x\in\St$, i.e. we parametrize a neighborhood of $x$ by 
\begin{equation} 
    \varphi(v_1,v_2)= (\cos{\|v\|_2 t } )x+ \sin{(\|v\|_2t)}\frac{v_1\eta_1^x+ v_2\eta_2^x}{\|v\|_2}.
\end{equation}
and the inverse parametrization for $y \in \St$ in a neighborhood of $x\in\St$ is given by 
\begin{equation}
v_i(y)= \frac{d(y,x)}{\sin{(d(x,y))}} \lr{y,\eta_i^x}, \quad i= 1, 2. 
\end{equation}
Moreover, the vectors 
\begin{equation}
    \frac{\partial}{\partial v_1}\varphi(v_1(y),v_2(y)), \quad  \frac{\partial}{\partial v_2}\varphi(v_1(y),v_2(y))
\end{equation}
form a basis of  $T_y\St $. One can show that in the center of the normal coordinates, the derivatives of the basis vectors have the following properties   
    \begin{equation}
    \frac{\partial^2}{\partial^2 v_i}\varphi(v_1(x),v_2(x))= -x, \quad  \frac{\partial^2}{\partial v_j\partial v_j}\varphi(v_1(x),v_2(x))=0. 
\end{equation}
Since in a normal coordinate system centered at $x\in \St$, the Christoffel symbols vanish at the point $x \in \St$, the gradient of a differentiable function $f:\St \to \C$ at the point $x \in \St$ has the
representation
\begin{equation}\label{gradient_in_normal_coo}
    \nabla f (x)= \begin{pmatrix}
     X_1f(x)\\
     X_2f(x)
    \end{pmatrix},
\end{equation}
where the differential operators $X_i$, $i=1,2$,   are defined by
\begin{equation}\label{eq:1:0}
     X_if(x)=\frac{\partial}{\partial v_1} (f\circ \varphi) (v_1(x),v_2(x))= \lim_{t\to 0} t^{-1}(f(\gamma_{x,\eta_i^x}(t))-f(x)).
\end{equation}
The  Hessian  matrix of a twice differentiable function ${f:\St \to \C}$ in the center $x\in \St$ of the normal coordinates, has also a special
representation, namely 
\begin{equation}\label{hessian_in_normal_coo}
   Hf(x)= \begin{pmatrix}
     X_1 X_1f(x)& X_1 X_2f(x) \\
     X_2 X_1f(x) & X_2X_2f(x)
    \end{pmatrix}. 
\end{equation}

However, the representations \eqref{gradient_in_normal_coo} and  \eqref{hessian_in_normal_coo} are only true in the center of the normal coordinates, since as mentioned before the Christoffel symbols vanish. For different
points, we would have to compute the Christoffel symbols with respect to the normal coordinates, which becomes quite complicated.
Alternatively, we introduce a second set of coordinates, such that the computation of the Christoffel
symbols is much more simpler. For a point $z\in \St$, we  parametrize the set $B_\pi (0)\setminus \{0\}$ by 
\begin{equation}\label{vrtheta}
    v(r, \theta)= r(\cos(\theta)\eta_1^z+ \sin(\theta)\eta_2^z)
\end{equation}
for $(r,\theta)\in (0,\pi)\times[0,2\pi)$. Combining \eqref{vrtheta} with the exponential map~\eqref{exp_map_sphere}, i.e.
\begin{equation}\label{polar_coordinates}
    \varphi^{\mathrm{pol}}(r,\theta)= \exp_z(v(r, \theta)), 
\end{equation}
yields the {\it polar coordinates} centered at $z\in \St$, which  parametrize ${\St\setminus \{z,-z\}}$. To give an example, for $z= (0,0,1)^{\mathrm{T}}$, these are the usual spherical coordinates on the sphere, given by
\begin{equation*}
  \varphi^{\mathrm{pol}}(r,\theta)=  \begin{pmatrix}
     \sin{r}\cos{\theta} \\
     \sin{r} \sin{\theta} \\
     \cos{r}
    \end{pmatrix}. 
\end{equation*}
Following the line, for a vector $x\in \St\setminus \{z,-z\}$, the inverse parametrization is given by
\begin{equation*}
\begin{split}
     r(x) &=\ar(\lr{x,z})= d(x,z),\\
   \theta(x) &=\mathrm{arctan}_2(\lr{x,\eta_2^z}, \lr{x,\eta_1^z}),
\end{split}
\end{equation*}
where $\mathrm{arctan}_2(x,y)$ denotes the $\mathrm{arctan}$ of $\frac{y}{x}$ with respect to the different branches of the tangent function,
which means
\begin{equation*}
    \cos{(\theta(x))} =\frac{\lr{x,\eta_1^z}}{\sin{(d(x,z))}},\quad \quad \sin{(\theta(x))} =\frac{\lr{x,\eta_2^z}}{\sin{(d(x,z))}}.
\end{equation*}
For each vector $x\in \St\setminus \{z,-z\}$, the vectors 
\begin{equation}\label{eq:1:1}
    \gamma_1^x= \frac{\partial \varphi^{\mathrm{pol}}}{\partial r}(r(x),\theta(x)), \quad \gamma_2^x=\frac{1}{\sin{(r(x))}} \frac{\partial \varphi^{\mathrm{pol}}}{\partial \theta}(r(x),\theta(x)), 
\end{equation}
form an orthonormal basis of the tangent space $T_x\St $. Notably, we have $\gamma_2^x=\gamma_1^x\times x $ and $\gamma_1^x=x\times \gamma_2^x $.  We remark, due to the singularities at the poles $z,-z$ there is no basis in spherical coordinates of the corresponding tangent spaces.

In the polar coordinates \eqref{polar_coordinates}, the Riemannian metric takes the form
\begin{equation}
   g(r,\theta)= \begin{pmatrix}
     1&0 \\
     0& \sin^2{(r)}
    \end{pmatrix}, 
\end{equation}
and the Christoffel symbols in these coordinates are therefore given by
\begin{equation}
   \Gamma^r(r,\theta)= \begin{pmatrix}
     0&0 \\
     0& -\sin{r}\cos{r}
    \end{pmatrix}, 
    \quad 
    \Gamma^\theta(r,\theta)= \begin{pmatrix}
     0&\cot{r} \\
     \cot{r} & 0
    \end{pmatrix}.
\end{equation}

For a twice differentiable function ${f:\St \to \C}$ and a vector ${x\in \St\setminus\{z,-z\}}$, 
the Hessian matrix with respect to the polar
coordinates centered at $z$, i.e. with respect to the basis \eqref{eq:1:1}, is represented by
\begin{equation}\label{hessian_polar}
  Hf=\begin{pmatrix}
     \frac{\partial^2 f\circ\varphi^{\mathrm{pol}}}{\partial r^2}&
     \frac{1}{\sin{r}} \frac{\partial^2 f\circ\varphi^{\mathrm{pol}}}{\partial r \partial \theta} \\
     \frac{1}{\sin{r}} \frac{\partial^2 f\circ\varphi^{\mathrm{pol}}}{ \partial \theta \partial r}& \frac{1}{\sin^2{r}} \frac{\partial^2 f\circ\varphi^{\mathrm{pol}}}{ \partial \theta^2}
    \end{pmatrix}- \frac{1}{\sin^2{r}} \frac{\partial f\circ\varphi^{\mathrm{pol}}}{\partial r}\Gamma^r- \frac{1}{\sin{r}}\frac{\partial f\circ\varphi^{\mathrm{pol}}}{\partial \theta}\Gamma^\theta.
\end{equation}

\subsection{Spherical harmonics}

Now,  let us describe the involved basis functions, known as spherical harmonics. For a detailed overview  see \cite{AtkHan}.

Let consider the space $L^2(\St)$ of all functions ${f:\St \to \C}$  such that 
\begin{equation*}
    \|f\|_2= \left(\int_{\St} |f(x)|^2\dx \Omega(x)\right)^{\frac{1}{2}}=\left(\int_{0}^{2\pi}\int_{0}^{\pi} |f(r,\theta)|^2 sin{(r)}\,\dx r\dx \theta \right)^{\frac{1}{2}} <\infty, 
\end{equation*}
where $\Omega$ is the Riemannian volume form on $\St$ for the metric $g$, and the second part of the equality represents $\|f\|_2$  in the spherical coordinates. 
It is well known, that the space  $L^2(\St)$  can be decomposed into an orthogonal sum
\begin{equation*}
L^2(\St)= \mathrm{cl}_{\|\cdot\|_{L_2}}\bigoplus\limits_{l=0}^\infty H_l
\end{equation*}
where $H_l$ is the eigenspace to the eigenvalue $\lambda_l=-l(l+1)$ of the Laplace-Beltrami operator on $\St$ with $\mathrm{dim}(H_l)=2l+1$. 
Looking for an  orthonormal  system of eigenfunctions of the Laplace-Beltrami operator on the unit sphere leads to the spherical harmonics. Namely,  the \textit{spherical harmonics} $Y_m^{l}: \St \to \C $ of degree $m\in \N$ and order $l$ are functions
\begin{equation*}
    Y_m^{l}(x(r, \theta))= N_{\ell m} P_{\ell}^m(\cos{(r)})\e^{i m\theta}, 
\end{equation*}
where  $ -l\le m \le l$,  the normalization constant $N_{\ell m}$ is given by 
\begin{equation*}
    N_{\ell m}= \frac{1}{\sqrt{2\pi}}\sqrt{\frac{2l+1}{2}\frac{(l-m)!}{(l+ m)!}},
\end{equation*}
 $P_{l}^m  (t)$ are {\it associated Legendre polynomials}, $r\in[0, \pi ]$, $\te \in[0, 2\pi )$ are  the inclination  and  azimuth respectively. 
With this,  we have that the system
\begin{equation*}
    \{Y_m^{l}\colon l\in \N, \; -l \le m \le l\}
\end{equation*}
constitutes  an orthonormal basis of  $L^2(\St)$.
Moreover, the following addition theorem is valid
 \begin{equation}\label{addition_theorem}
   P_{l}(\lr{x,y}):= P_{l}^0(\lr{x,y}) =\frac{4\pi}{ 2l+ 1} \sum\limits_{m=-l}^{l}Y_{m}^{l}(x)\overline{Y_{m}^{l}(y)}.
\end{equation}
The  space of all finite linear combinations
of spherical harmonics with degree less or equal to $N$ will be denoted as
\begin{equation*}
    \Pi_N(\St):= \mathrm{span}\{Y_m^{l}, -l\le m \le l, l\le N\}
\end{equation*}
and will be called generalized polynomials of degree $N$.
The projection operator onto the set of generalized polynomials $\Pi_N(\St)$, is given by 
\begin{equation}\label{projection_onto_Pi_N}
  \pP_N\colon L^2(\St)\to C(\St) \quad \text{with} \quad   \pP_N f(x)= \int_{\St} f(y) D_N(x,y)\dx \Omega(y),
\end{equation}
where the Dirichlet kernel on the sphere reads as 
\begin{equation*}
    D_N(x,y)= \sum\limits_{l=0}^N\frac{2l+ 1}{4\pi} P_{l}(\lr{x,y}). 
\end{equation*}

\subsection{Super-Resolution on the Unit Sphere}\label{SR_section}

In the following, we will introduce the problem of super-resolution, or in another words, so called deconvolution problem, on the unit sphere sphere $\St$, i.e. exact recovery of Dirac measure from its moments with respect to the spherical harmonics up to a degree $N$. 

To this aim, let us consider a weighted superposition of spikes 
\begin{equation*}
    \mu^\star= \sum\limits_{i=1}^M c_i\delta_{x_i}
\end{equation*}
where $ M\in \N $, $\delta_{x_i}$ is a Dirac measure centered at pairwise distinct  $x_i\in  \St$, $c_i \in \R $ are real valued amplitudes. 
We assume that all parameters $M, c_i, x_i$ are unknown and we  can only access
\begin{equation}\label{eq:1:2}
    \pP_N^*\mu^\star(x)= \int_{\St} D_N(x,y)\dx \mu^\star(y)
\end{equation}
for possibly low degree $N$. 
The super-resolution problem is to
recover the unknown locations $\X=\supp{(\mu^\star)}=\{x_i\}_{i=1}^M\subset \St$
and the coefficients  $c_i$ from the low frequency information~\eqref{eq:1:2}. 
Due to consideration of the total variation norm as being the continuous analog of the $\ell_1$ norm, in the light of convex optimization, see \cite{ DeCastro2012, Candes2013, Candes2014, Bendory2015a, FrankKrisrof2016},  the super-resolution problem can be formulated as the following minimization problem 
\begin{equation*}\label{eq:1:3}
 \tag{RP}
    \min\limits_{\mu \in \M (\St, \R)} \|\mu\|_{\mathrm{TV}},\quad\mbox{ subject to }\quad \pP_N^*\mu=\pP_N^*\mu^\star,
\end{equation*}
where the minimization is carried out over the set of all  finite measures $\mu$ supported on $\St$, and  the \textit{total variation} for a signed Borel measure is defined  by
\begin{equation*}
    \|\mu\|_{\mathrm{TV}}= |\mu|(\St)= \sup\sum_j|\mu(B_j)|,
\end{equation*}
and the supremum is taken over all partitions $B_j$ of $\St$. 

The main ingredient to ensure the fact $\mu^\star$ is the unique minimizer of the convex program \eqref{eq:1:3} is  the existence of a dual interpolating polynomial ${q\in \Pi_N(\St)}$, or so-called \textit{dual certificate}, that for each sign sequences ${u_i\in \{-1,1\}}$  satisfies 
\begin{equation}\label{eq:1:4}
\begin{split}
    q(x_i)&= u_i,  \quad x_i \in \X\\
|q(x)|&<1, \quad x \in \St \setminus\X. 
\end{split}
\end{equation}
The connection between the  uniqueness of optimal solution and the existence of a dual certificate
has been exploited in different settings, see e.g. \cite{DeCastro2012, Bredies2013, Candes2014, Bendory2015a, FrankKrisrof2016}.
In order to fulfill the interpolating conditions \eqref{eq:1:4} let us consider the following  \textit{Hemite-type} interpolation problem 
\begin{equation}\label{interpolation_problem}
\begin{split}
q(x_i)&= u_i,  \\
X_1 q(x_i)&= X_2q(x_i)=0
\end{split}
\end{equation}
for $x_i\in \X $, where $X_k$ are the differential operators defined in \eqref{eq:1:0}. This means, we ask not only the interpolation, but also for  local extrema at all the interpolation points. In order to approach the interpolation problem~\eqref{interpolation_problem},  we consider a kernel $J_N\colon \St\times\St \to \C$, such that the kernel $J_N(\cdot, y)$ itself,  and its derivatives $X_k^y J_N(\cdot, y)$, where the superscript indicates the action of the differential operators on the second variable, are generalized polynomials of order $N$ in the first variable, i.e. $J_N(\cdot, y)$, $X_k^y J_N(\cdot, y) \in \Pi_N(\St)$ for all vectors $y\in \St$. The dual certificate we construct is of the form 
\begin{equation}\label{DC_s1}
    q(x)= \sum\limits_{i=0}^{M}\alpha_{0,i}J_N(x, x_i)+ \alpha_{1,i}X_1^y J_N(x, x_i)+\alpha_{2,i}X_2^y J_N(x, x_i).
\end{equation} 
Thus, is easily follows that due to the construction $q\in \Pi_N(\St)$.
Applying the interpolation conditions \eqref{interpolation_problem} leads to the linear system of equations 
\begin{equation}\label{eq:1:7}
    K\alpha=   \begin{pmatrix}
J_N& X_1^x J_N & X_2^x J_N\\
X_1^y J_N & X_1^x X_1^y J_N&  X_2^x X_1^y J_N\\
X_2^y J_N & X_1^x X_2^y J_N&  X_2^x X_2^y J_N
   \end{pmatrix}
   \begin{pmatrix}
\alpha_0\\
\alpha_1 \\
\alpha_2 
   \end{pmatrix}=
   \begin{pmatrix}
u\\
0 \\
0
   \end{pmatrix}.
\end{equation}
where the block $J_N$ in the matrix $K$ correspond to matrix of the form 
$J_N=(J_N(x_i,x_j))_{i,j=1}^M$, and the same we have for the derivatives. The entries in the vectors are given by $\alpha_k=(\alpha_{k,j})_{j=1}^M$, $k=0,1,2$, and $u=(u_j)_{j=1}^M$. To find the vector of  coefficients $\al$, we need to show that the matrix $K$ is invertible. To this aim, we follow the general idea provided \cite{Candes2014, FrankKrisrof2016}. Namely,  due to the block structure of $K$ one can proof the invertibility of $K$ using an {\it iterative block inversion} and the fact that a matrix $A$ is invertible if $
    \|I-A\|_\infty<1,
$
where $\|A\|_\infty=\max_i \sum_j |a_{i,j}|$. In this case the norm of the  inverse matrix is bounded by
\begin{equation*}
    \|A^{-1}\|_\infty<\frac{1}{1-\|I-A\|_\infty}. 
\end{equation*}
Thus, to show the invertibility of $K$, we need to make use of  {\it  localization properties } %
of the entries of $K$, i.e. we need to bound the expressions $| J_N(x_i,x_j)|$
$|X_k^y J_N(x_i,x_j)|$ and  $|X_n^xX_k^yJ_N(x_i,x_j)|$. 
Namely,  we are looking for the estimation for the kernel $J_N$ of the form 
\begin{equation*}
    |J_N(x_i,x_j)|< \frac{c}{((N+1) d(x_i,x_j))^s},
\end{equation*}
with some constants $s$ and $c$, and for equivalent bounds for the derivatives. Using these estimates we can find explicit bounds on the supremum norm of  the coefficients $\alpha_k$. Once  the coefficients are bounded, we need to show that $|q(x)|<1$ for $x \in \St\setminus \X$. To this aim, we need to show convexity property of $q$, that can be done by estimating the entries of the \textit{Hessian matrix} of $q$. Since by construction \eqref{DC_s1} $q$  alredy includes first derivatives of $J_N$,   it will put into consideration third mixed derivatives of $J_N$ in the Hessian. Therefore, we also need some estimates of the derivatives of third order.  All this together results in  the topic of the next sections, where we choose a specific kernel and show needed  locality estimations. 

\section{Localized kernels}\label{loc_kernels_sec}

The aim of this section is to show localization results for the constituents of dual certificate~\eqref{DC_s1} for a particularly chosen interpolation kernel $J_N$. 

Let us star discussing the choice of the interpolation kernel. First, the interpolation kernel need to have an expansion in terms of generalized polynomials. Therefore we choose a kernel that owns the representation
\begin{equation*}
    J_N(x,y)= \sum_{l=0}^N \widetilde{w}_{l}\sum\limits_{m=-l}^{l}Y_m^{l}(x) \overline{Y_m^{l}(y)},
\end{equation*}
hence, due to the constriction  we have $J_N(\cdot,y),$ $ X_n^yJ_N(\cdot,y) \in {\Pi}_N(\St)$ for all $y\in \St$. Moreover, using the addition theorem~\eqref{addition_theorem} for spherical harmonics leads to 
\begin{equation*}
     J_N(x,y)= \widetilde{J}_N(d(x,y))= \sum_{l=0}^N\frac{2 l+1}{4\pi}\widetilde{w}_{l}P_{l}(\lr{x,y})
\end{equation*}
where $\widetilde{J}_N$ is a trigonometric polynomial. This means that $J_N$ is a zonal function, i.e. its value only depends on the
distance between $x$ and $y$. Such a property of $J_N$ lends us a hand in deriving estimates of the interpolation kernel $J_N$, since on condition localization estimates can derived from localization principles for the trigonometric polynomial $\widetilde{J}_N$. 



Relying on the discussed above, as an interpolation kernel we choose the specific kernel, given by 
\begin{equation}\label{Jackson}
    J_N(x,y)= \widetilde{J}_N(d(x,y))= \frac{1}{(\lrf{N/2}+1)^4}\frac{\sin^4\big((\lrf{N/2}+1) d(x,y)/2\big)}{\sin^4(d(x,y)/2)},
\end{equation}
i.e. the classical Jackson kernel evaluated at the distance between $x,y\in \St$.

\begin{lem}
The Jackson kernel $J_N(x,y)$ has an expansion of the form 
\begin{equation}\label{Jackson_as_sum}
    J_N(x,y)=\sum_{l=0}^N \frac{2l+1}{4\pi}\widetilde{w}_{l}P_{l}(\lr{x,y})
\end{equation}
with positive Legendre coefficients 
$  \ds {\widetilde{w}_{l}= 2\pi \int_{-1}^1 P_{l}(t)\widetilde{J}_N(\ar(t))\,\dx t,} $
and thus $J_N(\cdot,y), X_n^yJ_N(\cdot,y) \in \Pi_N(\St)$ for all $y\in \St$.
\end{lem}

\begin{proof}
For $t\in [-1,1]$, let us consider the function $\widetilde{F}_n$  defined as 
\begin{equation*}
    \widetilde{F}_n(\ar(t))=\frac{1}{n+1}\frac{\sin^2((n+1)^2\ar(t)/2)}{\sin^2(\ar(t)/2)}
\end{equation*}
i.e. Fej\'{e}r kernel evaluated at $\ar{(t)}$. Then for the trigonometric polynomial  $\widetilde{J}_N$ it holds 
\begin{equation*}
    \widetilde{J}_N(\ar{(t)})= \widetilde{F}_n^2(\ar(t)), \quad n=\lrf{N/2}
\end{equation*}
As it was shown in \cite{KeinerKunis2007},  the Fej\'{e}r kernel can be represented in terms of Legendre polynomials as 
\begin{equation}
    \widetilde{F}_n(\ar(t))= \sum_{l=0}^M\frac{2 l+1}{4\pi}\widetilde{v}_{l}P_l(t),
\end{equation}
which shows that $\widetilde{F}_n(d(x,\cdot)), \widetilde{F}_n(d(\cdot,y)) \in \Pi_n(\St)$,  and therefore $\widetilde{J}_N(d(x,\cdot))$, $\widetilde{J}_N(d(\cdot,y)) \in \Pi_N(\St)$. The positivity of $\widetilde{w}_{l}$ follows from the positivity of the linearization coefficients of a product of two Legendre polynomials 
\begin{equation*}
    P_n(t)P_m(t)= \sum_{l=0}^{\min (m,n)}\frac{2m+2n-4l+1}{2m+2n-2l+1}\frac{A(m-l)A(l)A(n-l)}{A(n+m-l)}P_{m+n-2}(t),
\end{equation*}
where $ A(m)=\frac{1\cdot3\cdot5 \cdot \, \cdots \, \cdot (2m-1)}{m!},$ for details see \cite{Adams}. 
\end{proof}

Before we proceed with stating the necessary localization estimates, we would like to discuss the behaviour of the  Jackson kernel. The geodesic distance of the sphere  behaves in a special way with respect to the boundedness of the derivatives. Namely the derivatives of the geodesic distance $d(x,y)$ have true poles at $x=y$ and $x=-y$. Whereas the first case can be handled using the point-wise estimates, the second case can not be covered in the same way, since the sign  of $\cos(kd(x,y))$ and $\sin(kd(x,y))$ alternates with $k$ in a neighborhood of $x=-y$. Consequently, we need to deal with the singularity at $x=-y$ induced by the derivatives of the geodesic distance in different way. Especially,  we need to bound the following trigonometric expressions
\begin{equation}\label{Gs}
\begin{split}
     G_1(\w)&= \frac{\widetilde{J}'_N(\w)}{\sin \w},    \quad  \quad  \quad \quad  \quad G_2(\w)= \widetilde{J}''_N(\w) \frac{\widetilde{J}'_N(\w)\cos \w}{\sin \w},  \\
    G_3(\w)&=\frac{\widetilde{J}''_N(\w)\sin \w- \widetilde{J}'_N(\w)\cos \w}{\sin^2\w}, 
    \end{split}
\end{equation}
that appear in the spherical derivatives. 
As we see, knowing the asymptotic estimates for $\widetilde{J}'_N$ and $\widetilde{J}''_N$ is not sufficient, so we have to consider the difference in the closed form, and  to achieve this, we use the closed form expression  \eqref{Jackson} of the Jackson kernel.

\begin{lem}\label{lemma_1}
Let $\widetilde{J}_N$ be a trigonometric polynomial defined in \eqref{Jackson}, then for some $\w\ne0$ and for  $n=\lrf{N/2}$, the following estimations holds true 
\begin{equation*}
    \begin{split}
 |\widetilde{J}_N(\w)|&\le \frac{\pi^4}{(n+1)^4|\w|^4} \quad |\widetilde{J}'_N(\w)|\le \frac{3\cdot \pi^4}{(n+1)^3|\w|^4} \quad  |\widetilde{J}''_N(\w)|\le \frac{12.5\cdot \pi^4}{(n+1)^2|\w|^4}
\\ |G_1(\w)|&\le \frac{2\cdot\pi^4}{(n+1)^2|\w|^4} \quad |G_2(\w)|\le  \frac{14.5\cdot\pi^4}{(n+1)^2|\w|^4}\quad |G_3(\w)|\le\frac{8\cdot\pi^4}{(n+1)|\w|^4}\\
    |\widetilde{J}'''_N(\w)|&\le \frac{50.5\cdot\pi^4}{(n+1)^4|\w|^4},       
    \end{split}
\end{equation*}
where the functions  $G_1,G_2,G_3$ are defined in \eqref{Gs}. Furthermore,  for $\ds|\w|\le\frac{\pi}{4(n+1)}$ we have
\begin{equation*}
    \begin{split}
|\widetilde{J}''_N(\w)-\cos(\w) G_1(\w)|&\le
\frac{\widetilde{J}^{(4)}_N(0)}{2}|\w|^2,  \; |\widetilde{J}''_N(0)- \widetilde{J}''_N(\w)|\le \frac{\widetilde{J}^{(4)}_N(0)}{2}|\w|^2
\\
|G_2(\w)|&\le \frac{\widetilde{J}^{(4)}_N(0)}{2}|\w|^2, \; |G_3(\w)|\le 0.52 \cdot \widetilde{J}^{(4)}_N(0) |\w|, \\
\left| \frac{\widetilde{J}'_N(\w)-\cos(\w)\sin(\w)\widetilde{J}''_N(\w)}{\sin^2(\w)} \right|&\le 0.52\cdot \widetilde{J}^{(4)}_N(0) |\w|, \quad |\widetilde{J}'''_N(0) |\le \widetilde{J}^{(4)}_N(0)|\w|,
 \end{split}
\end{equation*}
and  for $\w=0$ we have 
\begin{equation*}
    \begin{split}
   \widetilde{J}_N(0)&=1, \quad \quad  \widetilde{J}'_N(0)=\widetilde{J}'''_N(0)=0, \quad \quad  \widetilde{J}''_N(0)=-\frac{n(n+2)}{3}, \\
\widetilde{J}^{(4)}_N(0)&= \frac{1}{30}n(n+1)(9n(n+2)-2).
 \end{split}
\end{equation*}
\end{lem}
For the proof of Lemma~\ref{lemma_1} please see  Appendix~\ref{Lemma_1}. 

Now, using these bounds  we can formulate the localization property of the spherical derivatives of the Jackson kernel. More precisely, the bounds on the derivatives in normal coordinates are stated in the Theorem~\ref{deriv_bound_general} and the those in polar coordinates are derived in Theorem~\ref{deriv_bound_spheric}. In the proofs, we use several identities regarding the cross product, see Appendix~\ref{cross}. 

\begin{thm}\label{deriv_bound_general}
The Jackson kernel~\eqref{Jackson} fulfills for $x\ne y$ and $n=\lrf{N/2}$
\begin{align*}
   |J_N(x,y)|&\le \frac{\pi^4}{(n+1)^4d(x,y)^4}, \quad |X_n^y J_N(x,y)|\le \frac{3\cdot\pi^4}{(n+1)^3d(x,y)^4}, \\
|X_i^x X_n^y J_N(x,y)|&\le \frac{16.5\cdot\pi^4}{(n+1)^2d(x,y)^4}, \quad |X_i^x X_n^x J_N(x,y)|\le \frac{16.5\cdot\pi^4}{(n+1)^2d(x,y)^4}\\
|X_j^x X_i^x X_n^x J_N(x,y)|&\le \frac{101\cdot\pi^4}{(n+1)d(x,y)^4}
\end{align*}
and for the case $x= y$ we have 
\begin{align*}
   J_N(x,x)&=1, \quad X_i^x X_i^y J_N(x,x)= -X_i^x X_n^x J_N(x,x)= -\widetilde{J}''_N(0)\\
   X_n^yJ_N(x,x)&=X_iX_n^y J_N(x,x)= X_j^x X_i^y X_n^y J_N(x,x)=0.
\end{align*}

\end{thm}

\begin{proof} The first estimate follows directly from Lemma~\ref{lemma_1}. For the second estimate, we first calculate the derivative of $J_N$ using its representation~\eqref{Jackson_as_sum} to get 
\begin{equation}\label{X_n_y_Legendre}
    X_n^y J_N(x,y)= \sum\limits_{l=0}^N \ds\frac{2l+1}{4\pi}\widetilde{w}_{l}P_{l}'(\lr{x,y})\lr{x, \eta_n^y}
\end{equation}
which immediately yields $X_n^y J_N(x,x)=X_n^y J_N(x,-x)=0 $. In case $x\in \St\setminus\{ y,-y\}$, we have for $X_n^y J_N(x,y)$ the following representation 
\begin{align*}
    X_n^y J_N(x,y)&= - \widetilde{J}'_{N}(d(x,y))\frac{\lr{x,\eta_n^y}}{\sin(d(x,y))}= \pm  \widetilde{J}'_{N}(d(x,y))\frac{\lr{x,\eta_i^y\times y}}{\sin(d(x,y))}\\
    &=\mp \widetilde{J}'_{N}(d(x,y))\frac{\lr{\eta_i^y,x\times y}}{\sin(d(x,y))}=\mp \widetilde{J}'_{N}(d(x,y))\lr{\eta_i^y, n_{x,y}}, 
\end{align*}
where $ n_{x,y}$ denotes the unique unit vector perpendicular to $x$ and $y$. This together with Lemma~\ref{lemma_1} yields  the second estimate. For the next estimates, let us compute the second derivatives  using the representation $X_n^y J_N$ in terms of Legendre polynomials \eqref{X_n_y_Legendre} to get 
\begin{align*}
    X_i^x X_n^y J_N(x,y)&= \sum\limits_{l=0}^N \ds\frac{2l+1}{4\pi} \widetilde{w}_{l}\Big( P_{l}''(\lr{x,y})\lr{x, \eta_n^y}\lr{y, \eta_i^x}+ P_{l}'(\lr{x,y}) \lr{\eta_i^x, \eta_n^y } \Big)\\
     X_i^x X_n^x J_N(x,y)&= \sum\limits_{l=0}^N \ds\frac{2l+1}{4\pi} \widetilde{w}_{l}\Big( P_{l}''(\lr{x,y})\lr{\eta_n^x,y}\lr{y,\eta_i^x}-\delta_{in}
     P_{l}'(\lr{x,y})\lr{x,y}\Big).
\end{align*}
For $i\ne n$, it follows $X_i^x X_n^y J_N(x,x)=X_i^x X_n^y J_N(x,-x)= X_i^x X_n^x J_N(x,x)=X_i^x X_n^x J_N(x,-x)=0$. In case $i=n$, we have 
\begin{equation}\label{second_deriv_at_xx}
    \begin{split}
    -X_i^x X_i^x J_N(x,x)&= X_i^x X_i^y J_N(x,x)= \sum\limits_{l=0}^N \ds\frac{2l+1}{4\pi} \widetilde{w}_{l} P_{l}'(1)\\
    &= \lim\limits_{t\to 1} - \frac{\widetilde{J}'_{N}(\ar{(t)})}{\sqrt{1-t^2}}= \lim\limits_{\w\to 0} - \frac{\widetilde{J}'_{N}(\w)}{\sin{(\w)}}
    = -\widetilde{J}''_{N}(0). 
    \end{split}
\end{equation}
Following the same line, we obtain
\begin{equation}\label{second_deriv_at_x-x}
    \begin{split}
    -X_i^x X_i^x J_N(x,-x)= X_i^x X_i^y J_N(x,-x)&= \sum\limits_{l=0}^N \ds\frac{2l+1}{4\pi} \widetilde{w}_{l} P_{l}'(-1)\\
   &=\lim\limits_{\w\to \pi} \frac{\widetilde{J}'_{N}(\w)}{\sin{(\w)}}
    = -\widetilde{J}''_{N}(\pi). 
    \end{split}
\end{equation}

For $x\in \St\setminus\{ y,-y\}$, using the closed form of $X_n^y J_N$, the second derivatives  reads as
\begin{equation*}
    \begin{split}
    X_i^x X_n^y J_N(x,y)&=\widetilde{J}''_{N}(d(x,y))\frac{\lr{x,\et_n^y}\lr{\et_i^x,y}}{\sin^2{d(x,y)}}\\
   &- \widetilde{J}'_{N}(d(x,y)) \left(\frac{\lr{x,\et_n^y}\lr{\et_i^x,y}\cos{(d(x,y))}}{\sin^3{(d(x,y))}}+ \frac{\lr{\et_i^x,\et_n^y}}{\sin{(d(x,y))}}\right)\\
   &= \lr{\et_i^y,n_{x,y}}\lr{\et_n^x,n_{x,y}}G_2(d(x,y))-\lr{\et_i^x,\et_n^y}G_1(d(x,y)),
    \end{split}
\end{equation*}
and analogically we get 
\begin{equation*}
    X_i^x X_n^x J_N(x,y)= \lr{\et_i^x,n_{x,y}\lr{\et_n^x, n_{x,y}}}G_2(d(x,y))+\delta_{in}\cos{(d(x,y))}G_1(d(x,y)), 
\end{equation*}
where $G_1,G_2$ are defined in \eqref{Gs} and again $n_{x,y}$ denotes the unique unit vector perpendicular to both $x$ and $y$. This results in
\begin{equation*}
    \begin{split}
    |X_i^x X_n^x J_N(x,y) |, |X_i^x X_n^y J_N(x,y)|&\le |G_2(d(x,y))|+|G_1(d(x,y))|\\
    &\le \frac{16.5\cdot \pi^4}{(n+1)^2d(x,y)^4},
    \end{split}
\end{equation*}
and yields the estimates for the second derivatives. The third derivatives can be represented as
\begin{equation*}
    \begin{split}
        X_j^x X_i^x X_n^y J_N(x,y)&= \sum\limits_{l=0}^N \ds\frac{2l+1}{4\pi} \widetilde{w}_{l} \left(P_{l}'''(\lr{x,y})\lr{\et_j^x,y}\lr{\et_i^x,y}\lr{\et_j^y,x}\right.\\
        &+P_{l}''(\lr{x,y})\lr{\et_i^x,y}\lr{\et_n^y,\et_j^x}+P_{l}''(\lr{x,y})\lr{\et_i^x,\et_n^y} \lr{y,\et_j^x}\\
        &\left.-\delta_{ij}\lr{x,\et_n^y}\big(P_{l}'(\lr{x,y})+P_{l}''(\lr{x,y})\lr{x,y}\big)\right)
    \end{split}
\end{equation*}
which immediately shows that 
\begin{equation*}
   X_j^x X_i^x X_n^y J_N(x,x)=X_j^x X_i^x X_n^y J_N(x,-x)=0. 
\end{equation*}
Using the closed form of the second derivatives for $x\in \St\setminus\{ y,-y\}$, the third derivatives can also be represented as
\begin{equation*}
    \begin{split}
        X_j^x X_i^x X_n^y J_N(x,y)&= \frac{\lr{\et_j^x,\et_n^y}-\delta_{i j}\lr{x,\et_n^y}\lr{x,y}}{\sin{(d(x,y))}}G_3(d(x,y))\\
        &-\frac{\lr{x,\et_n^y}\lr{\et_i^x,y}\lr{\et_j^x,y}\lr{x,y}}{\sin^3{d(x,y)}}\left(\widetilde{J}_{N}'''(d(x,y))+\widetilde{J}'_{N}(d(x,y))\right.\\
        &-3\cos{(d(x,y)}G_3(d(x,y))\Big)\\
        &+\frac{\lr{\et_i^x,\et_n^y}\lr{\et_j^x,y}}{\sin{d(x,y)}}G_3(d(x,y))+\delta_{ij}\cos{(d(x,y))}G_1(d(x,y)).
    \end{split}
\end{equation*}
Hence, using the estimates from Lemma~\ref{lemma_1}~, we obtain 
\begin{equation*}
    |X_j^x X_i^x X_n^y J_N(x,y)|\le\frac{101\cdot\pi^3}{(n+1)d(x,y)^4},
\end{equation*}
which finishes the proof. 
 \end{proof}
 
 \begin{thm}\label{deriv_bound_spheric}
Let $x,y,z\in \St$ be pairwise different and $x\ne -z$,then  with ${n=\lrf{N/2}}$ the entries of  the Hessian matrix $H$ of $J_N(\cdot,y)$, $X_n^y J_N (\cdot,y)$ in polar coordinates centered at $z$ fulfills 
 \begin{equation*}
    \begin{split}
       |(H J_N(x,y))_{i i}|&\le\frac{16.5\cdot\pi^4}{(n+1)^2d(x,y)^4},  \quad |(H J_N(x,y))_{i j}|\le\frac{14.5\cdot\pi^4}{(n+1)^2d(x,y)^4}  \\
       |(H X_n^y J_N(x,y))_{ii}|&\le\frac{103\cdot\pi^4}{(n+1)d(x,y)^4}, \quad 
       |(H X_n^y J_N(x,y))_{i j}|\le\frac{93.5\cdot\pi^4}{(n+1)d(x,y)^4}. 
       \end{split}
\end{equation*}
In case $y=z$ and $d(x,y)\le \frac{\delta}{(n+1)}$ with $\delta\le \pi /4$, we have 
\begin{equation*}
    \begin{split}
       |\widetilde{J}_{N}''(0)-(H J_N(x,z))_{i i}|&\le \frac{3}{20}\delta^2(n+1)^2,  \quad |(H J_N(x,z))_{i j}|\le \frac{3}{20}\delta^2 (n+1)^2\\
       |(H X_n^y J_N(x,z))_{11}|&\le \frac{3}{10}\delta (n+1)^3, \quad |(H X_y^n J_N(x,z))_{22}|\le \frac{1}{5}\delta(n+1)^3\\
       |(H X_n^y J_N(x,z))_{i j}|&\le \frac{1}{5}\delta(n+1)^3. 
       \end{split}
\end{equation*}
 \end{thm}
 \begin{proof}
As it has been discussed at the beginning,  in polar coordinates centered at $z\in \St$, there is  the local parametrization 
 \begin{equation*}
   \phi^{\mathrm{pol}}  (r,\te )= \cos{(r)}z+\sin{(r)}\big(\cos{(\te)}\et_1^z+\sin{(\te)}\et_2^z\big), 
 \end{equation*}
 where $\et_1^z$, $\et_1^z $ form an orthonormal  basis  of $T_z \St$, such that $\et_2^z=\et_1^z\times z $, and the implicit inverse parametrization  for $x\in \St\setminus\{ z,-z\}$ is given by 
 \begin{equation*}
     \begin{split}
         r(x)&=d(x,z), \\
         \cos{(\te(x))}&=\frac{\lr{x,\et_1^z}}{\sin{(d(x,z))}}=-\lr{\et_2^z,n_{z,x}},\\
         \sin{(\te(x))}&= \frac{\lr{x,\et_2^z}}{\sin{(d(x,z))}}=\lr{\et_1^z,n_{z,z}}. 
     \end{split}
 \end{equation*}
First, let us compute the partial derivatives of the function $f_{\xi}(r, \te)=\lr{\phi^{\mathrm{pol}}  (r,\te ), \xi} $, for a vector  $\xi \in \St$. Carrying out simple computations results in 
 \begin{equation*}
     \begin{split}
         \pd_r  f_{\xi}(r, \te)&= - \sin{(r)}\lr{z,\xi}+\cos{(r)}
         \big( \cos{(\te)}\lr{\et_1^z, \xi}+\sin{(\te)}\lr{\et_2^z,\xi}\big),  \\
         \pd_{\te} f_{\xi}(r, \te)&=  \sin{(r)}\big( \cos{(\te)}\lr{\et_2^z,\xi}-\sin{(\te)} \lr{\et_1^z,\xi}\big), 
         \\
         \pd^2_{r r} f_{\xi}(r, \te)&=  - \cos{(r)}\lr{z,\xi}- \sin{(r)}\big( \cos(\te)\lr{\et_1^z,\xi}+ \sin(\te)\lr{\et_2^z,\xi}\big)
         \\
         \pd^2_{\te \te} f_{\xi}(r, \te)&=- \sin{(r)}\big(\cos(\te)\lr{\et_1^z,\xi}+ (\sin(\te)\lr{\et_2^z,\xi} \big)
         \\
          \pd^2_{r \te} f_{\xi}(r, \te)&= \pd^2_{\te r} f_{\xi}(r, \te)= \cos{(r)}\big( \cos{(\te)}\lr{\et_2^z,\xi}- \sin{(\te)}\lr{\et_1^z,\xi} \big). 
    \end{split}
 \end{equation*}
Thereafter, inserting the inverse  parametrization and simplifying the obtained expression, we get 
 \begin{equation*}
     \begin{split}
         \pd_r  f_{\xi}(r, \te)&=  \sin{(d(x,\xi))}\lr{n_{x,\xi},n_{z,x}},  \\
         \pd_{\te} f_{\xi}(r, \te)&= \sin{(d(x,z))}\sin{(d(x,\xi))} \lr{x, n_{x,\xi}\times n_{z,x}}, 
         \\
         \pd^2_{r r} f_{\xi}(r, \te)&=  - \cos{(d(x,\xi))},
         \\
         \pd^2_{\te \te} f_{\xi}(r, \te)&=- \sin{(d(x,z))} \cos{(d(x,z))}, \sin{(d(x,\xi))}\lr{n_{x,\xi},n_{z,x}} \\
         &\quad - \sin^2{(d(x,z)) \cos{(d(x,\xi))}},\\ 
          \pd^2_{r \te} f_{\xi}(r, \te)&= \pd^2_{\te r} f_{\xi}(r, \te)= \cos{(d(x,z))}\sin{(d(x,\xi))}\lr{x,n_{x,\xi}, n_{z,x}}. 
    \end{split}
 \end{equation*}
Now, let us  proceed with calculation of the full derivatives. For abbreviation, we use further the following notation $$f_{\xi}(x)= f_{\xi}(r(x), \te(x))=\lr{\phi^{\mathrm{pol}}  (r(x),\te(x) ), \xi}. $$ 
First, we start with the kernel $J_N(\cdot,y)$ and assume that $x\ne -y$. For the first element of the Hessian matrix we have
\begin{equation}\label{(H_J_N)_11}
    \begin{split}
        (H J_N(x,y))_{1 1}&= (\pd_r f_{y}(x))^2\left(\frac{\widetilde{J}_{N}''(\ar{(f_{y}(x))})}{(1-f_{y}(x)^2)}\right.\left.- \frac{\widetilde{J}_{N}'(\ar{(f_{y}(x)) f_{y}(x)}}{(1- f_{y}(x)^2)^{3/2}}\right)\\
        &- \pd^2_{r r} f_{y}(x)\frac{\widetilde{J}_{N}'(\ar{(f_{y}(x))}}{(1-f_{y}(x)^2)^{1/2}}\\
        &= (\lr{n_{x,y},n_{z,x}})^2 G_2(d(x,y))+\cos(d(x,y))G_1(d(x,y)), 
    \end{split}
\end{equation}
where $G_1$ and $G_2$ are given in \eqref{Gs}. Thus, using Lemma~\ref{lemma_1}, we get 
\begin{equation*}
    |(H J_N(x,y))_{1 1}|\le \frac{16.5\cdot\pi^4}{(n+1)^2 d(x,y)^4}. 
\end{equation*}

If $x=-y$, we again use the polynomial representation of $J_N$, given in the poof of Lemma~\ref{lemma_1}, to derive
\begin{equation*}
    \begin{split}
        (H J_N(x,-x))_{1 1}&= \sum\limits_{l=0}^N \frac{2l+1}{4\pi} \widetilde{w}_{l}\Big( P_{l}''(\lr{x,-x})(\pd_r f_{-x}(x))^2\\
        &+ P_{l}'(\lr{x,-x})\pd_{r r}^2 f_{-x}(x)\Big)\\
        &= \sum\limits_{l=0}^N \frac{2l+1}{4\pi}\widetilde{w}_{l} \Big( P_{l}''(-1)(\lr{x\times (-x), n_{z ,x}})^2\\
        &- P'_{l}(-1)\cos(d(x,-x))\Big) =  \sum\limits_{l=0}^N \frac{2l+1}{4\pi}\widetilde{w}_{l}P'_{l}(-1).
    \end{split}
\end{equation*}
Following the same argumentation as in \eqref{second_deriv_at_x-x}, we get 
\begin{equation*}
 (H J_N (x,-x))_{1 1}=  \widetilde{J}_N(\pi). 
\end{equation*}
In case $y=z$ and $d(x,z)\le \frac{\pi}{4(n+1)}$, from the representation \eqref{(H_J_N)_11}, we get
\begin{equation*}
(H J_N (x,-x))_{1 1}=  \widetilde{J}_N(d(x,z)). 
\end{equation*}
and combining this together with Lemma~\ref{lemma_1} results in 
\begin{equation*}
    |\widetilde{J}''_N(0) - (H J_N (x,z))_{11}|\le \frac{\widetilde{J}^{(4)}_N(0)}{2} d(x,z)^2\le \frac{3}{20}d(x,z)^2(n+1)^4\le \frac{3}{20} \delta^2(n+1)^2.
\end{equation*}
For the  second diagonal entry of the Hessian matrix, we have
\begin{equation*}
    \begin{split}
        (H J_N(x,y))_{2 2}&= \lr{x, n_{x,y}\times n_{z,x}}^2\left(\widetilde{J}''_N(d(x,y))-\frac{\widetilde{J}'_N(d(x,y)) \cos{(d(x,y))}}{ \sin{(d(x,y))})} \right)\\
        &+  \cos{(d(x,y))}\frac{\widetilde{J}'_N(d(x,y))}{\sin{(d(x,y))}}\\
        &= \lr{x, n_{x,y}\times n_{z,x}}^2 G_2(d(x,y))+ \cos{(d(x,y))}G_1(d(x,y))
    \end{split}
\end{equation*}
which provides the following estimate 
\begin{equation*}
    |(H J_N(x,y))_{2 2}|\le \frac{16.5\cdot \pi^4}{(n+1)^4 d(x,y)^4}
\end{equation*}
and gives zero entry $ (H J_N(x,-x))_{2 2}=0$ for $y=-x$.

In case $y=z$, we get 
\begin{equation*}
    (H J_N(x,z))_{2 2}=\cos{(d(x,z))}\frac{\widetilde{J}'_N(d(x,z))}{\sin{(d(x,z))}},
\end{equation*}
that together with Lemma~\ref{lemma_1} provides the estimate
\begin{equation*}
    |\widetilde{J}''_N(0) - (H J_N (x,z))_{22}|\le \frac{\widetilde{J}^{(4)}_N(0)}{2} d(x,z)^2\le \frac{3}{20}d(x,z)^2(n+1)^4\le \frac{3}{20} \delta^2(n+1)^2.
\end{equation*}

For the off-diagonal entries of the Hessian matrix, we have 
\begin{equation*}
    (H J_N (x,y))_{i j}= \lr{n_{x,y},n_{z,x}}\lr{x, n_{x,y}\times n_{z,x} } G_3(d(x,y)),
\end{equation*}
and analogous procedure to the one carried out  above provides
\begin{equation*}
 \begin{split}
    |(H J_N (x,y))_{i j}|&\le \frac{14.5\cdot \pi^4}{(n+1)^2  d(x,y)^4},  \quad    |(H J_N (x,z))_{i j}|\le \frac{3}{20} \delta^2 (n+1)^2, \\
    (H J_N (x,-x))_{i j}&=0.
     \end{split}
\end{equation*}
 For the derivative $X_n^y J_N(\cdot, y)$ of the kernel the estimates can be derived folowing the same lines. Namely, the derivative $X_n^y J_N(\cdot, y)$ is defined as 
 \begin{equation*}
     X_n^y J_N(x, y)= -\frac{\widetilde{J}'_N(d(x,y))}{\sin{(d(x,y))}} \lr{x,\et_n^y},
 \end{equation*}
where $\et_n^y\in T_z\St$. Then, with $\xi= \et_n^y$ and the abbreviations in \eqref{Gs}, we have 
\begin{equation*}
    \begin{split}
        \pd_r X_n^y J_N(x, y)&= \frac{\lr{x,\xi} \pd_r f_{y}(x)}{\sin(d(x,y))} G_3(d(x,y))- \pd_r f_{\xi}(x) G_1(d(x,y))\\
        \pd_{\te} X_n^y J_N(x, y)&= \frac{\lr{x,\xi} \pd_{\te} f_{y}(x)}{\sin(d(x,y))} G_3(d(x,y))- \pd_{\te} f_{\xi}(x) G_1(d(x,y)). 
    \end{split}
\end{equation*}
In polar coordinates centered at $z\in \St$, we have for $y\ne -x$ and $\psi\in \{r, \te\}$ 
\begin{equation*}
    \begin{split}
        \pd_{\psi} G_1(d(x,y))&=- \frac{\pd_{\psi} f_{y}(x)}{\sin(d(x,y))}G_3(d(x,y)),\\
        \pd_{\psi}\left(\frac{G_3(d(x,y))}{\sin(d(x,y))} \right) &= - \frac{\pd_{\psi} f_{y}(x)}{\sin^3(d(x,y))}\Big(\widetilde{J}'''_N(d(x,y))+ \widetilde{J}'_N(d(x,y))\\
        & \quad - 3\cos{(d(x,y))}G_3(d(x,y))\Big). 
    \end{split}
\end{equation*}
Therefore, for the first diagonal entry of the Hessian matrix one has
\begin{equation*}
    \begin{split}
        (H X_n^y J_N(x,y))_{11}& = \pd_{r r}X_n^y J_N(x,y)\\ 
        &= \frac{\pd_{r} f_{\xi}(x)\pd_{r} f_{y}(x) }{\sin(d(x,y))}G_2(d(x,y))+\frac{\lr{x,\xi}\pd_{r r} f_{y}(x)}{\sin(d(x,y))}G_3(d(x,y))\\
        &+ \lr{x,\xi} \pd_{r} f_{y}(x) \pd_r\left( \frac{G_3(d(x,y))}{\sin{(d(x,y))}}\right)\\
        &- \pd_{r r}f_{\xi}(x)G_1(d(x,y))-\pd_{r}f_{\xi}(x)  \pd_r G_1(d(x,y)).
    \end{split}
\end{equation*}
Inserting the pre-computed derivatives results in 
\begin{equation*}
    \begin{split}
        (H X_n^y J_N(x,y))_{11}&= 2 \lr{n_{x,\xi},n_{z,x} }\lr{n_{z,x},n_{x,y} }\sin{(d(x,\xi))}G_3(d(x,y))\\
        &- \frac{\lr{x,\xi} \cos{(d(x,y))}}{\sin{(d(x,y))}}G_3(d(x,y))+ \lr{x,\xi} G_1(d(x,y))\\
        &-\frac{\lr{x,\xi}\lr{n_{z,x},n_{x,y}}^2}{\sin{(d(x,y))}}\left(
       \widetilde{J}'''_N(d(x,y))+ \widetilde{J}'_N(d(x,y)) \right)\\
      & +\frac{3\lr{x,\xi} \lr{n_{z,x},n_{x,y}}^2\cos{(d(x,y))}}{\sin{(d(x,y))}}G_3(d(x,y)).
    \end{split}
\end{equation*}
Since the following equalities  hold 
\begin{equation*}
    \lr{x,\xi}= \lr{x,\et_n^y}=\pm \lr{\et_i^y,x\times y}= \pm \lr{\et_i^y,n_{x,y}} \sin(d(x,y)),
\end{equation*}
we have that 
\begin{equation*}
    \left|\frac{\lr{x,\xi}}{\sin{(d(x,y))}} \right|\le 1.
\end{equation*}
Then using Lemma~\ref{lemma_1},  the above derived representation of $(H X_n^y J_N(x,y))_{11}$ we get
\begin{equation*}
    |(H X_n^y J_N(x,y))_{11}|\le \frac{103.5 \cdot \pi^4}{(n+1)d(x,y)^4}.
\end{equation*}
In case $y=-x$, we again use the polynomial representation of $J_N$, which directly provides 
  $  (H X_n^y J_N(x,-x))_{11}=0.  $

In case $y=z$ and $d(x,z)\le \frac{\pi}{4(n+1)}$, we have 
\begin{equation*}
    (H X_n^y J_N(x,-x))_{11}=-\frac{\lr{x,\xi}}{\sin{(d(x,y))}} \widetilde{J}'''_N(d(x,z)),
\end{equation*}
and again making use of Lemma~\ref{lemma_1}, we get the estimate
\begin{equation*}
    |(H X_n^y J_N(x,-x))_{11}|\le \widetilde{J}^{(4)}_N(0) d(x,z)\le \frac{3}{10}\delta(n+1)^3. 
\end{equation*}
Using an analogous procedure as above, the second diagonal entry can be represented as 
\begin{equation*}
    \begin{split}
        (H X_n^y J_N(x,y))_{22}&= 2\lr{x, n_{x,\xi}\times n_{z,x} }\lr{x, n_{x,y}\times n_{z,x}}  \sin(d(x,\xi)) G_3(d(x,y))\\
        &- \frac{\lr{x,\xi}\cos{(d(x,y))}}{\sin{(d(x,y))}}G_3(d(x,y))+\lr{x,\xi}G_1(d(x,y))\\
        &- \frac{\lr{x,\xi}\lr{x, n_{x,y}\times n_{z,x}}^2}{\sin{(d(x,y))}} \left( \widetilde{J}'''_N(d(x,y))+ \widetilde{J}'_N(d(x,y))\right),\\
        &+ \frac{3\lr{x,\xi} \lr{x, n_{x,y}\times n_{z,x}}^2\cos(d(x,y))}{\sin{(d(x,y))}}G_3(d(x,y)),
    \end{split}
\end{equation*}
which shows that 
\begin{equation*}
    |(H X_n^y J_N(x,y))_{22}|\le \frac{103.5\cdot \pi^4}{(n+1)d(x,y)^4}.
\end{equation*}
Moreover, we also have $(H X_n^y J_N(x,-x))_{22}=0$, and using Lemma~\ref{lemma_1} gives 
\begin{equation*}
\begin{split}
    |(H X_n^y J_N(x,z))_{22}|&= \left|\frac{\lr{x,\xi}}{\sin^2(d(x,y))}\left(\widetilde{J}'_N(d(x,z))\right.\right.\\
    &\left.\left.-\cos(d(x,z))\sin{(d(x,z))}\widetilde{J}''_N(d(x,z))\right)\right|\\
    &\le 0.52 \cdot \widetilde{J}^{(4)}_N(0) d(x,z)\le \frac{1}{5}\delta(n+1)^3. 
    \end{split}
\end{equation*}
Lastly, similar computation for  the off-diagonal entries shows 
\begin{equation*}
 \begin{split}
    (H X_n^y J_N(x,y))_{i j}&= \lr{x,\lr{n_{x,y},n_{z,x}}(n_{x,\xi}\times n_{z,x})}\sin{(d(x,\xi))}G_3(d(x,y))\\
    &+ \lr{x,\lr{n_{x,\xi},n_{z,x}}(n_{x,y}\times n_{z,x})}\sin{(d(x,\xi))}G_3(d(x,y))\\
    &-\frac{\lr{x,\xi}\lr{x,n_{x,y}\times n_{z,x}}\lr{n_{x,y},n_{z,x}}}{\sin{(d(x,y))}}\left(\widetilde{J}'''_N(d(x,y)) +\widetilde{J}'_N(d(x,y))\right)\\
    &+ \frac{3 \lr{x,\xi}\lr{x, n_{x,y}\times n_{z,x}}\lr{n_{x,y}, n_{z,x}}\cos{(d(x,y))}}{\sin{(d(x,y))}}G_3(d(x,y)), 
     \end{split}
\end{equation*}
which results in 
\begin{equation*}
 \begin{split}
    |(H X_n^y J_N(x,y))_{i j}|&\le \frac{93.5\cdot \pi^4}{(n+1)d(x,y)^4}, \quad (H X_n^y J_N(x,-x))_{i j}=0, \\
    |(H X_n^y J_N(x,z))_{i j}|&= \left|\lr{x, n_{x,\xi}\times n_{z,x}} \sin{(d(x,y))}G_3(d(x,y)) \right|\\
    & \le \frac{1}{5} \delta (n+1)^3.
     \end{split}
\end{equation*}
 \end{proof}
 
 We finish this section bounding sums of point-wise expression from the previous theorems. 
 
 Let us consider  a discrete set $\X \subset \St$ such that the minimal separation distance between the set elements is bounded from below in the following way 
\begin{equation}\label{separation}
    \rho(\X)= \min\limits_{x_i,x_j\in \X, x_i\ne x_j} d(x_i,x_j)\ge \frac{\nu}{n+1}. 
\end{equation}
Involving classical ringing arguments on the sphere, the following result holds true.  

\begin{lem}\label{summation_lemma}
Let $x_m\in \X$, where $\X\subset \St$ is a discrete set that satisfies a separation condition \eqref{separation}. Let a vector $x\in\St$ be such that $d(x,x_m)\le \varepsilon \frac{\nu}{n+1}$ for $0\le \varepsilon\le 1/2$. Suppose that a function $f\colon \St \to \St $ fulfills 
\begin{equation}\label{locality_of_f}
    |f(x,y)|\le \frac{c_f}{((n+1) d(x,y))^s}
\end{equation}
for $x\ne y$ and some $s\ge 3$, then 
\begin{equation*}
    \sum\limits_{x_k\in \X \setminus \{x_m\}} |f(x,x_k)|\le \frac{c_f a_{\varepsilon}}{\nu^s},
\end{equation*}
where $a_{\varepsilon}= \zeta(s-1)\cdot \min\{9\cdot (1-\varepsilon)^{-s}+25, 25\cdot (1-\varepsilon)^{-s} \}$. Here $\zeta$ denotes the Riemannian Zeta function.
\end{lem}

Now, based on Lemma~\ref{summation_lemma} we obtain the following bound for the point-wise summation of the Jackson kernel and its derivatives. 

\begin{lem}\label{summation_lemma_for_Jackson}
Let $x_m\in \X$, where $\X\subset \St$ is a discrete set, which obeys a separation condition \eqref{separation}. Let $x\in\St$ such that $d(x,x_m)\le \varepsilon \frac{\nu}{n+1}$ for $0\le \varepsilon \le 1/2$, then it follows 
\begin{equation*}
\begin{split}
   & \sum\limits_{x_k\in \X \setminus \{x_m\}} |J_N(x,x_k)|\le \frac{\pi^4 \cdot \tilde{a}_{\varepsilon}}{\nu^4}, \; \sum\limits_{x_k\in \X \setminus \{x_m\}} |X_n^y J_N(x,x_k)|\le \frac{3\cdot \pi^4 \cdot \tilde{a}_{\varepsilon} (n+1)}{\nu^4},\\
    & \sum\limits_{x_k\in \X \setminus \{x_m\}} |X_i^x X_i^x J_N(x,x_k)|, \sum\limits_{x_k\in \X \setminus \{x_m\}} |(H J_N(x,x_k))_{ii}| \le \frac{16.5 \cdot \pi^4 \cdot \tilde{a}_{\varepsilon} (n+1)^2}{\nu^4}, \\
   & \sum\limits_{x_k\in \X \setminus \{x_m\}} |X_i^x X_i^x X_j^y J_N(x,x_k)|\le \frac{101 \cdot \pi^4 \cdot \tilde{a}_{\varepsilon} (n+1)^3}{\nu^4},\\
    &\sum\limits_{x_k\in \X \setminus \{x_m\}} |(H J_N(x,x_k))_{i j}|\le \frac{14.5 \cdot \pi^4 \cdot \tilde{a}_{\varepsilon} (n+1)^2}{\nu^4},\\
    &\sum\limits_{x_k\in \X \setminus \{x_m\}} |(H X_i^y J_N(x,x_k))_{i i}|\le \frac{103.5 \cdot \pi^4 \cdot \tilde{a}_{\varepsilon} (n+1)^3}{\nu^4},\\
    &\sum\limits_{x_k\in \X \setminus \{x_m\}} |(H X_i^y J_N(x,x_k))_{i j}|\le \frac{93.5 \cdot \pi^4 \cdot \tilde{a}_{\varepsilon} (n+1)^3}{\nu^4},
\end{split}
\end{equation*}
where $\tilde{a}_{\varepsilon}= \zeta(3)\cdot \min\{9\cdot (1-\varepsilon)^{-4}+25, 25\cdot (1-\varepsilon)^{-4} \}$ and  $\zeta$ denotes the Riemannian Zeta function.
\end{lem}
\begin{proof}
The above given estimates follow immediately after combining the estimates from Theorems~\ref{deriv_bound_general} and \ref{deriv_bound_spheric} and the results of Lemma~\ref{summation_lemma}. 
\end{proof}
Having  established the necessary localization estimates for the Jackson kernel and its derivatives in this section, we are now able to construct and validate a dual certificate using the Hermite interpolation in the next section. 

\section{Dual Certificate on the Sphere}

\subsection{Solution of the Interpolation Problem}

In this section, we construct a dual certificate as a solution of the Hermite type interpolation problem. Let us remind, we would like to solve the interpolation problem 
\begin{equation}\label{int_conds}
    \begin{split}
   q(x_i)&=u_i,     \quad x_i \in \X\\
   |q(x)|&< 1,  \quad x\in \St\setminus \X, 
    \end{split}
\end{equation}
where $q$ needs to be an element of $\Pi_N(\St)$. Following Section~\ref{SR_section}, we consider the interpolant of the form 
\begin{equation}\label{DC}
    q(x)= \sum\limits_{i=0}^{M}\alpha_{0,i}J_N(x, x_i)+ \alpha_{1,i}X_1^y J_N(x, x_i)+\alpha_{2,i}X_2^y J_N(x, x_i). 
\end{equation}
and  would like to show the the linear system of equations  
\begin{equation}\label{interpol_system}
    K\alpha=   \begin{pmatrix}
J_N& X_1^x J_N & X_2^x J_N\\
X_1^y J_N & X_1^x X_1^y J_N&  X_2^x X_1^y J_N\\
X_2^y J_N & X_1^x X_2^y J_N&  X_2^x X_2^y J_N
   \end{pmatrix}
   \begin{pmatrix}
\alpha_0\\
\alpha_1 \\
\alpha_2 
   \end{pmatrix}=
   \begin{pmatrix}
u\\
0 \\
0
   \end{pmatrix}
\end{equation}
has a solution for the Jackson kernel $J_N$  chosen in Section~\ref{loc_kernels_sec}. 

We assume that the interpolation points $\X=\{x_j\}_{j=1}^M$ obey a minimal separation distance of the form 
\begin{equation}\label{separation_DC_section}
    \rho(\X)= \min\limits_{x_i,x_j\in \X, x_i\ne x_j} d(x_i,x_j)\ge \frac{\nu}{n+1},  
\end{equation}
where $n=\lrf{N/2}$. 
For abbreviation, we use the following notation
\begin{equation*}
    J_{ij}= X_i^x X_j^y J_N.
\end{equation*}
As have been discussed  before, to show the existence of the interpolating polynomial $q$, we need to show the invertibility  of the matrix $K$. Even more, we need to partially compute the inverse of  $K$ to derive bounds on the coefficients~$\alpha$. The following Lemma  provides bounds the entries of the interpolation matrix. 

\begin{lem}\label{lemma_with_Cs}
Suppose the set of points $\X=\{x_j\}_{j=1}^M$ satisfies  the separation condition \eqref{separation_DC_section}.  Then the entries of the interpolation matrix $K$ are bounded in the following way 
\begin{equation*}
    \begin{split}
        \|I-J_{0 0}\|&\le \frac{C_0}{\nu^4}, \quad  \quad  \quad \|-J''_{N}(0)I -J_{ii}\|_{\infty}\le \frac{C_2(n+1)^2}{\nu^4},\\
        \|J_{0 i}\|_{\infty}, \;  \|J_{i 0}\|_{\infty}&\le \frac{C_1 (n+1)}{\nu},  \quad \|J_{ i j}\|_{\infty}\le \frac{C_2(n+1)^2}{\nu^4}, \quad \text{for } i\ne j, \; i, j \ne 0,
    \end{split}
\end{equation*}
where the constants are defined as
\begin{equation}\label{C_constants}
    \begin{split}
        C_{0}=25 \cdot \zeta(3) \cdot \pi^4,\quad  C_{1}= 75 \cdot \zeta(3) \cdot \pi^4, \quad C_2= 412.5 \zeta(3) \cdot \pi^4. 
    \end{split}
\end{equation}
If $n\ge 9$ and $\nu^4 > \frac{3}{0.99}\cdot C_2$, then for the inverse of $J_{0 0}$ and $J_{ii}$ the following bounds are valid 
\begin{equation*}
    \|J_{0 0}^{-1}\|_{\infty} \le \frac{1}{1-\frac{C_0}{\nu^4}}, \quad \quad \|J_{ii}^{-1}\|_{\infty} \le \frac{3}{0.99 (n+1)^2 (1-\frac{3C_2}{0.99 \nu^4})}. 
\end{equation*}
\end{lem}
\begin{proof}
The bounds follows directly from Theorem~\ref{deriv_bound_general} and Lemma~\ref{lemma_1}, as well as from the invertibility of a matrix $A$, if
\begin{equation*}
    \|I-A\|_{\infty}< 1,
\end{equation*}
and the bounds on the norm of the inverse given by
\begin{equation*}
        \|A^{-1}\|_{\infty}\le \frac{1}{1- \|I-A\|_{\infty}}.
\end{equation*}
For the bound on $J_{ii}^{-1}$, one uses additionally that for $n\ge 9$, so that $${|J''_{N}(0)|=\frac{n(n+1)}{3}\ge \frac{0.99 \cdot (n+1)^2}{3}}$$. 
\end{proof}

Relying on the above stated results, we move on to the main Theorem of this section, which provides a condition on the separation of the interpolation points to guarantee the invertibility of the interpolation matrix $K$ and gives bounds on the coefficients~$\alpha$. 

\begin{thm}\label{coeff_bounds_th}
Suppose the separation condition \eqref{separation_DC_section} is satisfied by the set $\X=\{x_j\}_{j=1}^M$ for some $\nu\ge 0$ and a constant $b\ge 3$ such that 
\begin{equation*}
    \nu^4\ge \frac{3}{0.99} \cdot b\cdot C_2, 
\end{equation*}
where the constant $C_2$ is given in \eqref{C_constants}. Then the matrix $K$ is invertible and the entries of the coefficient vector $\al$ in~\eqref{interpol_system} satisfy 
\begin{equation*}
    \begin{split}
        \|\alpha_{0}\|_{\infty}&\le 1 + \frac{1}{45(b-2)-1}, \quad  \quad   \quad 
        \|\alpha_j\|_{\infty}\le \frac{(n+1)^{-1}}{4.5(b-2)-0.1}, \quad j= 1,2.
    \end{split}
\end{equation*}
Furthermore, for elements of the vector $\alpha_{0}$ we have the lower bound
\begin{equation*}
    |\alpha_{0,i}|\ge 1 -\frac{1}{45(b-2)-1}.
\end{equation*}
\end{thm}
\begin{proof} In the first place, we partition the matrix $K$ into blocks as 
\begin{equation*}
   K= \begin{pmatrix} K_0 & \widetilde{K}_1\\
   K_1 & K_2
   \end{pmatrix}
\end{equation*}
with the blocks given by 
\begin{equation*}
    \begin{split}
        K_{0}&=J_{00}=J_N, \\
       K_1&= \left[ J_{01} \; J_{02}\right]^{\trans}= \left[ X_1^y J_N \; X_2^y J_N \right]^{\trans},\\
        \widetilde{K}_1&= \left[ J_{10} \; J_{20}\right]^{\trans}= \left[ X_1^x J_N \; X_2^x J_N \right]^{\trans},\\
        K_2&=  \begin{bmatrix} J_{11} & J_{21}\\
        J_{12}& J_{22}\end{bmatrix}= \begin{bmatrix} X_1^x X_1^y J_N & X_2^x X_1^y J_N \\
        X_1^x X_2^y J_N & X_2^x X_2^y J_N \end{bmatrix}, 
    \end{split}
\end{equation*}
after that we use an iterative block inversion. For abbreviation, we set
\begin{equation*}
   a_1:= \frac{3\cdot C_2}{0.99 \nu^4}
\end{equation*}
that represents the quotient of the off-diagonal upper bound and the on-diagonal lower bound. 
Thus, the assumption of the theorem reads  now as 
\begin{equation*}
    a_1\le \frac{1}{b}, \quad \quad \text{where} \quad  b\ge 3.
\end{equation*}
Using the bounds from Lemma~\ref{lemma_with_Cs}, one can  shown that 
\begin{equation*}
    \left\|I - \frac{K_2/J_{22}}{\widetilde{J}''_N(0)}\right\|_{\infty}\le \frac{1}{b-1}. 
\end{equation*}
Due to this, the invertibility of the matrix $K_2$ follows easily from
\begin{equation*}
    \|(K_{2}/J_{22})^{-1}\|_{\infty}\le \frac{3(b-1)}{0.99 (n+1)^2(b-2)}.
\end{equation*}
With the abbreviation $T=K_2/J_{2 2}$, the inverse matrix of $K_2$ has the representation 
\begin{equation}\label{K_2_-1}
    K_2^{-1}= \begin{pmatrix} T^{-1} & -T^{-1}J_{21} (J_{22})^{-1}\\
    -(J_{22})^{-1} J_{12}T^{-1} & (J_{22})^{-1}+(J_{22})^{-1} J_{12}T^{-1} J_{21} (J_{22})^{-1}. 
    \end{pmatrix}
\end{equation}
In the next step, again relying on Lemma~\ref{lemma_with_Cs} and on  the bound for $(K_{2}/J_{22})^{-1}$,  it can be shown that 
\begin{equation*}
    \|I-K/K_2\|_{\infty}\le \frac{1}{45(b-2)},
\end{equation*}
which in turn results in 
\begin{equation*}
    \|(K/ K_2)^{-1}\|\le \frac{1}{1-\frac{1}{45(b-2)}}
\end{equation*}
and in the invertibility of $K$. Now, using the representation \eqref{K_2_-1} of the inverse of $K_2$ and the abbreviation $S=K/K_2$, one comes to  the solution of the interpolation problem is given by 
\begin{equation*}
    \begin{split}
        \alpha_{0}&= S^{-1}u\\
        \alpha_1 &= -T^{-1}(J_{01}-J_{21}(J_{22})^{-1} J_{02})\alpha_0\\
        \alpha_2&= - J_{22}^{-1}(J_{12}\alpha_{1}+J_{02}\alpha_0). 
    \end{split}
\end{equation*}
This yields the bounds 
\begin{equation*}
        \|\alpha_0\|_{\infty} \le 1+\frac{1}{45(b-2)-1},\quad  \quad 
        \|\alpha_1\|_{\infty} , \, \|\alpha_2\|_{\infty}\le \frac{(n+1)^{-1}}{4.5(b-2)-0.1}. 
\end{equation*}
Moreover, the absolute value of $\alpha_{0,i}$ has the lower bound 
\begin{equation*}
    \begin{split}
        |\alpha_{0,i}|&\ge \left|  \left( \left(I- \left(I-(K/K_2)^{-1}\right)\right)u\right)_i\right|\\
       &\ge  \left|  1- \left|\left( \left(I-(K/K_2)^{-1}\right)u\right)_i\right| \right|\\
        &\ge\left| 1- \left|\left( I- K/K_2\right)(K/K_2)^{-1} \right|\right|. 
    \end{split}
\end{equation*}
Since $b\ge 3$, the subtrahend in the last inequality can be bounded as
\begin{equation*}
    \left|(I-K/K_2)(K/K_2)^{-1}\right|\le\|I- K/K_2\|_{\infty}\|(K/K_2)^{-1}\|_{\infty}, 
\end{equation*}
which results in the corresponding bound for  $\alpha_{0,i}$, namely 
\begin{equation*}
    |\alpha_{0,i}|\ge 1 -\frac{1}{45(b-2)-1}. 
\end{equation*}
\end{proof}

Looking for explicit bounds for the coefficients $\al$, we obtain the following result. 

\begin{cor} \label{solution_coeff_bounds_cor}
Suppose the set of interpolation points $\X= \{x_i\}_{m=0}^M$ satisfies a separation distance of 
\begin{equation}
    \rho(\X)\ge \frac{19.2\pi}{N}
\end{equation}
for $N\ge 20$. Then the interpolation problem \eqref{interpol_system} has a unique solution, such that
\begin{equation*}
            \|\alpha_{0}\|_{\infty}\le 1+ 6.3 \cdot 10^{-3}, \quad  \quad   
             \|\alpha_{1}\|_{\infty}, \, \|\alpha_{2}\|_{\infty}\le \frac{6.3 \cdot 10^{-2}}{(n+1)}, 
 \end{equation*}
 and for elements of the vector $\al_0$ we have the lower bound
 \begin{equation*}
     |\alpha_{0,i}|\ge 1- 6.3 \cdot 10^{-3}. 
 \end{equation*}
 \end{cor}
 \begin{proof}
 For $n= \lrf{N/2}$, we have $(n+1)\ge N/2$ and therefore 
 \begin{equation*}
     \rho(\X)\ge \frac{9.6 \pi }{(n+1)}. 
 \end{equation*}
It is easy to check that with $\nu = 9.6 \pi $ and $b=5.6$, we have 
 \begin{equation*}
     \nu^4\ge \frac{3}{0.99}\cdot b\cdot C_2.
 \end{equation*}
Thus,  it follows  from Theorem~\ref{coeff_bounds_th} that 
 \begin{equation*}
        \begin{split}
            \|\alpha_{0}\|_{\infty}&\le 1 + \frac{1}{45\cdot 5.6 -1}\le 1+ 6.3 \cdot 10^{-3}\\
             \|\alpha_{j}\|_{\infty}&\le \frac{(n+1)^{-1}}{4.5\cdot 5.6 -0.1}\le \frac{6.3 \cdot 10^{-2}}{(n+1)}, \quad j= 1,2, 
        \end{split}
 \end{equation*}
 and the lower bound for elements of $\alpha_{0}$ fulfills
 \begin{equation*}
     |\alpha_{0,i}|\ge 1 - \frac{1}{45\cdot 5.6-1}\ge 1- 6.3 \cdot 10^{-3}. 
 \end{equation*}. 
 \end{proof}
 
 \subsection{Bounds for the Interpolant}
 Using the derived bounds on the coefficients of the interpolant, we proceed by showing the upper bound in the absolute value of the interpolating function $q$ with coefficients of Corollary~\ref{solution_coeff_bounds_cor}, i.e. 
 \begin{equation*}
     |q(x)|\le 1,
 \end{equation*}
where $x$ is not an interpolation point. We split the proof into two parts. In the first place, we consider the points that are close to an interpolation point. This case in covered by  Lemma~\ref{q_less_1_for_close}, where we rely on  the convexity argument via the definiteness of Hessian matrix. Thereafter,  we set our sights on the bounds for points that are well separated from any interpolation points, which forms the content of Lemma~\ref{q_less_1_for_not_close}. 

\begin{lem}\label{q_less_1_for_close}
Suppose the set of interpolation points $\X= \{x_i\}_{m=0}^M$ satisfies the separation condition \eqref{separation_DC_section} and $x\in \St$ is such that $0<d(x, x_i)\le \frac{\pi}{6(n+1)}$, for an interpolation point $x_m\in \X$. Then the interpolating function $q$ from Corollary~\ref{solution_coeff_bounds_cor} fulfills 
\begin{equation*}
    |q(x)|\le 1. 
\end{equation*}
\end{lem}

\begin{proof}
Let us consider the interpolant of the form \eqref{DC}, namely 
\begin{equation*}
     q(x)= \sum\limits_{k=0}^{M}\alpha_{0,k}J_N(x, x_k)+ \alpha_{1,k}X_1^y J_N(x, x_k)+\alpha_{2,j}X_2^y J_N(x, x_k).
\end{equation*}
Without loss of generality, let us assume that for a vector $x_m\in \X$ we have ${q(x_m)=u_m=1}$. To show that $u_m=1$ is a local maximum of $q$,  one needs to show that the Hessian matrix of $q$ is  negative definite at $x_m$. 
Considering the normal coordinates at $x_m$, the Hessian matrix of $q$  at the interpolation point $x_m$ has the form
\begin{equation*}
    \widetilde{H}q(x_m)=
    \begin{pmatrix} 
    X_1 X_1 q(x_m) & X_1X_2 q(x_m)\\
    X_2 X_1 q(x_m)& X_2 X_2 q(x_m)
    \end{pmatrix}. 
\end{equation*}
Now, using the bounds of Theorem~\ref{deriv_bound_general}, Lemma~\ref{summation_lemma_for_Jackson} and Corollary~\ref{solution_coeff_bounds_cor}, we have 
\begin{equation}\label{late_night}
\begin{split}
    X_i^x X_i^x q(x_m) &= \sum\limits_{x_k\in \X }\alpha_{0,k}  X_i^x X_i^x J_N(x_m, x_k)+ \alpha_{1,k}  X_i^x X_i^x X_1^y J_N(x_m, x_k)\\
    &+X_i^x X_i^x \alpha_{2,k}X_2^y J_N(x_m, x_k)\\
    &= \alpha_{0,m}  X_i^x X_i^x J_N(x_m, x_m)+ \sum\limits_{x_k\in \X \setminus \{x_m\}}\Big(\alpha_{0,k}  X_i^x X_i^x J_N(x_m, x_k)\\
    &+ \alpha_{1,k}  X_i^x X_i^x X_1^y J_N(x_m, x_k)
    + \alpha_{2,k} X_i^x X_i^x X_2^y J_N(x_m, x_k)\Big)\\
    &\le -|\al_{0,m}|\cdot\frac{n(n+2)}{3}+ \Big( 16.5\cdot\nofty{\al_0}+ 202 \cdot\nofty{\al_1}\Big)\frac{ \tilde{a}_0 (n+1)^2}{(9.6)^4}.
    \end{split}
\end{equation}
Since $N\ge 20$, we have $n+1\ge 10$ and therefore
\begin{equation}\label{J_N_trick}
   |\widetilde{J}''_N(0)|= \frac{n(n+2)}{3}\ge  \frac{n(n+1)}{(n+1)^2}\cdot \frac{(n+1)^2}{3}\ge 0.99 \cdot \frac{(n+1)^2}{3}. 
\end{equation}
And together with \eqref{late_night} it provides the upper bound for the on-diagonal entries 
\begin{equation*}
   X_i X_i q(x_m)\le -0.2248 \cdot(n+1)^2. 
\end{equation*}
In the same way,  one can bound the off-diagonal entries by 
\begin{equation*}
   |X_i X_j q(x_m)|\le 0.1032\cdot (n+1)^2. 
\end{equation*}
Combing these two bounds, we obtain 
\begin{equation*}
    \tr(\widetilde{H}q(x_m))\le -0.4496(n+1)^2, \quad \mathrm{det}(\widetilde{H}q(x_m))\ge 0.0398(n+1)^2,
\end{equation*}
which means that the Hessian matrix of $q$ at point $x_m$ is strictly negative definite and, thus, $x_m$ is an isolated local maximal point of $q$ and $q(x_m)=1$ is local maximum. 

For $x\ne x_m$ with $d(x,x_m)\le \frac{\pi}{6(n+1)}$ we use the same arguments, but instead of the bounds from Theorem~\ref{deriv_bound_spheric} we use the bounds derived in Theorem~\ref{deriv_bound_spheric} and the corresponding estimates from Lemma~\ref{summation_lemma_for_Jackson}. Thus, for the on-diagonal entries we get 
\begin{equation*}
    \begin{split}
        &(H q(x))_{ii} = \\
        &=\sum\limits_{x_k\in \X } \al_{0,k} (H J_N(x,x_k))_{ii}+ \al_{1,k}(H X_1^y J_N(x,x_k))_{ii}+ \al_{2,k}(H X_2^y J_N(x,x_k))_{ii}\\
        &\le \al_{0,m} \widetilde{J}''_N(0)+ \nofty{\al_0}\left(|\widetilde{J}''_N(0)- (H J_N(x,x_m))_{ii}|  + \sum\limits_{x_k\ne x_m} |(H J_N(x,x_k))_{ii}| \right)\\
        &+ \nofty{\al_1}\left(|(H X_1^y J_N(x,x_m))_{ii}| + \sum\limits_{x_k\in \X \setminus \{x_m\}} |(H X_1^y J_N(x,x_k))_{ii}| \right)\\
        &+ \nofty{\al_2}\left(|(H X_2^y J_N(x,x_m))_{ii}| + \sum\limits_{x_k\in \X \setminus \{x_m\}} |(H X_2^y J_N(x,x_k))_{ii}| \right). 
    \end{split}
\end{equation*}
For the first entry $(H q(x))_{11}$, the above given inequality results  in  
\begin{equation*}
    \begin{split}
     (H q(x))_{11}&\le (n+1)^2 \left(-|\al_{0,m}|\frac{0.99}{3}+ \nofty{\al_0}\left(\frac{3}{20}\delta^2+ \frac{16.5 \cdot \tilde{a}_{\delta/\nu}}{(9.6)^4}\right)\right)   \\
     & + 2\nofty{\al_1} (n+1)^3\left(\frac{3}{10}\delta +\frac{103.5 \cdot \tilde{a}_{\delta/\nu}}{(9.6)^4}\right)
     \\
     &\le -0.1552 (n+1)^2,
    \end{split}
\end{equation*}
where $\delta= \pi/6$ and $\nu= 9.6 \pi$.
Following the same line, for  $(H q(x))_{22}$ and  $(H q(x))_{i j}$ for $i\ne j$ we obtain 
\begin{equation*}
    (H q(x))_{22}\le -0.1617 (n+1)^2, \quad |(H q(x))_{i j}| \le 0.1538 (n+1)^2,
\end{equation*}
that together provide us the following bounds 
\begin{equation*}
    \tr({H}q(x))\le -0.3169(n+1)^2, \quad \mathrm{det}({H}q(x))\ge 0.0004(n+1)^2,
\end{equation*}
Hence, the function $q$ is strictly concave on $B_{\frac{\pi}{6(n+1)}}(x_m)\setminus \{x_m\}$, which shows that $q(x)< 1$. Moreover, the Taylor expansion of the cosine and sine function shows 
\begin{equation*}
    \begin{split}
    J_N(x,x_m)&\ge 1 +\frac{\widetilde{J}''_N(0)}{2}d(x,x_m)^2\ge 1- \frac{(n+1)^2}{6}d(x,x_m)^2, \\
    |X_n^y J_N(x,x_m)|&\le |\widetilde{J}'_N(d(x,x_m))|\le |\widetilde{J}''_N(0)| d(x, x_m)\le \frac{(n+1)^2}{3} d(x,x_m), 
    \end{split}
\end{equation*}
meaning  that for $d(x, x_m)\le\frac{\pi}{6(n+1)} $ we have 
\begin{equation*}
    \begin{split}
        q(x)&\ge  \alpha_{0,m} J_{N}(x, x_m)- \Big( \|\alpha_1\|_{\infty} |X_1^y J_N(x,x_m)|+ \|\al_2 \|_{\infty}|X_2^y J_N(x,x_m)|\\
        &+  \sum\limits_{x_k\in \X \setminus \{x_m\}}  \|\al_{0}\|_{\infty}|J_N(x,x_k)|+ \nofty{\al_1}|X_1^y J_N(x,x_k)|\\
        &+\nofty{\al_2}|X_2^y J_N(x,x_k)|\Big)\ge 0.92.
    \end{split}
\end{equation*}
This shows that $0.92\le q(x)\le 1$ for $x$ such that $d(x, x_m)\le\frac{\pi}{6(n+1)} $. In case $q(x_i)=-1$, the analogous arguments with changing sings show that $x_i$ is an isolated local minimal point. 
\end{proof}

\begin{lem}\label{q_less_1_for_not_close}
Suppose the set of interpolation points $\X= \{x_i\}_{m=0}^M$ satisfies the separation condition \eqref{separation_DC_section}  \eqref{separation_DC_section} and $x\in \St$ is such that  $d(x, x_m)\ge \frac{\pi}{6(n+1)}$, for all interpolation point $x_m\in \X$. Then the interpolating function $q$ from Corollary~\ref{solution_coeff_bounds_cor} fulfills 
\begin{equation*}
    |q(x)|\le 1. 
\end{equation*}
\end{lem}

\begin{proof}

Here we split the proof into three parts, namely the first case corresponds to those $x\in \St$ with $\frac{\pi}{6(n+1)}\le d(x,x_m)\le \frac{11\cdot \pi}{10(n+1)}$, the second to $\frac{11\cdot \pi}{10(n+1)}\le d(x,x_m)\le \frac{4\cdot \pi}{(n+1)} $ and the third to those $x\in \St$ such that $ \frac{ 4 \cdot \pi}{(n+1)}\ge d(x,x_m)$ for all interpolation point $x_m\in \X$. 
So we need to bound 
\begin{equation}\label{DC_for_not_close}
    \begin{split}
     |q(&x)|=   \left|\sum\limits_{k=0}^{M}\alpha_{0,k}J_N(x, x_k)+ \alpha_{1,k}X_1^y J_N(x, x_k)+\alpha_{2,j}X_2^y J_N(x, x_k)\right|\\
     &\le  \nofty{\al_0}|J_N(x,x_m)|+ \nofty{\al_1}|X_1^y J_N(x,x_m)|+ \nofty{\al_2}|X_2^y J_N(x,x_m)|\\
     &+ \sum\limits_{x_k\ne x_m}  \|\al_{0}\|_{\infty}|J_N(x,x_k)|
   + \nofty{\al_1}|X_1^y J_N(x,x_k)|+\nofty{\al_2}|X_2^y J_N(x,x_k)|
    \end{split}
\end{equation}

In the first case, due to the Taylor expansion of the cosine and sine function and the positivity of the Jackson kernel we have 
\begin{equation}\label{eq 2:24}
    |J_N(x,x_m)|=\widetilde{J}_N(d(x,x_m))\le 1- \frac{|\widetilde{J}''_N(0)|}{2} d(x,x_m)^2+ \frac{|\widetilde{J}^{(4)}_N(0)|}{24} d(x,x_m)^4.
\end{equation}
Then using again  the fact for $N\ge 20 $,
$(n+1)\ge 10$ the inequality \eqref{J_N_trick} holds, and additionally 
\begin{equation*}
   |\widetilde{J}^{(4)}_N(0)|=\frac{1}{30}n(n+2)(9n(n+2)-2)\le \frac{3}{10}(n+1)^4,
\end{equation*}
 we can rewrite the inequality \eqref{eq 2:24}  for $t\in [\frac{\pi}{6}, \frac{11\pi}{10}]$  in the form 
 \begin{equation*}
     \widetilde{J}_N(t/(n+1))\le 1-\frac{0.99}{6}t^2+ \frac{1}{80}t^4. 
 \end{equation*}
 The polynomial on the right hand side is monotonic  decreasing for $t\in [\frac{\pi}{6}, t_0]$, where $t_0=\sqrt{\frac{20\cdot 0.99}{3}}\approx 2.5691$ and monotonic  increasing for $t\in [t_0, \frac{11\pi}{10}]$.  Similarly,  due to $|\sin{(r\w)}|\le k\w$ we have 
 \begin{equation*}
     |X_n^yJ_N(x,x_m)|\le |\widetilde{J}''_N(d(x,x_m))|\le |\widetilde{J}''_N(0)| d(x,x_m)\le \frac{(n+1)^2}{3} d(x,x_m). 
 \end{equation*}
  
 Accordingly, for $\frac{\pi}{6(n+1)} \le d(x,x_m) \le  \frac{11\pi}{10(n+1)}$, we can estimate \eqref{DC_for_not_close} using the bounds above and the estimates from Lemma~\ref{summation_lemma_for_Jackson} and Corollary~\ref{solution_coeff_bounds_cor} as 
 \begin{equation*}
     \begin{split}
         |q(x)|&\le \nofty{\al_0}\left(1-\frac{0.99}{6}t^2+ \frac{1}{80}t^4\right)+ 2\nofty{\al_1} \cdot \frac{t}{3}(n+1)\\
         &+\nofty{\al_0}\cdot\frac{\tilde{a}_{t/ \nu}}{(9.6)^4} + 2\nofty{\al_1} \cdot\frac{3\cdot \tilde{a}_{t/ \nu}\cdot(n+1)}{(9.6)^4},
     \end{split}
 \end{equation*}
 where $t=d(x,x_m)(n+1)$, which results in the bounds
 \begin{equation*}
     |q(x)|\le \left\{\begin{matrix} 0.9902
     & \frac{\pi}{6(n+1)} \le d(x,x_m) \le  \frac{t_0}{(n+1)},\\
  
    0.9682 &  \frac{t_0}{(n+1)} \le d(x,x_m) \le  \frac{11\pi}{10(n+1)},
     \end{matrix}\right. 
 \end{equation*}
and completes the first case. For the second case. i.e. $ \frac{11\pi}{10(n+1)} \le d(x,x_m) \le  \frac{4\pi}{(n+1)}$ we use the bounds from Theorem~\ref{deriv_bound_general}
\begin{equation*}
\begin{split}
     |J_N(x,x_m)|&\le \frac{\pi^4}{(n+1)^4 d(x,x_m)^4}\le \frac{\pi^4}{t^4},\\
     |X_n^y J_N(x,x_m)|&\le \frac{3\cdot\pi^4}{(n+1)^3 d(x,x_m)^4}\le \frac{3\pi^4}{t^4}(n+1),
\end{split}
\end{equation*}
for $t= d(x,x_m)(n+1)$. For this reason, we have the estimate 
\begin{equation*}
    |q(x)|\le \nofty{\al_0}\frac{\pi^4}{t^4}+ 2\nofty{\al_1} \frac{3\pi^4}{t^4}(n+1)\\
    +  \nofty{\al_0}\frac{\tilde{a}_{t/\nu}}{(9.6)^4}+ 2 \nofty{\al_1}\frac{3\tilde{a}_{t/\nu}\cdot(n+1)}{(9.6)^4},
\end{equation*}
which shows that for $ \frac{11\pi}{10(n+1)} \le d(x,x_m) \le  \frac{4\pi}{(n+1)}$ the absolute values of $q$ can be bounded as 
\begin{equation*}
    |q(x)|\le 0.9618. 
\end{equation*}
Lastly, in case $d(x,x_m)>  \frac{4\pi}{(n+1)}$ for all interpolation points $x_m\in \X$ the set $\X\cup \{x\}$ fulfills a separation distance of $\frac{4\pi}{(n+1)}$. Therefore we can again apply the bounds from Theorem~\ref{deriv_bound_general}, Lemma~\ref{summation_lemma} and Corollary~\ref{solution_coeff_bounds_cor} to bound
the absolute values of $q$. Thus, we have 
\begin{equation*}
\begin{split}
    |q(x)|&\le \sum\limits_{x_k\in \X}  \|\al_{0}\|_{\infty}|J_N(x,x_k)|
   + \nofty{\al_1}|X_1^y J_N(x,x_k)|+\nofty{\al_2}|X_2^y J_N(x,x_k)|\\
   & \le \|\al_{0}\|_{\infty}\frac{\tilde{a}_{0}}{4^4}+ 2\|\al_{1}\|_{\infty} \frac{3 \tilde{a}_{0}(n+1)}{4^4}\le 0.1618. 
\end{split}
\end{equation*}
\end{proof}
Combining Corollary~\ref{solution_coeff_bounds_cor}, Lemma~\ref{q_less_1_for_close} and Lemma~\ref{q_less_1_for_not_close} shows the existence of a dual certificate $q$. We summarize this result in the following theorem. 
\begin{thm}\label{DC_existence}
Suppose that the set $\X= \{x_i\}_{i=0}^M$ 
fulfills a separations distance of $\rho(\X)\ge \frac{19.2 \pi}{N}$ for $N\ge 20$. Then for each sign  combinations $u_i\in \{-1,1\}$, there exists a generalized polynomial  $q\in \Pi_N(\St)$ such that 
\begin{equation*}
    \begin{split}
   q(x_i)&=u_i,     \quad x_i \in \X,\\
   |q(x)|&< 1,  \quad x\in \St\setminus \X. 
    \end{split}
\end{equation*}
\end{thm}

The existence of a dual certificate immediately provides the recovery of a sought measure via the minimization of the total variation. 

\begin{cor}
Suppose the support of the singed measure $\mu^\star$  fulfills the separation condition 
\begin{equation*}
    \min\limits_{x\ne y} d(x,y)\ge \frac{19.2\pi}{N}, \quad x,y \in \supp{(\mu^\star)}\subset \St,
\end{equation*}
for $N\ge 20$, then the measure $\mu^\star$ is the unique solution of the minimization problem 
\begin{equation*}
    \min\limits_{\mu \in \M (\St, \R)} \|\mu\|_{\mathrm{TV}},\quad\mbox{ subject to }\quad \pP_N^*\mu=\pP_N^*\mu^\star.
\end{equation*}
\end{cor}
\begin{proof}
Theorem~\ref{DC_existence} guarantees the existence of a dual certificate. Hence, the operator $\pP^*$ has the null-space property  with respect to $\supp{(\mu^\star)}$ and the therefore $\mu^\star$ is a unique real solution of the minimization problem, which finishes the poof and this section. 
\end{proof}

\section{Numerical Solution }

\subsection{Semidefinite Formulation of the Optimization Problem}\label{subsec:5:1}

At first glance, finding the solution of the total variation minimization problem~\eqref{eq:1:3} might seem rather complicated since it is an infinite dimensional optimization problem over the whole measure space $\M (\St, \R)$,  therefore numerically not feasible. 
The aim of this section is to provide a formulation of the convex program  problem~\eqref{eq:1:3} such that it can be handled by the existing convex optimization engines. To do so, we follow the ideas in \cite{Candes2014, Bendory2015b, FrankKrisrof2016}. 

First, let us switch to the convex pre-dual to  \eqref{eq:1:3} that is  given by 
\begin{equation*}
    \max\limits_{f\in L^2(\St)} \mathrm{Re}\lr{\pP_N f,\pP_N^*\mu^\star} \quad\mbox{ subject to }\quad \nofty{\pP_N f} \le 1.  
\end{equation*}
Since $\pP_N$ in the projection operator given in~\eqref{projection_onto_Pi_N}, immediately  the equivalent formulation follows   
\begin{equation}\label{dual_problem_con}
\tag{dRP}
    \max\limits_{f\in\Pi_N(\St)} \mathrm{Re}\lr{f,\pP_N^*\mu^\star} \quad\mbox{ subject to }\quad \nofty{f} \le 1, 
\end{equation}
where the constraint imposes that the modulus of the generalized polynomial 
\begin{equation*}
\begin{split}
    f(x)=f(x(r, \te))&=\sum_{l=0}^N\sum_{m=-l}^{l}f_{l,m} Y_{l}^{m}(r,\te)= \sum_{l=0}^N\sum_{m=-l}^{l} f_{l,m}  N_{l m} P_{l}^{ m}(\cos (r)) \e^{\im m \te}
     \end{split}
\end{equation*}
is uniformly bounded by $1$ over the whole $\St$, i.e. for   $(r,\te )\in [0,\pi)\times [0,2\pi ]$.
On the grounds that there is a deeply developed theory for multi-variate trigonometric polynomials providing numerous condition on their boundedness on a frequency  domain, see e.g. \cite{Dumitrescu2007}, it would be clearly better to represent $f$ as a purely trigonometric expressions. 
Since each associated Legendre polynomial $P_{l}^{ m}(\cos (r))$ can be uniquely represented in as $P_{l}^{ m}(\cos (r))=\sum_{k=-l}^{l} p_{lm k}\e^{-\im k  r }$, we easily obtain  
\begin{equation*}
\begin{split}
   f(r, \te)=\sum_{l=0}^N\sum_{m=-l}^{l} \sum_{k=-l}^{l}  f_{l,m}  N_{l m}  p_{l m k}\e^{-\im k  r }\e^{\im m \te}
     = \sum_{m=-N}^N\sum_{k=-N}^{N} \tilde{f}_{m k}\e^{-\im k  r }\e^{\im m \te},
 \end{split}
 \end{equation*}
where the coefficients $\title{f}_{m k}$, see \cite {Bendory2015b}, are defined as  
 \begin{equation*}
    \tilde{f}_{m k}= \sum_{l=0}^N  \tilde{f}_{m k l}, \quad \text{with} \quad  \tilde{f}_{l m k}=  \left\{ \begin{matrix}  f_{l,m}  N_{l m}  \title{p}_{l m k}, & m,k\in [-l, l],\\
     0, &else. 
     \end{matrix}\right.
 \end{equation*}
Now, to replace the norm constraint in \eqref{dual_problem_con} by a finite dimensional conditions we use so-called \textit{the Bounded  Real Lemma}. To do this, let us first throw light on  the notion of the half-space $\HH$. S set $\HH $ is called a half-space of $\Z^2$  if $ \HH \cap (-\HH)=\{0\}$, and $\HH\cup (-\HH)= \Z^2$ and $\HH+ \HH\subset \HH$. A standard way to construct half-space is given iteratively. We start with $\HH_1=\N$ and we say that $\bb k \in \HH $ if either $k_2>0$ or $k_2=0$ and $k_1\in \N $, of course, such iterative representation is very useful for numerical purposes.  With this preparation, a particular version of the Bounded  Real reads as following. 
\begin{lem}\label{BRl}
 \cite{Dumitrescu2007} 
 Let $h(r, \te)$ be a positive orthant polynomial defined  by
\begin{equation*}
h(r, \te)=\sum_{m=0}^n \sum_{k=0}^n h_{m,k} \e^{\im k  r }\e^{\im m \te}, 
\end{equation*}
where $ {(r,\te)}\in \mathcal{D}= [-\pi, \pi]^2$. Then 
providing that $|H (r,\te)|\le 1 $ on the frequency domain $ \mathcal{D}$, there exist a positive semi-definite matrix $Q_0 \in \C^{(2 N+1)^2\times(2 N+1)^2}$ and s half-space $\HH \subset \Z^2$ such that 
\begin{equation}\label{finite_dim_const}
    \begin{bmatrix} 
  Q & h\\
  h^{\mathrm{H}}& 1
    \end{bmatrix}\succcurlyeq 0,  
     \quad \tr(\Omega_{\bb k} Q)= \delta_{\bb k}, \quad  \bb k \in \HH, \quad  -n \le \bb k \le n,
\end{equation}
where ${h\in \C^{(2N+1)^2}}$ is a column stacked  vector of coefficients~$h_{m,k}$.
\end{lem}

 Returning to \eqref{dual_problem_con}, we see that  the domain $(r,\te )\in [0,\pi]\times [0,2\pi )$ of  the generalized polynomial $f$ is slightly shifted version of the frequency domain $\mathcal{D}$ in Lemma~\ref{BRl}, moreover $f$ is not a positive orthant polynomial. So to reach the desired frequency domain and to have the means to apply the Bounded  Real Lemma~\ref{BRl} we can consider, for example, a trigonometric polynomial 
 \begin{equation*}
     h(r,\te)= \e^{N r+ N\te}f(r+ \pi ,\te +\pi )= \sum_{m=0}^{2N}\sum_{k=0}^{2N} (-1)^{(m+k)}\tilde{f}_{m-N, k-N} \e^{\im k  r }\e^{\im m \te}.  
 \end{equation*}
The polynomial $h(r,\te)$  is
positive orthant,  has the same magnitude as $f$ does, and is defined on $(r,\te)\in [-\pi,0]\times [-\pi,\pi )\subset \mathcal{D}$. Hence, the dual problem \eqref{dual_problem_con} is equivalent to 
 
\begin{equation}\label{finite_dual_problem}
\tag{fdRP}
    \max\limits_{f\in\Pi_N(\St), \, Q} \mathrm{Re}\lr{f,\pP_N^*\mu^\star} \quad\mbox{ subject to }\quad \eqref{finite_dim_const}, 
\end{equation}
with $h_{m,k}=  (-1)^{(m+k)}\tilde{f}_{m-N, k-N}$, for $m,k=0,\dots 2N.$

By strong duality, if $\mu^\star$ is a solution of the primal problem~\eqref{eq:1:3} and $f^\star$  is any solution of the pre-dual problem, then  it follows 
\begin{equation*}
    \mathrm{Re}\lr{f^\star,\pP_N^*\mu^\star}= \mathrm{Re}\lr{\pP_N f^\star, \mu^\star}= \|\mu^\star\|_{\mathrm{TV}}.
\end{equation*}
In the case $\mu^\star$ is a discrete measure, this implies that the generalized polynomial is exactly equal to the sign  of $\mu^\star$, when $\mu^\star$ is not vanishing, namely
\begin{equation*}
    f^\star(x_i)= \mathrm{sign}_{\mu^\star}(x_i), \quad x_i \in \supp{(\mu^\star)}.
\end{equation*}
 which in turn means that the supporting points form a subset of the zeros of the polynomial $1-|f^\star(x)|\in \Pi_{2N}(\St)$. 
 
Summing up, we obtain the following computational algorithm.

\begin{algorithm}[H]
\KwData{low frequency information $\pP_N^* \mu^\star$ of a measure  ${\mu^\star\in \M(\R, \St)}$}
\Begin{
 Solve for $f^\star\in \Pi_N(\St)$,  $Q\in \C^{(2N+1)^2\times (2N+1)^2}$
 \begin{equation*}\label{dual_con_problem}
 \begin{split}
    \max\limits_{f\in \Pi_{N}, Q} \mathrm{Re}\lr{c,\pP_N^*\mu^\star} \quad\mbox{ s.t. } \quad & \begin{bmatrix} 
  Q & h\\
  h^{\mathrm{H}}& 1
    \end{bmatrix}\succcurlyeq 0, \quad \tr(\Omega_{\bb k} Q)= \delta_{\bb k}, \\
    & \quad   \quad   \quad \bb k \in \HH,   \quad  -n \le \bb k \le n 
     \end{split}
\end{equation*}
using  interior point method; 

\vspace{2mm}

Randomly choose $P$ initial points $\{x_1,\dots, x_p \}\subset \St$

\vspace{2mm}

\For{$k= 1, \dots, P $}{
Find $x_k^*$ via Conjugate Gradient  with initial point $x_l$;

\vspace{2mm}

\If{$(1-|q(x_k^*)|\le tol)$}{
 $\X= \X\cup \{x_k^*\}$}
 }
 
 \vspace{2mm}
 
 Set $\nu= \sum_{x_i \in \X} c_i \delta_{x_i}$ with $c_i$ such that 
 \begin{equation*}
    \nu= \mathrm{argmin} \|\pP_N^* \nu - y\|_2
 \end{equation*}
}
\KwResult{$\nu \in  \M(\R, \St)$}
 \caption{Super-resolution on $\St$ via SDP}
\end{algorithm}

\subsection{Discretization of the Optimization Problem}

To avoid the hing complexity of the the semi-definite program for the higher order of moments we propose to use a discretization  of the primal problem. In this section we discuss the convergence behavior of the solution of the discretized  problem. 

Let us choose a sequence of discrete sets $\G_n\subset \St$, then the  filling distance of $\G_n $ is defined as
\begin{equation*}
    \rho(\G_n)= \sup\limits_{x\in \St } \inf\limits_{y\in \G_n } d(x,y). 
\end{equation*}
For the chosen sets $\G_n$, we consider the discretized version of the primal problem \eqref{eq:1:3}, namely
\begin{equation}\label{eq_disc}
 \tag{RP${}_{n}$}
    \min\limits_{\supp{(\mu)}\subset \G_n } \|\mu\|_{\mathrm{TV}},\quad\mbox{ subject to }\quad \pP_N^*\mu=\pP_N^*\mu^\star,
\end{equation}
To discuss the convergence behavior of the solution of the discretized  problem~\eqref{eq_disc}, we follow ideas from \cite{Tang2013}, where the convergence for continuously parameterized dictionaries has been discussed. 

\begin{thm}
Let assume that the measure $\mu^\star= \sum_{ x_i\in X} c_i \delta_{x_i}$  is the unique solution of the primal problem \eqref{eq:1:3} and sequence  $\G_n$ of discretizations is chosen in such way  that 
\begin{equation*}
    \rho(\G_n) \to 0, 
\end{equation*}
when $n\to \infty$. Then each sequence of solutions $\mu_n$ of  \eqref{eq_disc} converges to $\mu^\star$ in the weak$^* $-topology. 
Moreover, there exists $\varepsilon> 0$, such that 
\begin{equation*}
    \mu_n(B_{\varepsilon}(x_i))\to c_i, \quad |\mu_n|(B_{\varepsilon}(x_i))\to c_i, 
\end{equation*}
and 
\begin{equation}\label{convergence_discrite}
    |\mu_n|\Big((\cup_ i B_{\varepsilon}(x_i))^c\Big)\to 0.  
\end{equation}
\end{thm}
\begin{proof}
Following the same line as in \cite{Tang2013}, we show that each sequence of solutions is bounded and thus, due to the sequentially Banach-Alaoglu Theorem, admits a weak$^*$ convergent subsequence that  converges to as solution of the continuous problem \eqref{eq:1:3}. 
For the boundedness, let us consider the discretized  convex pre-dual problem to \eqref{eq_disc}
\begin{equation}\label{dual_problem_con_dis_dis}
\tag{dRP$_n$}
    \max\limits_{f\in\Pi_N(\St)} \mathrm{Re}\lr{f,\pP_N^*\mu^\star} \quad\mbox{ subject to }\quad \nofty{f} \le 1, \, x\in \G_n
\end{equation}
We show that the feasible sets of the dual problem, i.e. the sets of function $f\in \Pi_N(\St)$ such that $|f(x)|\le 1$ for $x\in G_n$, are bounded and therefore compact. By assumption  $ \rho(\G_n) \to 0$, therefore we have the for large enough $n$, 
\begin{equation*}
    \rho(\G_n)\le \frac{1}{N}.
\end{equation*}
Applying the Marcinkiewicz-Zygmund inequality  \cite{KeinerKunis2007, Frank2008}, yields for all $f\in \Pi_N(\St)$
\begin{equation*}
    \nofty{f}\le (1-N \rho(\G_n))^{-1} \max\limits_{x\in \G_n} |f(x)|,
\end{equation*}
meaning that all feasible sets are uniformly bounded and, thus, compact. This shows that the discretization problem has a solution, and  we denote these minimizers by $f_n$. 

The rest of the proof is identical to the proof of Theorem~2 in \cite{Tang2013}. Let us briefly sketch it. Due to the uniform boundedness, it can be shown that the sequence $f_n$ of solutions of the discretized dual problem converges to the solution $f^\star$ of the continuous dual problem. Since strong duality holds for both the discretized problems and  the continuous one, we have \begin{equation*}
    \mathrm{Re}\lr{f_n,\pP_N^*\mu^\star}= \|\mu_n\|_{\mathrm{TV}} \quad \text{and } \quad  \mathrm{Re}\lr{f^\star,\pP_N^*\mu^\star}= \|\mu^\star\|_{\mathrm{TV}},
\end{equation*}
which shows that the sequence $\mu_n$ is bounded. Then due to the Banach-Alaoglu Theorem we get the weak$^*$ convergence to the minimizer $\mu^\star$.  The convergence of the measure of the epsilon balls follows analogically to  Corollary 1 in \cite{Tang2013}. 
\end{proof}

Additionally, as it  has been mentioned in \cite{Tang2013}, for fine enough discretization the property  \eqref{convergence_discrite} suggests that the support of the solutions $\mu_n$ of the discrete minimization problems cluster around the support of $\mu^\star$. 

Now let us proceed with the exact  digitization procedure via considering the following grid \cite{Rockmore1996} in the spherical coordinates for inclination $[0, \pi)$ and azimuth $[0, 2\pi)$
\begin{equation*}
    r_k= \frac{\pi(2j+1)}{4n}, \quad \quad \theta_k= \frac{2\pi k}{2n},\quad  \quad j,k= 0,\dots, 2n-1,  
\end{equation*}
for some $ n\in \N$. These points generate a grid $\G_n$ of $4n^2$ points. 

Let the matrix ${\bf Y}_N$ of spherical harmonics from $ \Pi_N (\St)$ evaluated at the grid points, 
\begin{equation}\label{matrix_of_harm}
{\bf Y_N} = \big( Y_{m}^{\ell}(x(r_k, \theta_k)\big)_{\ell=0,  |m| \le \ell}^{N} \quad \in \C^{\dim( \Pi_N) \times 4n^2}. 
\end{equation}
Using such a matrix notation, the problem \eqref{eq_disc} can be transformed to the following finite dimensional basis pursuit problem
\begin{equation}\label{basis_pursuit}
 \tag{dicRP${}_{n}$}
    \min\limits_{c\in \C^{4n^2}} \|c\|_{1},\quad\mbox{ subject to }\quad {\bf Y_N}\,  c = y,
\end{equation}
where $y= \big( \lr{ Y_{m}^{\ell} , \mu^* } \big )_{\ell=0,  |m| \le \ell}^{N}$ is the given data, and $c\in \C^{4n^2}$ represents the vector of coefficients of the spherical harmonics. Since we measure is supported only on a few points of the grid $\G_n$, we need to impose sparsity of $c$, which can be done by minimizing the $\ell_1$-norm of $c$. 

Due to the well-known basis mismatch phenomenon, a completely sparse solution of \eqref{basis_pursuit} can not be obtained and one needs a certain threshold in absolute value keeping only entries that are large enough, i.e. one keeps only points with solution $c^\star\in \G_{n}$ such that $|c^\star|> thresh$ . After that one can see that recover grid points cluster around the support of $\mu^\star$. To find the centers of such cluster, and correspondingly the support of $\mu^\star$ we use an algorithm known as the bivariate {\it kernel density estimator}, for more details see \cite{Comaniciu2002}, with the normal kernel $\ds K_N({\bf x})= (2\pi)^{-1/2}\exp{\left(-\frac{1}{2}\|{\bf x}/h\|^2_2\right)}$ and a scaling parameter $h>0$. 
Combining all above described steps results in our second algorithm based on the discretized optimization problem  \eqref{eq_disc}.

\begin{algorithm}[H]
\KwData{low frequency information $y= \big( \lr{ Y_{m}^{\ell} , \mu^* } \big )_{\ell=0,  |m| \le \ell}^{N}$  of the measure  ${\mu^\star\in \M(\R, \St)}$ 

\vspace{2mm}

{\bf Parameters:} grid-parameter $n\in \mathbb{N}$,  threshold $thresh>0$, scaling parameter $h$, $tol > 0$
}

\vspace{2mm}

\Begin{
 Solve for $c\in \C^{4n^2}$
\begin{equation}
 \tag{dicRP${}_{n}$}
    \min\limits_{c\in \C^{4n^2}} \|c\|_{1},\quad\mbox{ subject to }\quad {\bf Y_N}\,  c = y
\end{equation}
\vspace{2mm}

Choose $x_k\in G_n$ such that 
\begin{equation*}
    |c_k|>thresh. 
\end{equation*}

Set $\X=\{x_k\}_k$. 
\vspace{2mm}

Apply the kernel density estimation algorithm with the  scaling parameter $h$ and the tolerance $tol$ to the set $\X = \{x_k\}_k$ and generate $\X_{mean}=\{x_i^{mean}\}_i$

\vspace{2mm}

 Set $\nu= \sum_{x_i \in \X} c_i \delta_{x_i}$ with $c_i\in \R$ such that 
 \begin{equation*}
    \nu= \mathrm{argmin} \|\pP_N^* \nu - y\|_2
 \end{equation*}
}
\KwResult{$\nu \in  \M(\R, \St)$}
 \caption{Super-resolution on $\St$ via Discretization}
\end{algorithm}

\subsection{Numerical Experiments}

\subsubsection{Semi-definite Program}

~\
 
\textbf{Experiment 1} [Noise-free recovery].
The first experiment demonstrates the performance of Algorithm 1 for the signal reconstruction in the noiseless data case. 
A discrete measure 
\begin{equation*}
    \mu^\star= \sum_{i=1}^6 c_i \delta_{x_i}
\end{equation*}
with randomly generated support points $x_i=x( r_i,\theta_i)$ and amplitudes $c_i$, given in the Table~\ref{Tab:Tcr} is considered.  
The support  fulfills the minimal separation distance condition
\begin{equation*}
    \min\limits_{i\ne j} d(x_i, x_j)\ge0.549335, 
\end{equation*}
 and the given information about $\mu^\star$ is its low frequency information up to degree $N=6$.

\begin{table}[ht]

\centering
\begin{tabular}{| c || c| c |c| }
 \hline
\quad $i$  \quad \, & \quad  \quad  \quad $\theta_i$  \quad  \quad  \quad & \quad  \quad  \quad $r_i$ \quad  \quad  \quad &  \quad  \quad  \quad$c_i$  \quad  \quad  \quad\\ %
 \hline
1& 1.366427 &  0.412278  & 4.296273\\ 
 \hline
 2& 1.983298 &  2.591331  &  $-1.594284$\\  
 \hline
3& 2.589166  & 4.898989 &   $-1.058452$ \\
 \hline
4& 0.630283 &  3.460063  &  2.005496\\
 \hline
5& 1.294025  & 4.299585 & 4.071419\\
 \hline
6&  3.016381  & 3.455708 & 3.196665 \\
 \hline
\end{tabular}

\vspace{2mm}

\caption{The support points and amplitudes of the test measure  $\mu^\star$. }
 \label{Tab:Tcr}
\end{table}


Having this data at hand, we first compute the solution $f^*$ of the convex-optimization problem \eqref{dual_problem_con} using CVX-\texttt{Matlab} tools. Then the measure support is captured by looking for zeros of the function $1-|f^*|^2$. For this purpose,  the \texttt{Matlab} built-in solver \texttt{fminunc} is applied  with $P=20000$ initial  randomly generated points on the sphere.  We set the tolerance to ${tol=10^{-8}}$,  and identify the resulting  minima $x_i^{rec}$ that fulfill $1-|f^*|^2< tol$. To ensure that there are no several points clustered near a true support point, we finally apply the kernel density estimator procedure described in the previous section. 
The recovery error is measured respectively for the measure support and intensities by the quantities  $\epsilon_{x}$ and $\epsilon_{c}$ defined as 
\begin{equation}
\epsilon_{x}=\max\limits_i \min\limits_j d(x_i,x_j^{rec}) \quad  \quad \text{and} \quad  \quad 
\epsilon_{c}={\max\limits_{i}|c_i-c_i^{rec}|}.
\end{equation}

In such way,  after running  Algorithm 1 we get six recovered support points and coefficients  $\{x_j^{rec}\}_{j=1}^6$ and  $\{c_i^{rec}\}_{i=1}^6$  such that 
$$
\epsilon_{x} <  2.5895457 \cdot 10^{-7},   \quad  \quad \quad  \quad \epsilon_{c}  < 3.4205068\cdot 10^{-7}
$$
The results of the experiment are illustrated in Figure \ref{fig:num:1}. As can be seen the test measure $\mu^{\star} $ is successfully recovered in the noise-free case via Algorithm~1.

  \begin{figure}[ht!]
  \hspace*{\fill}%
  
  \subcaptionbox{ Support of $\mu^\star$. \label{fig12:a}}{\includegraphics[width=2.3in]{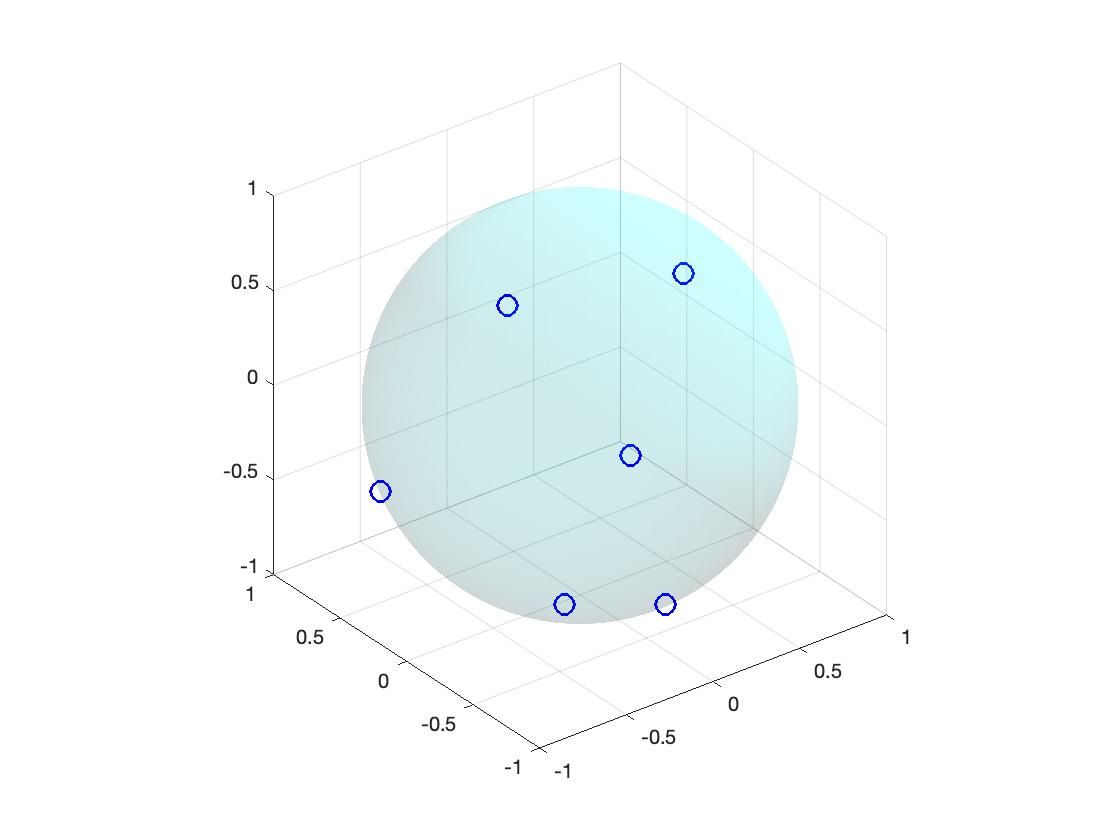}}
  \hfill%
  \subcaptionbox{  Low-resolution information in longitude-latitude.  \label{fig12:b}}{\includegraphics[width=2.3in]{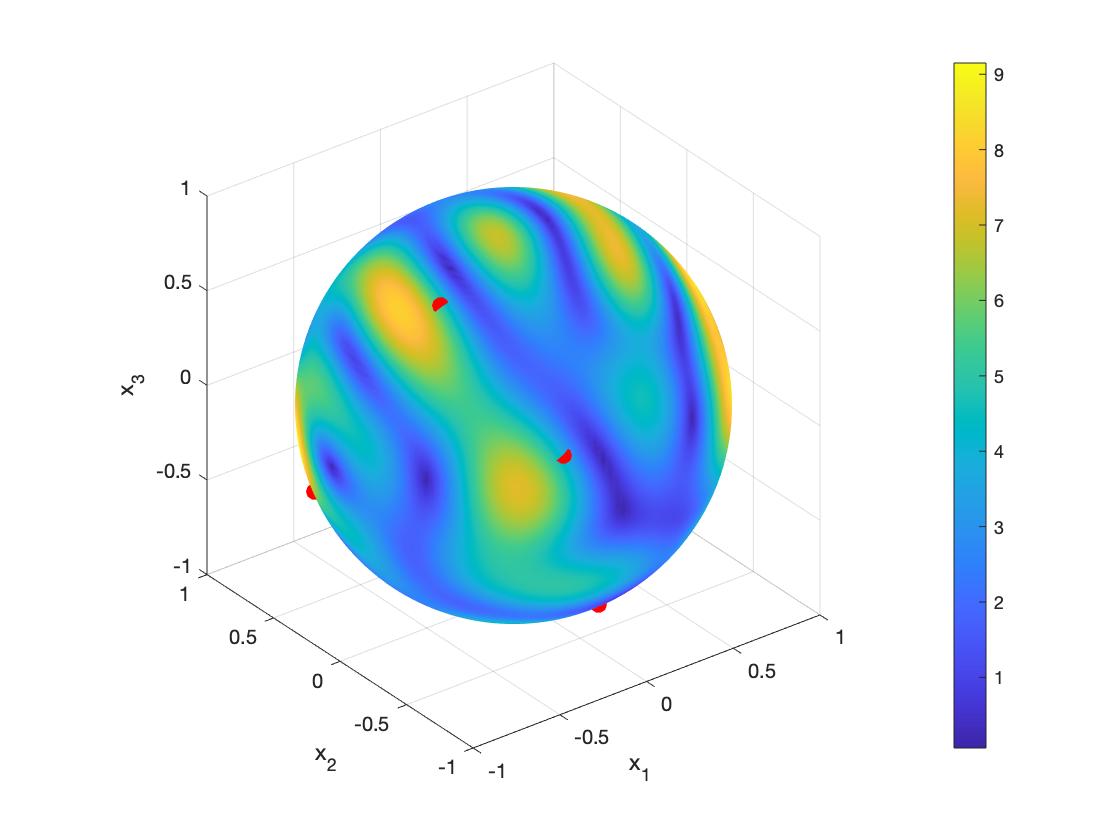}}%
  
   \bigskip
   
   \subcaptionbox{ Low-resolution information in longitude-latitude. \label{fig21:a}}{\includegraphics[width=2.3in]{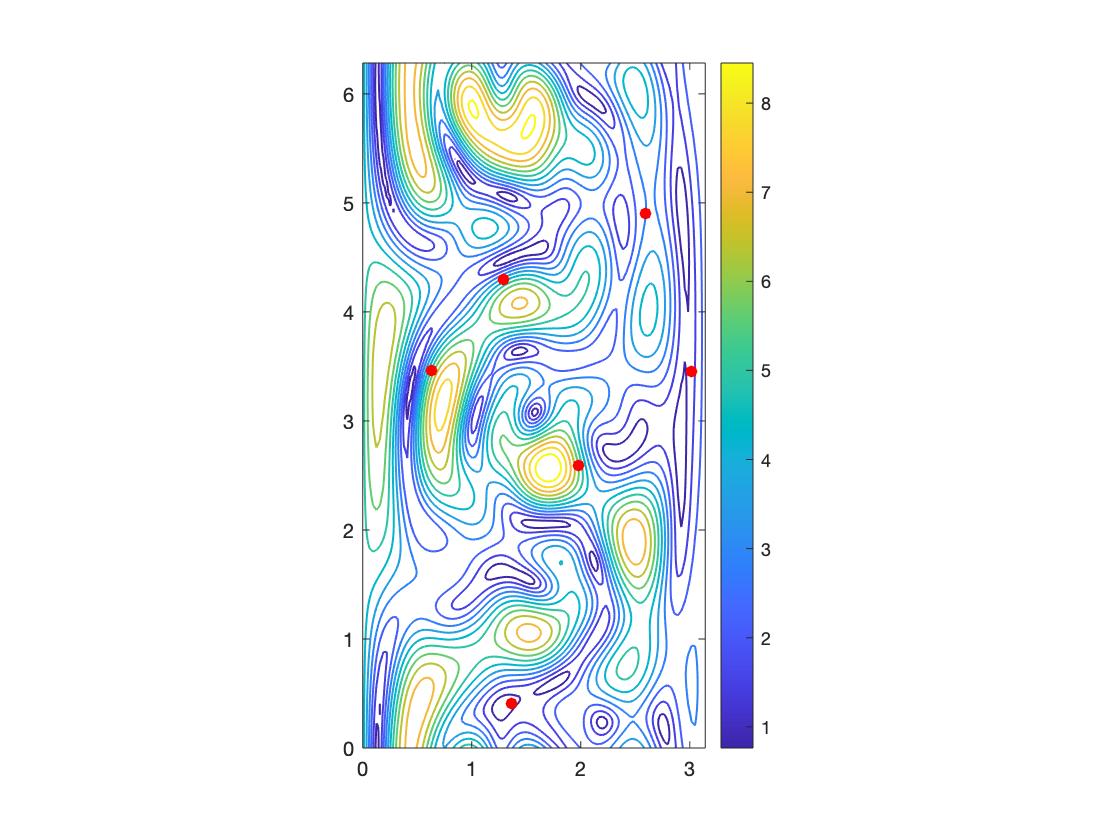}}
  \hfill%
  \subcaptionbox{  Solution $f^\star$ of \eqref{dual_problem_con} on $\St$.  \label{fig22:b}}{\includegraphics[width=2.3in]{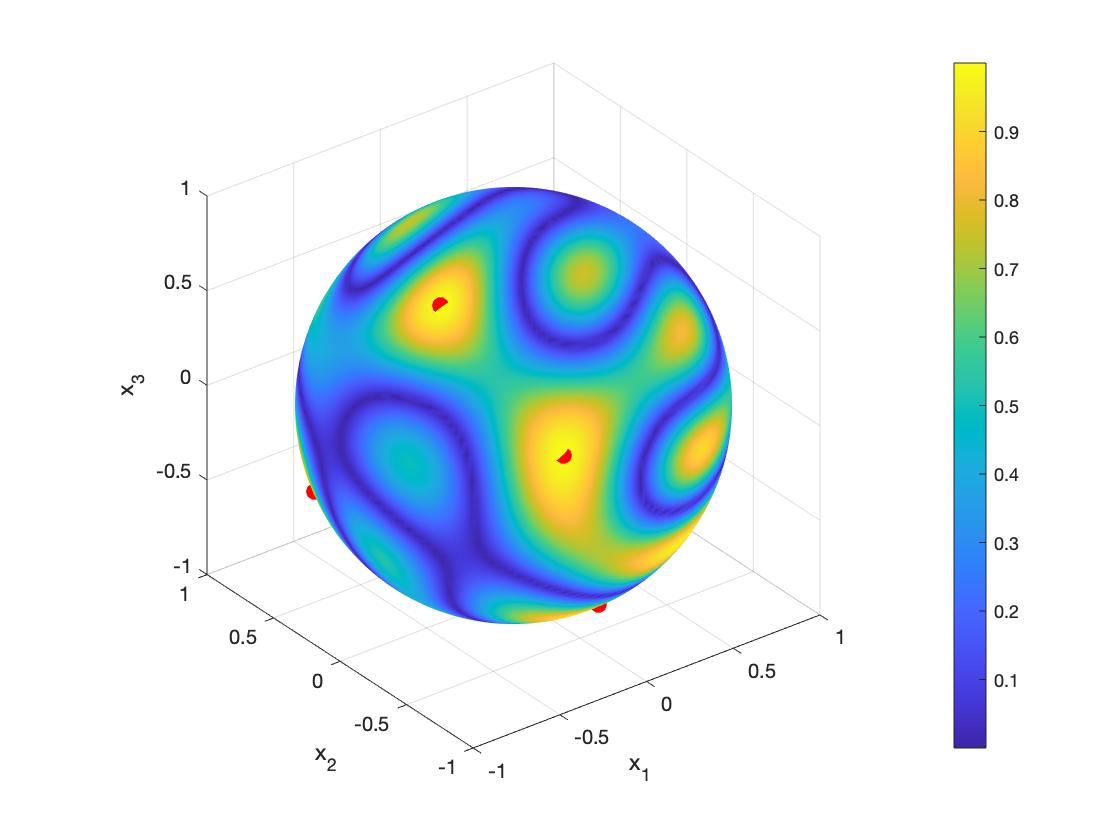}}%
   \bigskip
   
    \subcaptionbox{ Solution $f^\star$ of \eqref{dual_problem_con} in longitude-latitude.  \label{fig311:a}}{\includegraphics[width=2.3in]{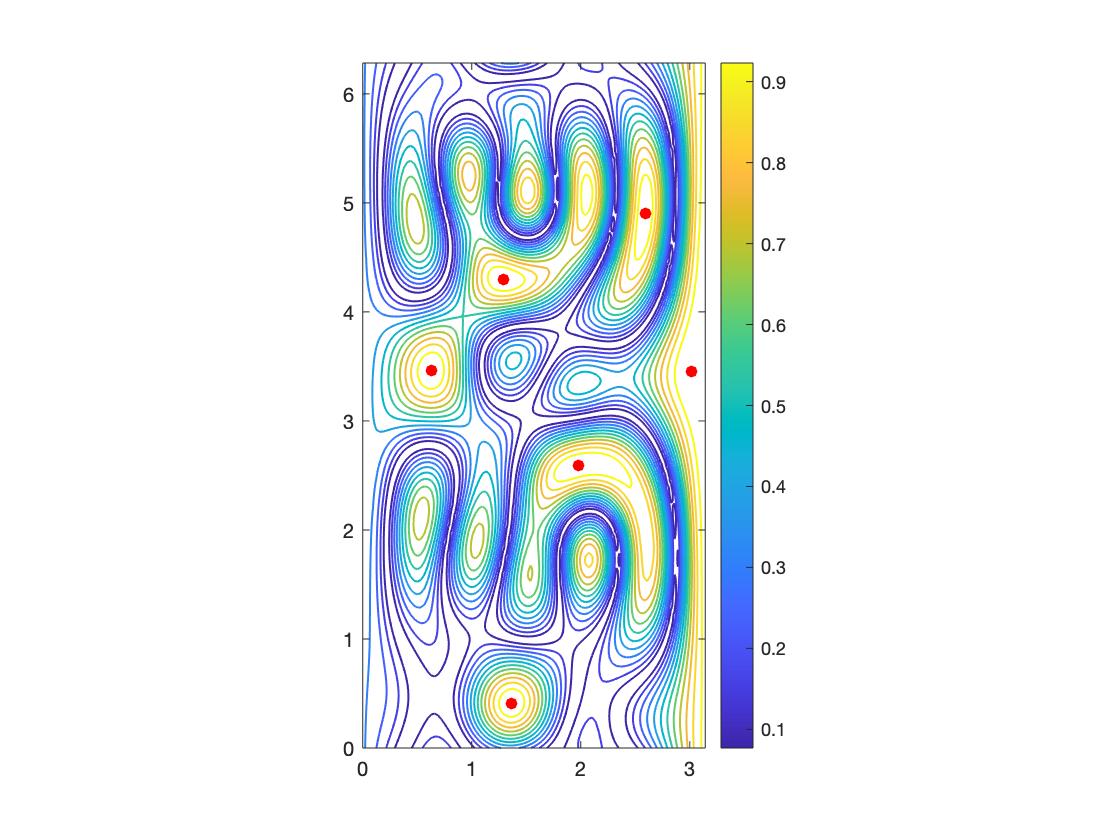}}
  \hfill%
  \subcaptionbox{  Recovered support of $\mu^\star$.  \label{fig321:b}}{\includegraphics[width=2.3in]{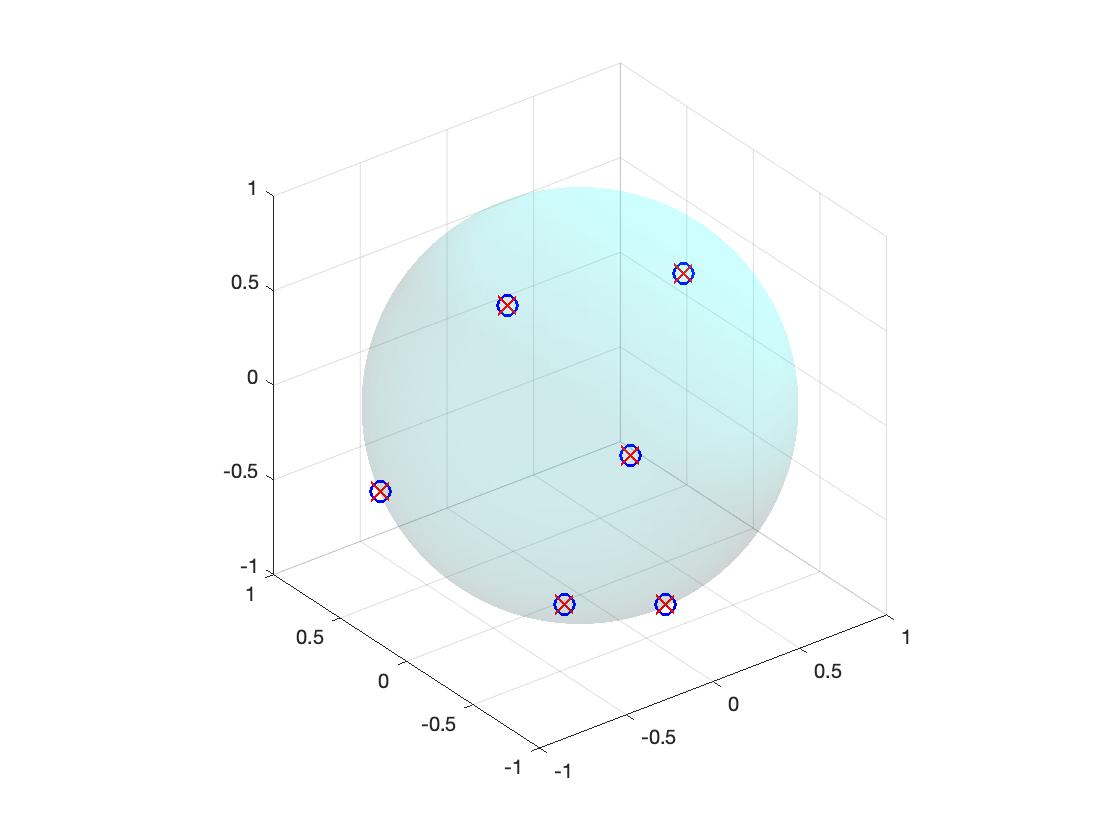}}%

  \caption{ \it  Reconstruction scenario of test measure $\mu^\star$ given in Tab.\,\ref{Tab:Tcr}. The true support  $x_i$ is marked with blue circles and the reconstructed support $x_i^{rec}$ is indicated by small red crosses. }
  
 \label{fig:num:1}
\end{figure}

\clearpage

\textbf{Experiment 2} [Super-resolution constant]. 
From the theoretical point of view, see Corollary \ref{solution_coeff_bounds_cor}, a sufficient criteria for measure recovery is that the support points $\{x_i\}$ satisfy the  separation condition  
$\min\limits_{i\ne j} d(x_i,x_j)\ge \frac{\nu}{N},$ for $\nu=19.2\pi$ and $N\ge 20$. Nevertheless, from numerical point of view the minimal separation distance may be much smaller.   To investigate this issue, we proceed as follows.  Defining for $i=1,\dots, 20$ the values $\gamma_i= i/10$, we generate twenty two-point sets $M_{ij}$ on the sphere that enjoy the separation distance with in the range $[\gamma_i-0.05,\gamma _i+0.05]$ for every $j=1,\dots, 20$.  Each of obtained sets $M_{ij}$ is considered as support of some measure $\mu_{ij}^\star$, while the amplitudes of $\mu_{ij}^\star$ are chosen randomly in the rage $[-1,1]\setminus\{0\}$. 
Then Algorithm~1 with $P=20000$ initial points is applied individually to each test measure $\mu_{ij}^\star$. 

\begin{figure}[ht!]
\centering
	\includegraphics[width=3.5in]{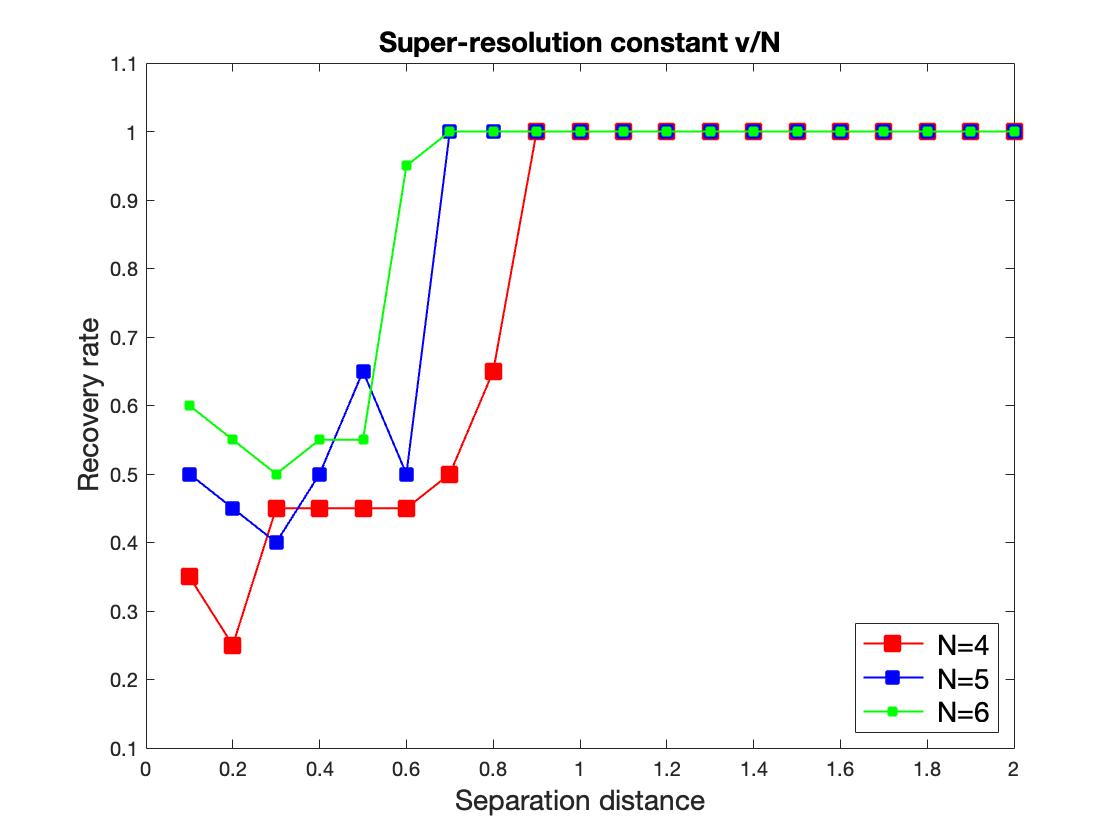} 
	\caption{\small Results of Experiment 2.  }
		\label{fig:SRcont}
\end{figure}

The operation of the algorithm is considered as successful, i.e. the support of the corresponding measure is considered as recovered, provided that the maximal recovery error fulfills 
$
   \epsilon_{x} < 1.056\cdot 10^{-4}. 
$
Thereafter, we counted the number of successful runs per separation distance $\gamma_i$.  A value $\gamma_i$ is considered as the numerical super-resolution constant once all algorithm runs were successful.
 As experiments show, see in Fig.~\ref{fig:SRcont}, the numerical super-resolution constant  $\nu_{{num}}$ is located within the interval~$[1.15 \pi, 1.33 \pi]$. 
 
~\

\textit{Remark}. When one of the measure support points has spherical coordinates $(0,0)$, some difficulties by locating zeros  of  $1-|f^*|^2$ have been observed. This might be avoided by using  some other solver instead of \texttt{fminunc}. 

~\

\textbf{Experiment 3} [Noisy data].
Assume the noise data scenario, i.e. the obtained low-resolution information of the measure $\mu^\star\in M(\St)$ is of the form 
\begin{equation}
y=\mathcal{P}_N^\star(\mu^\star+\varepsilon)
\end{equation}
with for some random (and independent of $k$ and $n$) noise term $\varepsilon$. Here we consider the deterministic noise model assuming that  $\|\mathcal{P}_N\varepsilon\|_{L_{2}}\le\delta$ for some $\delta\in \R$. 
  The optimization problem  \eqref{eq:1:3}, in this case, surely requires a regularization techniques in order to provide meaningful results. To this purpose, we consider the corresponding Thikonov-type problem 
\begin{equation}\label{eq:Ex3:TRPn}
\min_{\mu\in M(\S^2)} \frac{1}{2}\| \mathcal{P}_N^\star(\mu^\star+\varepsilon-\mu)\|_{L^{2}(\St)}+\tau\|\mu\|_{TV}
\end{equation}
and its semi-definite relaxation 
\begin{equation}\label{eq:Ex3:TRPn:relax}
\max\limits_{f\in\Pi_N(\S^2)} \mathrm{Re}\lr{f,\mathcal{P}_N^\star(\mu^\star+\varepsilon)}- \tau\|f\|_{L_{2}(\St)} \quad\text{ subject to } \; \|f\|_{\infty}\le1. 
\end{equation}
Following the same idea as in Subsection \ref{subsec:5:1}, it can be shown that  \eqref{eq:Ex3:TRPn:relax} can be represented as the next finite-dimensional optimization program
\begin{equation}\label{eq:Ex1:3}
\max\limits_{f\in\Pi_N(\S^2), \, Q} \mathrm{Re}\lr{f,\mathcal{P}_N^\star\mu^\star}- \tau\|f\|_{L_{2}(\St)} \quad\text{ subject to\   \eqref{finite_dim_const}} \tag{dRP$_{\tau}$}
\end{equation}
which we use in this experiment  instead of the ordinary optimization program in Algorithm 1. 

To showcase the reconstruction process in the noise corruption scenario we consider the setting of Experiment 1, but in this case, the data are perturbed by different levels of noise. $\delta\in \{10^{-j}\}_{j=1}^{3}$.  Figure \ref{fig:num:3}  and Table \ref{Tab:num:2} illustrate the results of the experiment.

\begin{table}
\centering
\begin{tabular}{| c | c| c |c | } \hline

{Noise level $\delta$ }                             & \quad  \quad  \quad$\epsilon_{x}$   \quad \quad \quad & \quad  \quad  \quad $\epsilon_{c}$ \quad  \quad  \quad& $ \quad  \tau  \quad$  \\ \cline{1-4}   

\multicolumn{4}{|c|}{$N=6$ } \\ \cline{1-4}  
$10^{-3}$ &  $0.003281$   &  $0.002101$   &  ${10^{-2}}$    \\  \cline{1-4}  

$10^{-2}$ &   0.038025  &   0.029545  &   $ 10^{-1}$   \\   \cline{1-4}  

$10^{-1}$&  0.927668  & 0.909458&   $10^{-1} $ \\  \cline{1-4}  

\multicolumn{4}{|c|}{$N=5$ } \\ \cline{1-4}  

$10^{-3}$&  0.014631  & 0.006691 & ${10^{-2}}$    \\  \cline{1-4}  
$10^{-2 }$& 0.180824   &   0.091935 &  $10^{-1} $  \\  \cline{1-4}  

$10^{-1}$ & 0.941574& 0.762983 &   $10^{-1} $ \\  \cline{1-4}  
 \hline
 
\multicolumn{4}{|c|}{$N=4$ } \\ \cline{1-4}  
  
$10^{-3}$& 0.019439&  0.020083 &  ${10^{-2}}$    \\ \cline{1-4}  

$10^{-2}$&  0.448509  & 0.331496  &${10^{-1}}$  \\ \cline{1-4}  

$10^{-1}$&  1.154241  & 0.517158 &   ${10^{-1}}$   \\ \cline{1-4}  
 \hline
\end{tabular}

\vspace{2mm}

\caption{Recovery error for the test measure  $\mu^\star$ via  the regularized optimization  problem \eqref{eq:Ex1:3} given the low-frequency information for  $N=4,5,6$. }
 \label{Tab:num:2}
\end{table}

 \begin{figure}[ht!]
 \centering
 
  \hspace*{\fill}%
  
  \subcaptionbox{Solution $f^\star$ of \eqref{dual_problem_con} on $\St$.   \label{fig31:a}}{\includegraphics[width=2.3in]{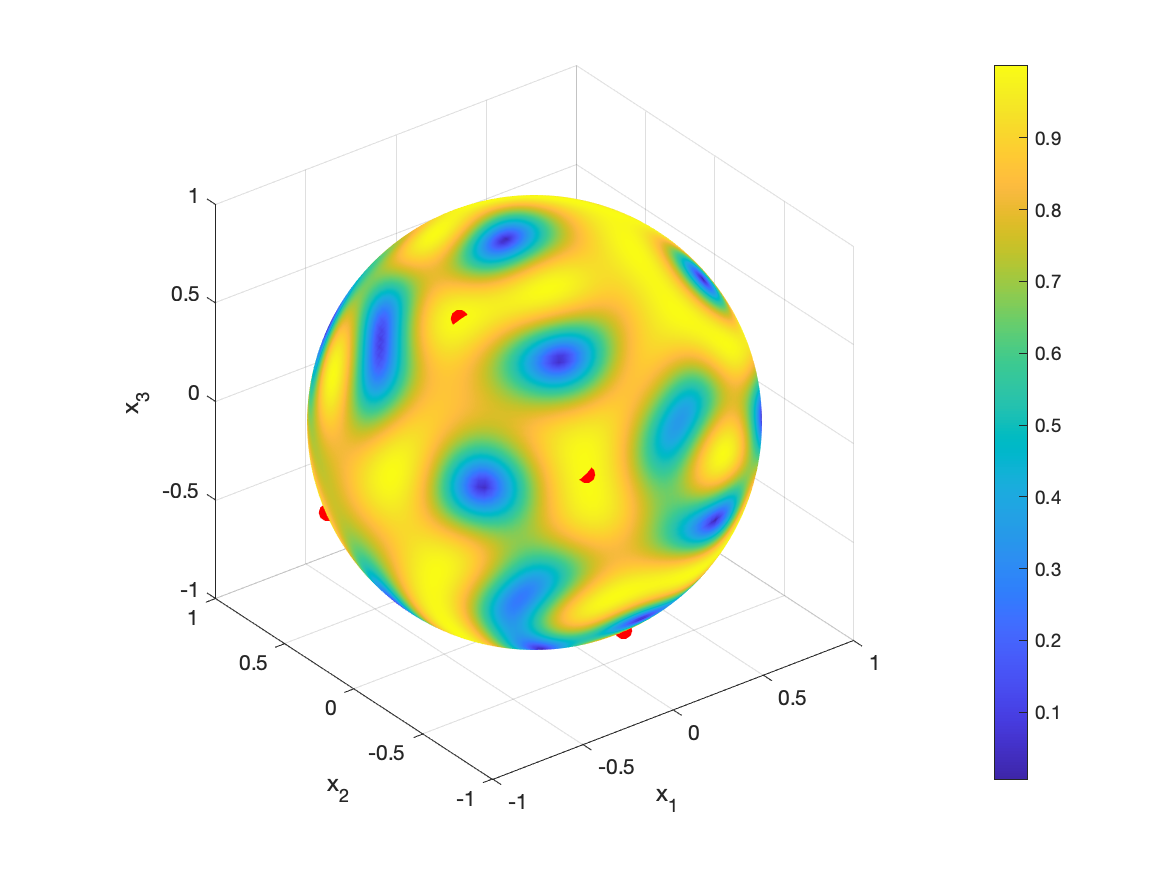}}
  \hfill%
  \subcaptionbox{  Support of $\mu^\star$ recovered via \eqref{dual_problem_con}.         \label{fig32:b}}{\includegraphics[width=2.3in]{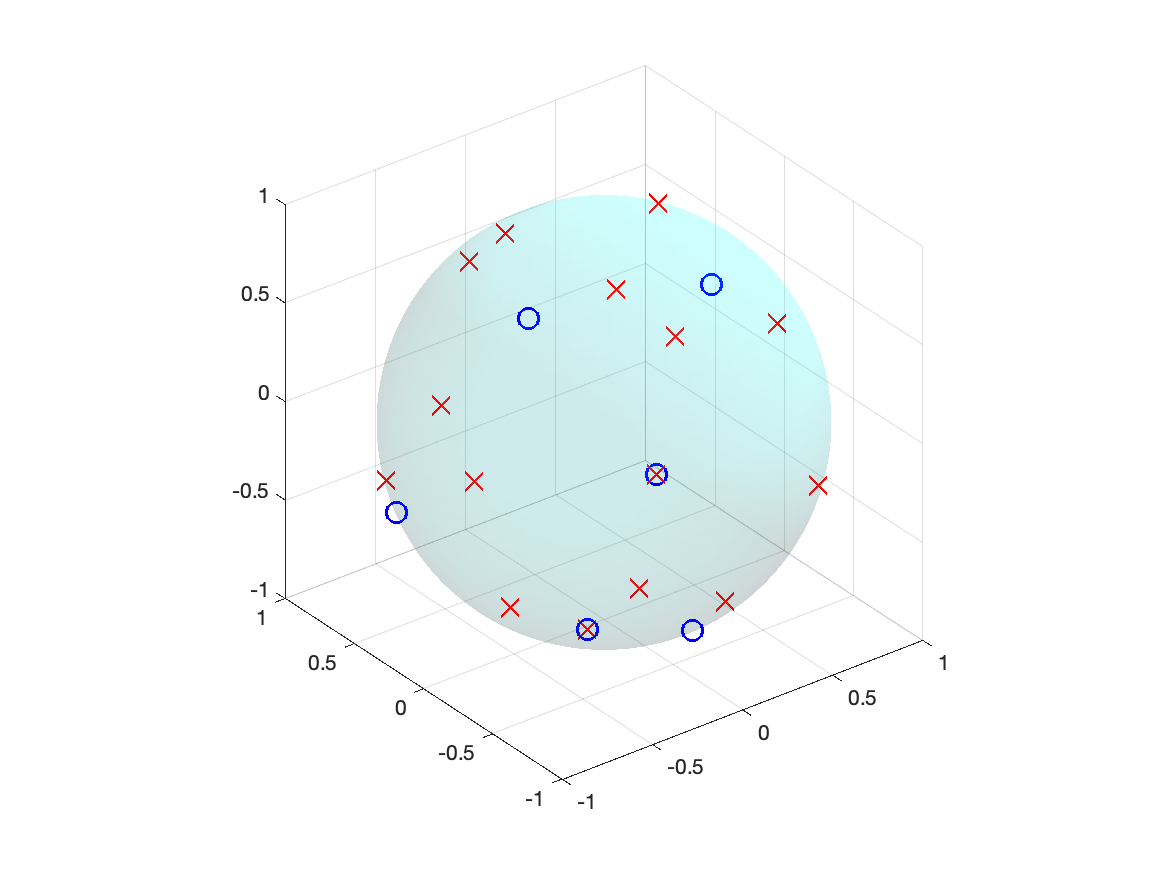}}%
 \centering
   
   \subcaptionbox{ Solution $f^\star$ of \eqref{eq:Ex1:3} on $\St$.    \label{fig33:a}}{\includegraphics[width=2.3in]{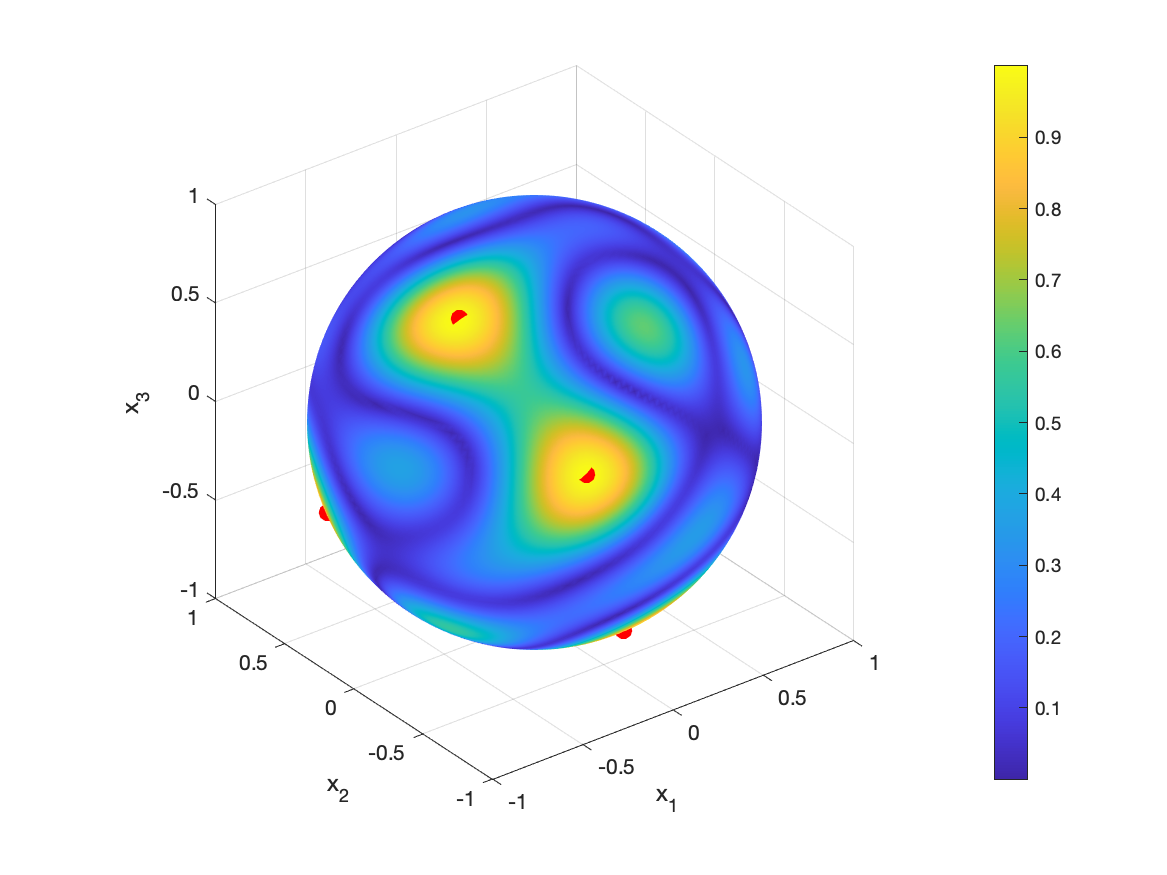}}
  \hfill%
  \subcaptionbox{ Support of $\mu^\star$ recovered via \eqref{eq:Ex1:3}.     \label{fig34:b}}{\includegraphics[width=2.3in]{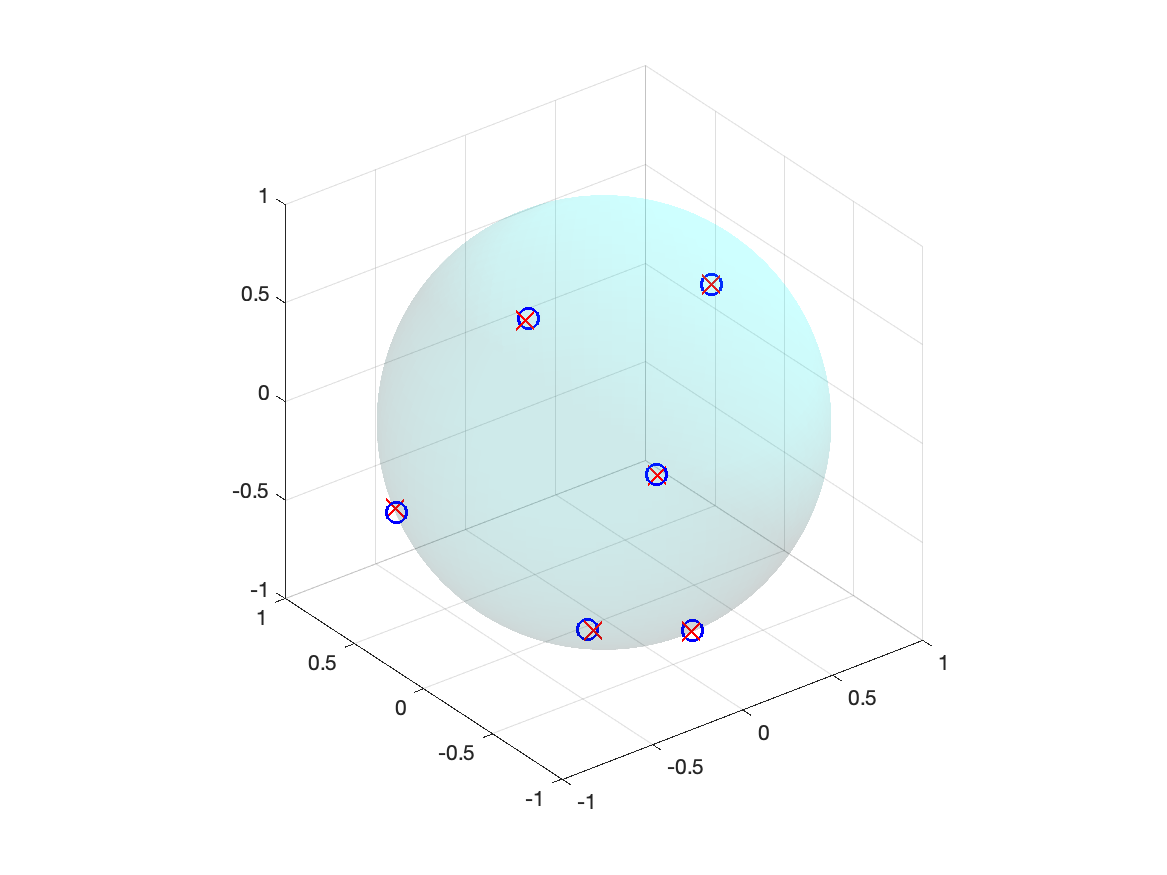}}%

  \caption{ \it Reconstruction of the test measure $\mu^\star$ given the low-resolution information up to $N=6$ corrupted by deremninistic noise of level $\delta=10^{-2}$.  The true support  $x_i$ is marked with blue circles and the reconstructed support $x_i^{rec}$ is indicated by small red crosses. }
  
 \label{fig:num:3}
\end{figure}

~\

\subsubsection{Discretization of Semi-Definite Program}

~\

\textbf{Experiment 1}{ [Grid size].} To showcase the performance  of Algorithm 2, 
here we consider the discrete measure $\mu^{\star}$ defined in Table~\ref{Tab:Tcr}  as a test measure once more, assuming  the low-frequency information of $\mu$ up to $N=6$
is given. 
Choosing different sizes $n\in \N$ of the grid  $\G_n$ and thresholding with $thresh=0.1$, Algorithm 2 is applied to recover the support of the mass. The results of such experiments are plotted in Fig. \ref{Fig:disc:rec} and the corresponding recovery error is presented in Table~\ref{Tab:discr:error}. 

As expected, with increasing grid size the point clusters become denser around the true support of $\mu^{\star}$. This, in turn, leads to a good support approximation after when applying the kernel density estimator procedure. And the recovery error decreases with the grid size $n$ grows.

\begin{table}[ht]

\vspace{15mm}

\centering
\begin{tabular}{| c|| c| c | }
 \hline
\quad Grid size $n$ \quad  &   \quad   \quad$ \quad \epsilon_{x}$ \quad   \quad  \quad  & \quad   \quad  \quad$\epsilon_{c}$  \quad  \quad \quad  \\ 
 \hline
20 & 0.568226 & 
 1.923598  \\ 
 \hline
40& 0.020896 & 0.018803\\
 
 \hline
 
80 & 0.008867 & 0.014393 \\ 

 \hline
\end{tabular}

\vspace{2mm}

\caption{Recovery error for different grid size.}
 \label{Tab:discr:error}
\end{table}

\begin{figure}[ht!]
  \hspace*{\fill}%
  
  \subcaptionbox{For $n=20$, support clusters. \label{fig102:a}}{\includegraphics[width=2.3in]{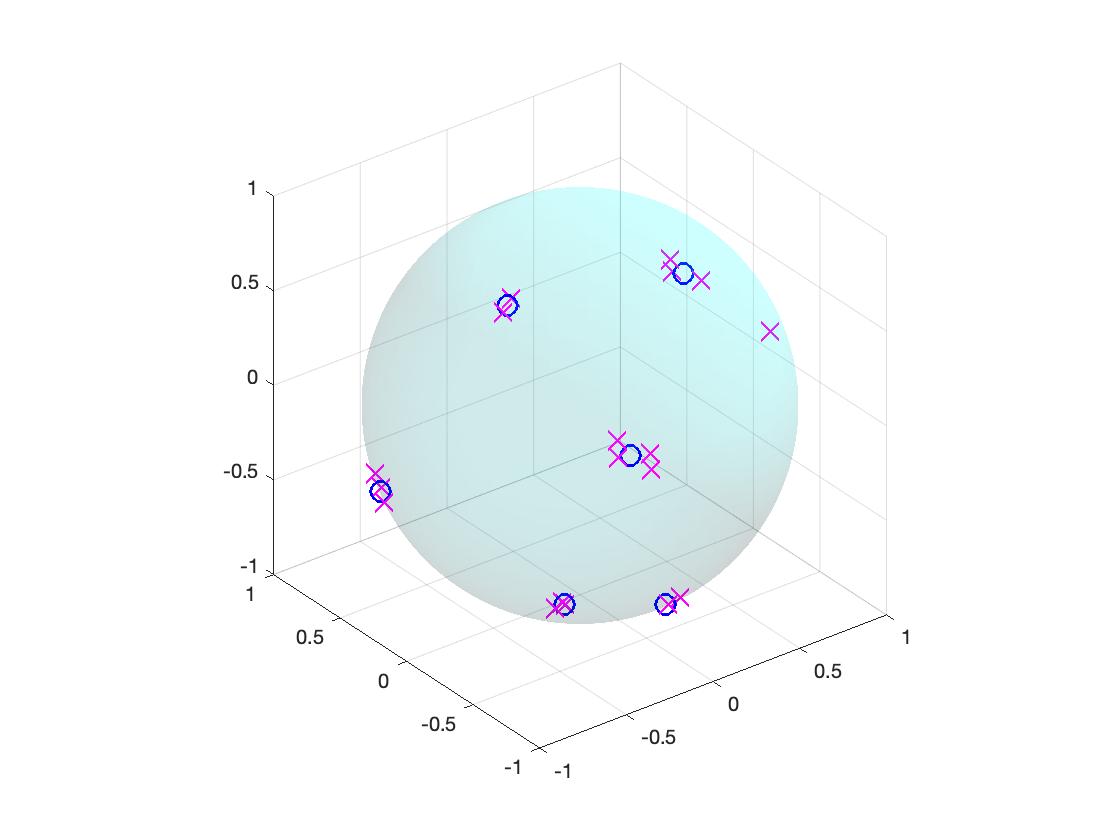}}
  \hfill%
  \subcaptionbox{ For $n=20$, recovered support. \label{fig112:b}}{\includegraphics[width=2.3in]{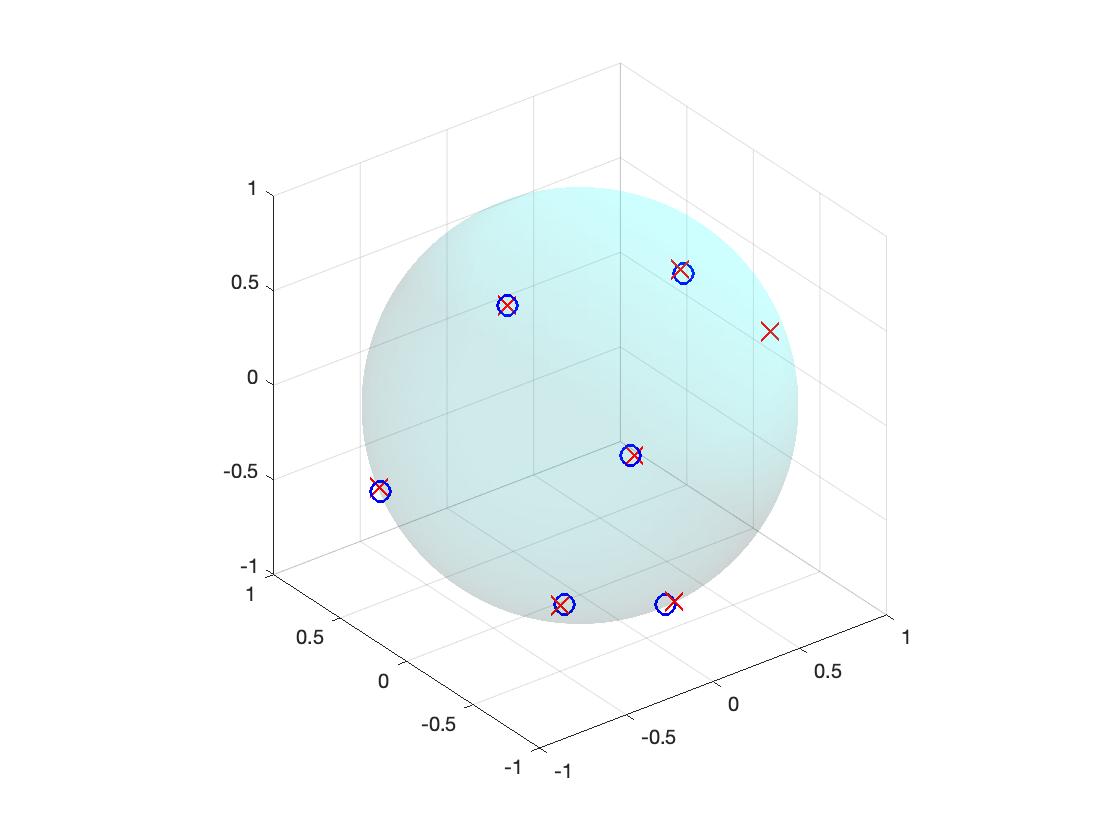}}%
   \bigskip
   
   \subcaptionbox{ For $n=40$, support clusters. \label{fig122:a}}{\includegraphics[width=2.3in]{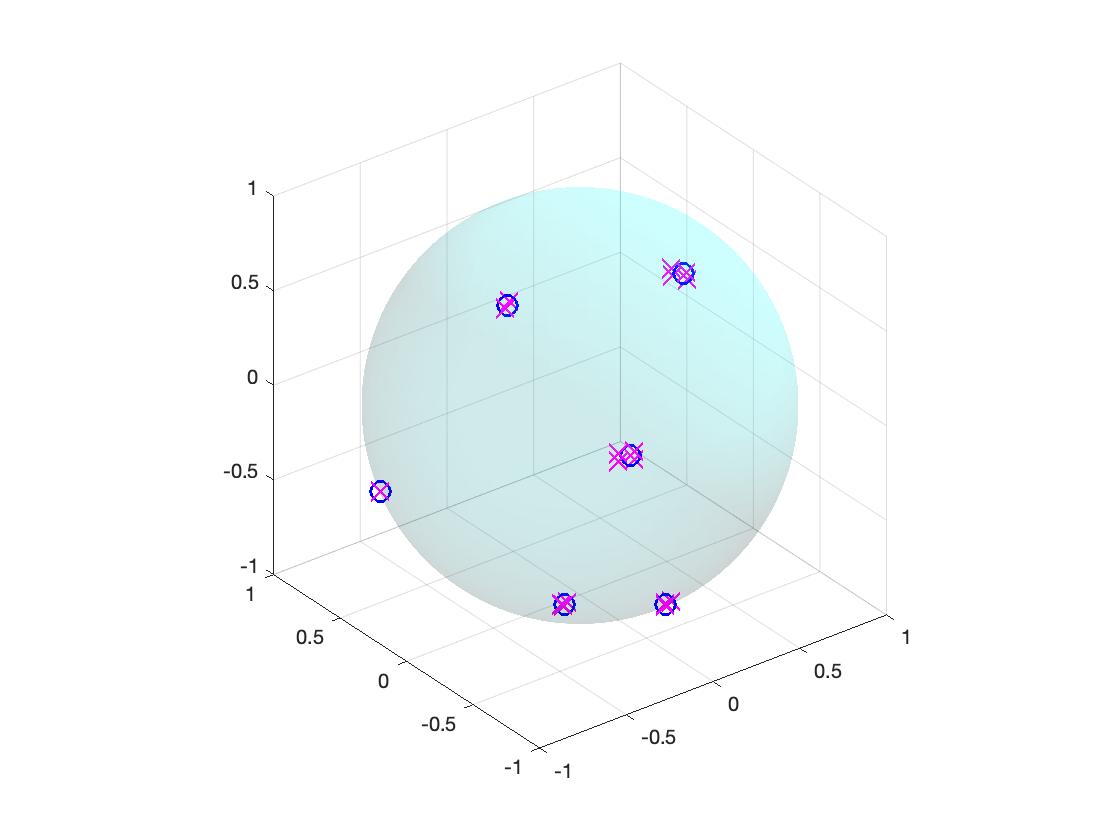}}
  \hfill%
  \subcaptionbox{  For $n=40$, recovered support. \label{fig132:b}}{\includegraphics[width=2.3in]{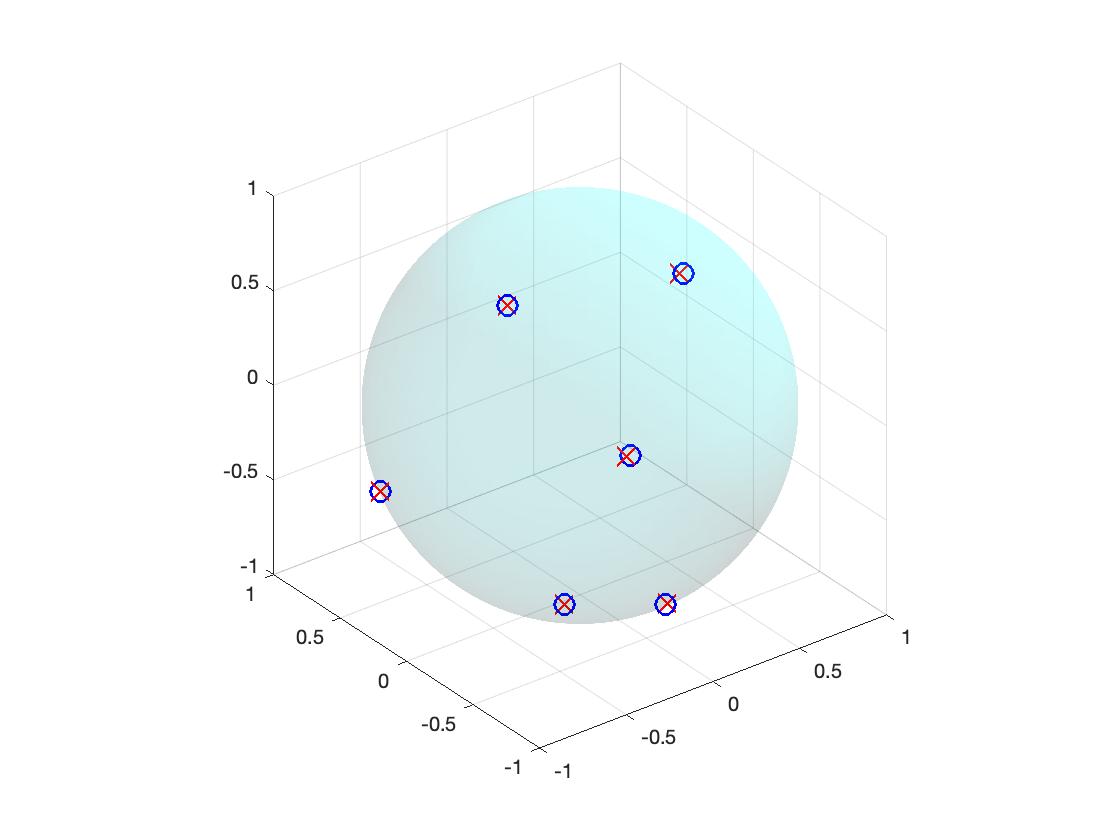}}%
   \bigskip
   
    \subcaptionbox{ For $n=80$, support clusters.  \label{fig142:a}}{\includegraphics[width=2.3in]{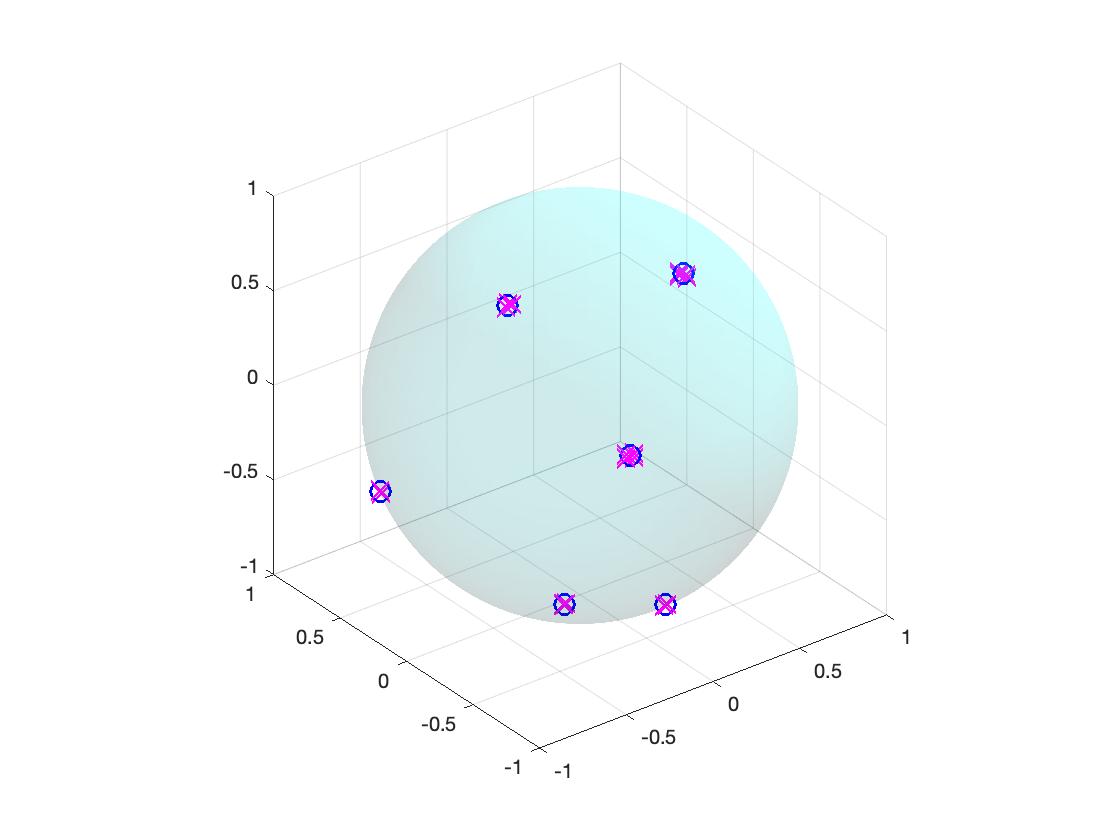}}
  \hfill%
  \subcaptionbox{ For $n=80$, recovered support.  \label{fig152:b}}{\includegraphics[width=2.3in]{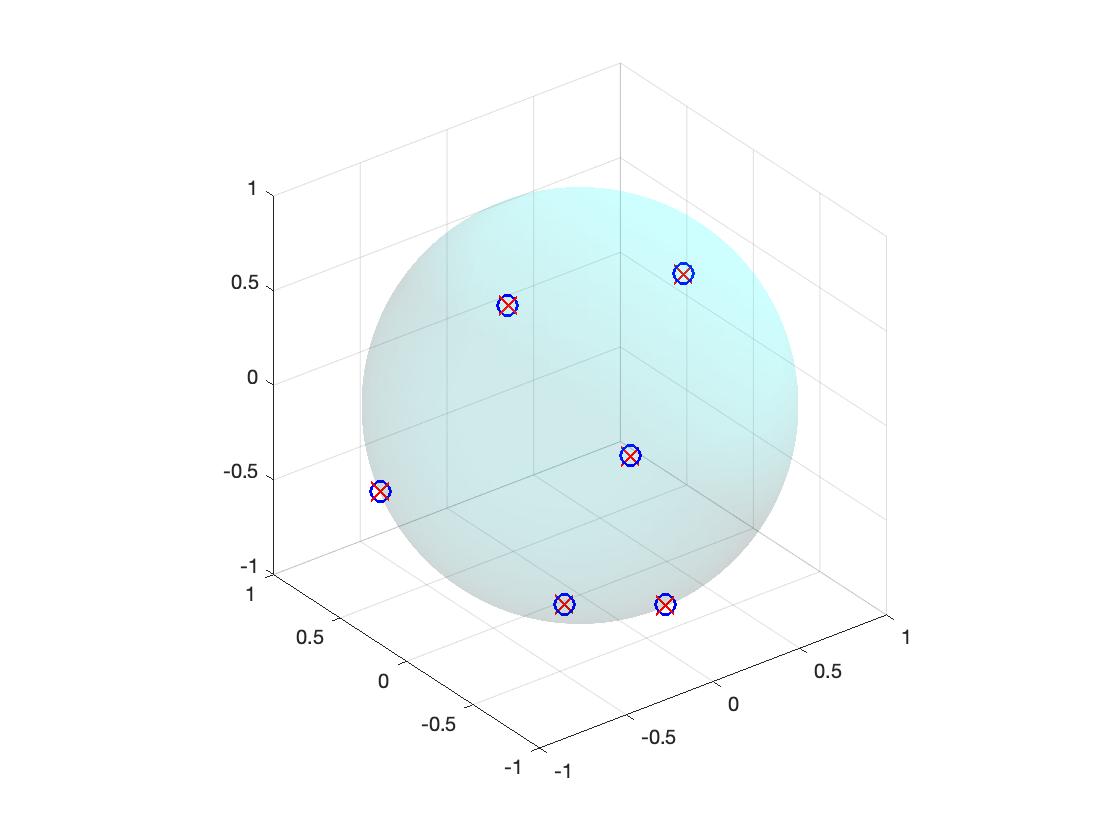}}%
  
  \caption{Performance of Algorithm 2 for different grid sizes. The true measure support  is marked with blue circles, clustered solutions are denoted by magenta crosses and the reconstructed support is indicated by red crosses.  }
  
  \label{Fig:disc:rec}
\end{figure}

\clearpage


\clearpage

\appendix
\section{Proof of Lemma~\ref{lemma_1} }\label{Lemma_1}

\begin{proof}
For the asymptotic estimates, we will use the property 
\begin{equation}
    |\sin(\w/2)|\ge \frac{|w|}{\pi}\quad \text{for}\; \w \in [-\pi,\pi], 
\end{equation}
and component-wise estimates of the expressions $|\w|\le\frac{\pi}{4(n+1)}$. The estimate for the kernel itself  follows immediately. For the first derivative, we have the following representation 
\begin{equation*}
    \widetilde{J}'_N(\w)= 2 \widetilde{F}_n(\w)\widetilde{F}'_n(\w),
\end{equation*}
where the second multiplier reads as 
\begin{equation*}
\begin{split}
    \widetilde{F}'_n(\w)&= \frac{1}{2(n+1)^2\sin^2(\w/2)}\left((n+1)\sin((n+1)\w)\right.\\
    &\left.- \frac{2\cos(\w/2)\sin^2((n+1)\w/2)}{\sin(\w/2)}\right) \\
    &= \frac{1}{2(n+1)^2\sin^2(\w/2)}\Big((n+1)\sin((n+1)\w)\\
    &-2\cos(\w/2)\sin((n+1)\w/2)U_n(\cos(\w/2)) \Big)
\end{split}
\end{equation*}
and  $U_n$ denotes the $n$-th order Chebychev polynomial of the second kind. Due to   $\|U_n\|_\infty=n$, we get 
\begin{equation*}
    |\widetilde{F}'_n(\w)|\le \frac{1.5}{(n+1)\sin^2(\w/2)}.
\end{equation*}
Since the Fej\'{e}r kernel can be written as 
\begin{equation*}
    \widetilde{F}_n(\w)=\frac{1}{(n+1)}\left(1+ 2\sum_{k=1}^n\left(1-\frac{k}{n+1}\right)\right)\cos(k \w),
\end{equation*}
we easily get $\widetilde{F}'_n(\pi)=0$ and therefore $\widetilde{J}'_N(\pi)=0$. Moreover, we have 
\begin{equation*}
    \begin{split}
    \frac{\widetilde{F}'_n(\w)}{\sin \w}&= \frac{1}{2(n+1)^2\sin^2(\w/2)} \left((n+1)\frac{\sin((n+1)\w)}{\sin(\w)}\right.
    \\
    & \left.- \frac{2\cos{(\w/2)}\sin^2((n+1)\w/2)}{\sin(\w)\sin(\w/2)}\right), \\
    &= \frac{1}{2(n+1)^2\sin^2(\w/2)}\left( (n+1)U_n(\cos(\w))-(n+1)^2 \widetilde{F}_n(\w)\right).
   \end{split}
\end{equation*}
Then the boundedness  $\|\widetilde{F}_n\|_\infty=1$ provides the estimates  
\begin{equation*}
    \left| \frac{\widetilde{F}_n(\w)}{ \sin(\w)}\right|\le \frac{1}{\sin^2(\w/2)}  \quad  \quad \text{and}   \quad \quad \lim\limits_{\w\to 0} \frac{\widetilde{F}'_n(\w)}{\sin(\w)}= \widetilde{F}''(0). 
\end{equation*}
For $|\w|\le \frac{\pi}{4(n+1)}$, it can be shown that 
\begin{equation*}
    k^2 \cos(k\w)- \frac{k\sin(k\w)\cos(\w)}{\sin(\w)}\le 0, 
\end{equation*}
which leads to the inequality 
\begin{align*}
    \left| k^2-\frac{k\sin(k\w)\cos(\w)}{\sin(\w)}\right|&= k^2-  \frac{k\sin(k\w)\cos(\w)}{\sin(d(x,z))} \\
    &\le k^2(1-\cos(k \w)) \le k^4\frac{\w^2}{2}.
\end{align*}
Since, the first derivative of the Jackson kernel can be written as 
\begin{equation*}
    \widetilde{J}'_N(\w)= \frac{1}{(n+1)^2}\left(-2\sum\limits_{k=1}^{2n}c_k k\sin(k\w)\right)
\end{equation*}
with positive Fourier coefficients $c_k$, component-wise estimation shows  that 
\begin{equation*}
    \left|\widetilde{J}''_N(0)-\frac{\cos(\w) \widetilde{J}'_N(\w)}{\sin(\w)}\right|\le \frac{\widetilde{J}^{(4)}_N(0)}{2}|\w|^2
\end{equation*}
In the same way, one derives 
\begin{equation*}
    |\widetilde{J}''_N(0)-\widetilde{J}''_N(\w)|\le \frac{\widetilde{J}^{(4)}_N(0)}{2}|\w|^2. 
\end{equation*}
For the function $G_2$, we have  the following representation  
\begin{equation*}
\begin{split}
     G_2(\w)&=\widetilde{J}''_N(\w)- \frac{\widetilde{J}'_N(\w)\cos(\w)}{\sin{(\w)}}\\
     &= 2\left( \left(\widetilde{F}'_n(\w)\right)^2
      + \widetilde{F}_n(\w)\left( \widetilde{F}''_n(\w)- \frac{\widetilde{F}'_n(\w)\cos(\w)}{\sin(\w)}\right)\right).
\end{split}
\end{equation*}
Since the second derivative of $\widetilde{F}_n$ reads as   
\begin{equation*}
    \begin{split}
       \widetilde{F}''_n(\w)&=\frac{1}{2(n+1)^2\sin^2(\w/2)}\Big((n+1)^2\Big((2+\cos(\w))\widetilde{F}_n(\w)+ \cos((n+1)\w)\Big)\\
   & - 2(n+1)(1+\cos(\w)U_n(\cos(\w))\Big),  
    \end{split}
\end{equation*}
then we get 
\begin{equation}
\begin{aligned}
    \widetilde{F}''_n(\w)- \frac{\widetilde{F}'_n(\w)\cos(\w)}{\sin(\w)}
 &= \frac{1}{2(n+1)^2\sin^2(\w/2)}\left((n+1)^2\Big(2(1+\cos(\w))\widetilde{F}_n(\w) \right. \\
   & \left.+\cos((n+1)\w)\Big) - (n+1)\Big(2+3\cos(\w)U_n(\cos(\w))\Big)\right)\\
    &= \frac{1}{2(n+1)^2\sin^2(\w/2)}\left( (n+1)^2\left(2(1+\cos(\w) \right)\cdot\right.\\
    &\left( \widetilde{F}_n(\w)- \frac{U_n(\cos(\w))}{n+1}\right)+ (n+1)^2\cos((n+1)\w) \\
    &\left.- (n+1)\cos(\w)U_n(\cos(\w))\right).
\end{aligned}
\end{equation}
Using the above derived representations yields the following estimate
\begin{equation*}
\left| G_2(\w) \right| =\left| \widetilde{J}''_N(\w) - \frac{\widetilde{J}'_N(\w)\cos(\w) }{\sin(\w)}\right|\le \frac{14.5 \cdot \pi^4}{(n+1)^2|\w|^2}
\end{equation*}
and similar result for the second derivative
\begin{equation*}
    \left| \widetilde{J}''_N(\w) \right|\le\frac{12.5 \cdot \pi^4}{(n+1)^2|\w|^2}. 
\end{equation*}

For $|\w|\le \frac{\pi}{4(n+1)}$, the estimate follows by estimating component-wise in the trigonometric representation, namely
\begin{equation*}
   \left| G_2(\w) \right| =\left| \widetilde{J}''_N(\w) - \frac{\widetilde{J}'_N(\w)\cos(\w) }{\sin(\w)}\right|\le \frac{\widetilde{J}^{(4)}(0)}{2}w^2. 
\end{equation*}
Furthermore,  at $w=0$ we have
\begin{equation*}
    \widetilde{J}''_N(0)= 2\widetilde{F}''_N(0)= -\frac{4}{(n+1)}\sum\limits_{k=1}^n\left(k^2-\frac{k^3}{(n+1)}\right)=\frac{n(n+2)}{3}.
\end{equation*}
Next, we consider the expression 
\begin{equation*}
    G_3(\w)= \frac{\widetilde{J}''_N(\w)\sin(\w)-\widetilde{J}'_N(\w)\cos(\w)}{\sin^2(\w)}.
\end{equation*}
Observe, that for $\w=\pi$ this expression vanishes.  However, for $\w\ne\pi$, we have 
\begin{equation*}
\begin{split}
    \frac{\widetilde{J}''_N(\w)\sin(\w)-\widetilde{J}'_N(\w)\cos(\w)}{\sin^2(\w)}&=2\left(\left(\frac{\widetilde{F}'_n(\w)}{\sin(\w)}\right)\widetilde{F}'_n(\w)\right.\\ 
  & +  \left. \widetilde{F}_n(\w)\left( \frac{\widetilde{F}''_n(\w)\sin(\w)-\widetilde{F}'_n(\w)\cos(\w)}{\sin^2(\w)}\right)\right), 
\end{split}
\end{equation*}
with 
\begin{equation*}
\begin{split}
\widetilde{F}_n(\w) &\left(    \frac{(\widetilde{F}''_n(\w)\sin(\w)-\widetilde{F}'_n(\w)\cos(\w))}{\sin^2(\w)}\right)=\frac{\sin^2((n+1)\w/2)}{2(n+1)^4\sin^4(\w/2)}\times\\
&\left(2(n+1)^2\frac{(1+\cos{(\w)})}{\sin(\w)}\left(\widetilde{F}_n(\w)- \frac{U_n(\cos(w))}{n+1}\right) \right.\\
&  \quad \quad \left.+(n+1)\frac{(n+1)\cos((n+1)\w)-\cos(\w)U_n(\cos(\w))}{\sin(\w)}\right). 
\end{split}
\end{equation*}
Here the first  the right-hand side summand can be represented as 
\begin{equation*}
    \begin{split}
    &\left|\frac{\sin^2((n+1)\w/2)}{(n+1)^2\sin^4(\w/2)}\cdot\frac{(1+\cos{(\w)})}{\sin(\w)}\left(\widetilde{F}_n(\w)- \frac{U_n(\cos(w))}{n+1}\right) \right|\\
    &=\frac{1}{(n+1)\sin^4(\w/2)}\left| \frac{\sin^2((n+1)\w/2)}{(n+1)}\frac{\cos(\w/2)}{\sin(\w/2)} \left(\widetilde{F}_n(\w)- \frac{U_n(\cos(w))}{n+1}\right) \right|\\ 
     &=\frac{1}{(n+1)\sin^4(\w/2)}\left| \frac{\sin^2((n+1)\w/2)}{(n+1)\sin(\w/2)} \right|\times \\
    & \left|\sin((n+1)\w/2)\cos(\w/2)\left(\widetilde{F}_n(\w)- \frac{U_n(\cos(w))}{n+1}\right) \right|\\
     &= \frac{\sqrt{\widetilde{F}_n(\w)}}{(n+1)\sin^4(\w/2)} \left|\sin((n+1)\w/2)\cos(\w/2) \left(\widetilde{F}_n(\w)- \frac{U_n(\cos(w))}{n+1}\right) \right|\\
     &\le \frac{2}{(n+1)\sin^4(\w/2)}.
\end{split}
\end{equation*}
For the second summand, due to the derivative representation of Chebychev polynomials, we have
\begin{equation*}
    \frac{(n+1)\cos((n+1)\w)- \cos(\w)U_n(\cos{(\w)})}{\sin(\w)}= - \sin(\w)U'_n(\cos(\w)).
\end{equation*}
Using the Bernstein inequality for algebraic polynomials, i.e. 
\begin{equation*}
    |P_n'(t)|\le \frac{n}{\sqrt{1-x^2}}\|P_n\|_\infty, \quad -1< x<1,
\end{equation*}
 for a polynomial of degree $n$, see e.g. \cite{Bernstein1912}, we derive 
 \begin{equation*}
   \left| \frac{(n+1)\cos((n+1)\w)- \cos(\w)U_n(\cos{(\w)})}{\sin(\w)}  \right|\le n  \|U_n\|_\infty\le (n+1)^2.
 \end{equation*}
Hence, the   second summand can be estimated as  
\begin{align*}
     \left|  \frac{\sin^2((n+1)\w/2)}{(n+1)^2\sin^4(\w/2)} \cdot  (n+1) \cdot \frac{(n+1)\cos((n+1)\w)-\cos(\w)U_n(\cos(\w))}{\sin(\w)} \right|\\
     \le \frac{0.5}{(n+1)\sin^4(\w/2)}.
\end{align*}
Together, this yields 
\begin{equation*}
      \left| \widetilde{F}_n\left(    \frac{(\w)(\widetilde{F}''_n(\w)\sin(\w)-\widetilde{F}'_n(\w)\cos(\w))}{\sin^2(\w)}\right)\right|\le \frac{2.5}{(n+1)\sin^4(\w/2)}
\end{equation*}
and consequently we get an estimate for $G_3$ of the form 
\begin{equation*}
     |G_3(\w)|=  \left| \frac{\widetilde{J}''_N(\w)\sin(\w)-\widetilde{J}'_N(\w)\cos(\w)}{\sin^2(\w)}\right|\le \frac{8}{(n+1)\sin^4(\w/2)}. 
\end{equation*}
Again, for $|\w|\le \frac{\pi}{4(n+1)}$, we estimate the expression component-wise. Observe, that 
\begin{equation*}
  \left|\frac{k^2\cos(k\w)\sin(\w)-k\sin(k\w)\cos(\w)}{/sin^2(\w)} \right|  \frac{1}{1+\cos(\w)}k^3\sin(k\w)\le \frac{k^4|w|}{1+\cos(\w)}.
\end{equation*}
For this reason, we have for $|\w|\le \frac{\pi}{4(n+1)}$
\begin{eqnarray*}
  \left| \frac{\widetilde{J}''_N(\w)\sin(\w)-\widetilde{J}'_N(\w)\cos(\w)}{\sin^2(\w)}\right|& \ds 
  \le \frac{\widetilde{J}^{(4)}_N(0)|\w|}{1+\cos(\w)}\le  \frac{\widetilde{J}^{(4)}_N(0)|\w|}{1+\cos({\pi}/{(4(n+1))})}\\\
&\ds \le  \frac{\widetilde{J}^{(4)}_N(0)|\w|}{1+\cos(\pi/8)}\le 0.52 \cdot\widetilde{J}^{(4)}_N(0)|\w|
\end{eqnarray*}
It can be also shown that the following inequality holds 
    \begin{equation*}
        \left|\frac{k^2\cos(k\w)\cos(\w)\sin(\w)-k\sin(k\w)}{\sin^2(\w)}\right|\le \frac{1}{1+\cos(\w)}k^3\sin(k\w)\le \frac{k^4|\w|}{1+\cos(\w)}, 
    \end{equation*}
and thus 
 \begin{equation*}
     \left|\frac{\widetilde{J}'_N(\w)-\cos(\w)\sin(\w) \widetilde{J}''_N(\w)}{\sin^2(\w)}\right|\le 0.52 \cdot\widetilde{J}^{(4)}_N(0)|\w|
 \end{equation*}
 
 Similarly, we can compute the third derivative of $ \widetilde{J}_N$, that is 
 \begin{equation*}
    \widetilde{J}'''_N(\w)=2\Big(3\widetilde{F}'_n(\w)\widetilde{F}''_n(\w)+ \widetilde{F}_n(\w)\widetilde{F}'''_n(\w)\Big).
 \end{equation*}
Writing explicitly  the first and second summand, we get 
 \begin{equation*}
\begin{split}
\widetilde{F}_n(\w)\widetilde{F}'''_n(\w)= \frac{1}{(n+1)\sin^4(\w/2)}\left( - \frac{1}{2(n+1)}\cos(\w/2)\sin((n+1)\w/2) \cdot \right. \\
  U_n(\cos(\w/2)) \big((5+\cos{(\w)})\widetilde{F}_n(\w)+3\cos((n+1)\w)\big)+ \frac{1}{4}\sin((n+1)\w)\cdot\\
 \left. \big( \widetilde{F}_n(\w)(6+3\cos(\w))+\cos((n+1)\w)-1\big)\right)
\end{split}
\end{equation*}
and
\begin{equation*}
\begin{split}
    \widetilde{F}'_n(\w)\widetilde{F}''_n(\w)= \frac{1}{3(n+1)\sin^4(\w/2)}\left( -\frac{3}{2(n+1)}\cos(\w/2)\sin((n+1)\w/2)\cdot\right.\\
    U_n(\cos(\w/2))\big((2+\cos(\w))\widetilde{F}_n(\w)+3\cos((n+1)\w)+2\big)+\frac{3}{4}\sin((n+1)\w)\cdot\\
    \left. \big(\widetilde{F}_n(\w)(4+3\cos(x))+\cos((n+1)\w)\big)\right). 
    \end{split}
\end{equation*}
Then using the  triangle inequality,  we can obtain the estimate for $\widetilde{J}'''_N$, that is 
\begin{equation*}
    |\widetilde{J}'''_N(\w)|\le \frac{50.5 \cdot \pi^4}{(n+1)|w|^4}.
\end{equation*}
With $\widetilde{J}'''_N(\pi)=0$, we get for $|\w|\le \frac{\pi}{4(n+1)}$,
\begin{equation*}
    |\widetilde{J}'''_N(\w)|\le \widetilde{J}^{(4)}_N(0)|\w|. 
\end{equation*}
The fourth derivative can be written as 
\begin{equation*}
    \widetilde{J}^{(4)}_N(\w)=2\Big(3(\widetilde{F}_n(\w))^2+ 4\widetilde{F}'_n(\w)\widetilde{F}'''_n(\w)+ \widetilde{F}_n(\w)\widetilde{F}^{(4)}_n(\w)\Big)
\end{equation*}
and therefore  we obtain
\begin{equation*}
    \widetilde{J}^{(4)}_N(0)=2\Big(3(\widetilde{F}_n(0))^2+  \widetilde{F}^{(4)}_n(0)\Big). 
\end{equation*}
Now, since for $\w=0$ one has 
\begin{equation*}
    \widetilde{F}^{(4)}_n(0)=\frac{2}{n+1}\sum\limits_{k=1}^n \left(k^4-\frac{k^5}{n+1}\right)= \frac{(n(n+2))}{30}(2n(n+2)-1),
\end{equation*}
we easily  get  the value 
\begin{equation*}
    \widetilde{F}^{(4)}_n(0)= \frac{1}{30}n(n+2)(9n(n+2)-2).
\end{equation*}
\end{proof}

\section{Proof of Lemma~\ref{summation_lemma} }
\begin{proof} For $x\in \St$, with $d(x,x_j)\le\varepsilon \frac{\nu}{n+1} $ for some $x_j \in \X$,  we define the ring around $x \in \St$ by 
\begin{equation*}
    S_m:=\left\{y\in \St\colon   \frac{\nu m}{n+1}\le d(x,y) \frac{\nu (m+1)}{n+1}\right\}
\end{equation*}
for $m\in \N $. 
As it have been shows in \cite{KeinerKunis2007}, we can estimate the number of elements in the intersection of $S_m$ with the set $\X \setminus \{x_j\}$ for $m\ge 1$ by 
\begin{equation*}
    \mathrm{card} (\X \setminus \{x_j\} {\cap } S_m )\le 25 m. 
\end{equation*}

Thus, it remains to estimate the number of elements in $\X \setminus \{x_j\} {\displaystyle \cap } S_0$. We are going to use the same argument as in \cite{KeinerKunis2007}, namely, we see that for $x_i,x_n \in \X \in \{x_j\} {\displaystyle \cap } S_0$  one has ${B_{\frac{\nu}{2(n+1)}}(x_i) {\displaystyle \cap }B_{\frac{3\nu}{2(n+1)}}(x_n) =\emptyset}$  and 
\begin{equation*}
  \displaystyle   \cup_{x_i \in \{x_j\}{ \cap } S_0 }   B_{\frac{\nu}{2(n+1)}}(x_i) \subseteq  B_{\frac{3\nu}{2(n+1)}}(x_j).
\end{equation*}
Since $\varepsilon\le 1/2$ and the Riemannian volume form is rotation invariant, we can bound the number of elements by 
\begin{equation*}
    \mathrm{card} (\X \setminus \{x_j\} {\displaystyle \cap } S_0 )\le 2\frac{\Omega(  B_{\frac{3\nu}{2(n+1)}}(e_3) )}{\Omega(  B_{\frac{\nu}{2(n+1)}}(e_3) )}
\end{equation*}
where $e_3=(0,01)^{\mathrm{T}}$ is the north pole on the sphere. In polar coordinates around  $e_3$, we consequently have the bound 
\begin{equation*}
    \begin{split}
        \mathrm{card} (\X \setminus \{x_j\} {\displaystyle \cap } S_0 )&\le \frac{\Omega(  B_{\frac{3\nu}{2(n+1)}}(e_3) )}{\Omega(  B_{\frac{\nu}{2(n+1)}}(e_3) )}= \frac{  \ds \int_{0}^{\frac{3\nu}{2(n+1)}}\sin(r) \dx r}{ \ds \int_{0}^{\frac{\nu}{2(n+1)}}\sin(r) \dx r}\\
        &= \frac{1-\cos{\left(\frac{3\nu}{2(n+1)}\right)}}{1-\cos{\left(\frac{\nu}{2(n+1)}\right)}}= \left(1+ 2\cos{\left( \frac{\nu}{2(n+1)}\right)}\right)^2\le 9. 
    \end{split}
\end{equation*}

Due to $d(x,x_j)\le \varepsilon \frac{\nu}{n+1}$, we have $d(x,x_i)\ge  \frac{(1- \varepsilon)\nu}{n+1}$ for $ x_i \in \X \setminus\{x_j\} {\displaystyle \cap } S_0 $. Using this and the locality assumption \eqref{locality_of_f}, we can therefore estimate 
\begin{equation*}
    \begin{split}
        \sum\limits_{x_i\in \X \setminus \{x_j\}} |f(x,x_i)|&\le 
         \sum\limits_{x_i \in \X \setminus \{x_j\} {\displaystyle \cap } S_0} \frac{c_f}{((n+1) d(x,x_i))^s}\\
         &+ \sum_{m=1}^{\infty}\sum\limits_{x_i \in  \X \setminus \{x_j\} {\displaystyle \cap } S_m} \frac{c_f}{((n+1) d(x,x_i))^s}\\
         &\le \frac{9 c_f (1-\varepsilon)^{-s}}{\nu^s}+ 25 c_f \sum_{m=1}^{\infty}\frac{m}{(m\nu)^s}\\
         & \le \frac{9 c_f (1-\varepsilon)^{-s}}{\nu^s}+ \frac{25 c_f }{\nu^s}\sum_{m=1}^{\infty}\frac{1}{(m)^{s-1}}\\
         &\le \frac{(9(1-\varepsilon)^{-s}+25)c_f\zeta(s-1)}{\nu^s},
    \end{split}
\end{equation*}
where the last inequality follows by the definition of the Zeta function. On the other hand, we can define the ring around $x_j$ again by 
\begin{equation*}
    \widetilde{S}_m=\left\{y\in \St: \frac{(1-\varepsilon)\nu m}{n+1}\le d(x_j,y)\le \frac{(1-\varepsilon)\nu (m+1)}{n+1}\right\}. 
\end{equation*}
Since $d(x,x_j)\le \varepsilon \frac{\nu}{n+1}$, we have $d(x,x_j)\le \varepsilon d(x_i,x_j)$ for $x_i\in \X \setminus \{x_j\} {\displaystyle \cap } \widetilde{S}_m$ and therefore $d(x,x_i)\ge d(x_i,x_j)-d(x,x_j)\le \frac{(1-\varepsilon)\nu m}{n+1}$. Using this and the locality assumption \eqref{locality_of_f}, for $s\ge 3$ we can estimate the sum as 
\begin{equation*}
    \begin{split}
        \sum\limits_{x_i\in \X \setminus \{x_j\}} |f(x,x_i)|&\le 
          \sum_{m=1}^{\infty}\sum\limits_{x_i \in  \X \setminus \{x_j\} {\displaystyle \cap } \widetilde{S}_m} \frac{c_f}{((n+1) d(x,x_i))^s}\\
           &\le 25 c_f   \sum_{m=1}^{\infty}\frac{m}{(1-\varepsilon)^s(m\nu)^s}\\
         &\le \frac{25 c_f}{ (1-\varepsilon)^{s}\nu^s} \sum_{m=1}^{\infty}\frac{1}{m^{s-1}}= \frac{25 c_f\zeta(s-1)}{ (1-\varepsilon)^s\nu^s}. 
    \end{split}
\end{equation*}
Afterwards, we  choose $a_{\varepsilon}= \zeta(s-1)\cdot \min\{9\cdot (1-\varepsilon)^{-s}+25, 25\cdot (1-\varepsilon)^{-s} \}$ that finishes the proof. 
\end{proof}

\section{ Properties of the Cross Product }\label{cross}
Let us consider two vectors $a,b \in \R^3$, then the cross product is given by the vector 
\begin{equation*}
    a\times b= \begin{pmatrix} a_2b_3-a_3b_2\\
    a_3b_1-a_1b_3\\
    a_1b_2-a_2b_1\end{pmatrix}. 
\end{equation*}
Then for $a,b,c,d\in \R^3$ we have the following relations between  the dot product "$\cdot$" and the cross product "$\times$":
\begin{itemize}
\item[(i)]  $a\times b= - (b\times a)$,
 \item[(ii)] $a\cdot (b\times c)= b\cdot (c\times a)= c\cdot (a\times b) $,
 \item[(iii)] $a\times (b\times c)= b (a\cdot c)-  c (a\cdot b) $,
  \item[(iv)] $(a\times b) \times(a\times c)= (a\cdot (b\times c))a$,
   \item[(v)] $(a\times b) \cdot (c\times d)= (a\cdot c) (b\times d)- (a\cdot d) (b\times c)$. 
\end{itemize}


\end{document}